\documentclass[11pt,reqno]{amsart}
\usepackage{times}
\usepackage{amsmath,amsfonts,amstext,amssymb,amsbsy,amsopn,amsthm,eucal}
\usepackage{txfonts}
\usepackage{dsfont}
\usepackage{graphicx}   
\usepackage{hyperref}
\usepackage{color}
\usepackage{verbatim}

\numberwithin{equation}{section}

\hypersetup{
  pdftitle={Rectifiable-Reifenberg and the Regularity of Stationary and Minimizing Harmonic Maps},
  pdfauthor={Aaron Naber and Daniele Valtorta},
  pdfsubject={},
  pdfkeywords={},
  pdfpagelayout=SinglePage,
  pdfpagemode=UseOutlines,
  colorlinks,
  linkcolor=[rgb]{0,0,0.7},
  urlcolor=[rgb]{0,0,0.4},
  citecolor=[rgb]{0.4,0.1,0}
}

\newcommand{\Vol}{\text{Vol}}
\newcommand{\diam}{\text{diam}}

\newcommand{\inj}{\text{inj}}
\newcommand{\Sing}{\text{S}}

\newcommand{\dN}{\mathds{N}}

\newcommand{\dQ}{\mathds{Q}}
\newcommand{\dR}{\mathds{R}}

\newcommand{\norm}[1]{\left\|#1\right\|}
\newcommand{\ps}[2]{\left\langle#1,#2\right\rangle}
\newcommand{\ton}[1]{\left(#1\right)}
\newcommand{\qua}[1]{\left[#1\right]}
\newcommand{\cur}[1]{\left\{#1\right\}}
\newcommand{\abs}[1]{\left|#1\right|}
\newcommand{\wto}{\rightharpoonup ^*}

\newcommand{\tSr}{\tilde S _{\bar r}}

\newcommand{\B}[2]{B_{#1}\ton{#2}}
\newcommand{\tB}[2]{\tilde B_{#1}\ton{#2}}
\newcommand{\supp}[1]{\operatorname{supp}\ton{#1}}

\newcommand{\N}{\mathds{N}}

\newcommand{\R}{\mathds{R}}

\renewcommand{\paragraph}[1]{\ \newline \ \textbf{#1\ }}

\newcommand{\hol}{H\"older }
\newcommand{\al}{Ahlfors }

\newtheorem{theorem}{Theorem}[section]

\newtheorem{proposition}[theorem]{Proposition}
\newtheorem{lemma}[theorem]{Lemma}

\theoremstyle{definition}
\newtheorem{definition}[theorem]{Definition}
\theoremstyle{remark}
\newtheorem{remark}{Remark}[section]
\theoremstyle{remark}
\newtheorem{example}{Example}[section]
\theoremstyle{remark}

\theoremstyle{remark}

\theoremstyle{remark}

\begin{document}

\title{Rectifiable-Reifenberg and the Regularity of Stationary and Minimizing Harmonic Maps}


\author{Aaron Naber and Daniele Valtorta}\thanks{The first author has been supported by NSF grant DMS-1406259, the second author has been supported by SNSF grant 149539}

\date{\today}

\begin{abstract}
In this paper we study the regularity of stationary and minimizing harmonic maps $f:B_2(p)\subseteq M\to N$ between Riemannian manifolds.  If $S^k(f)\equiv\{x\in M: \text{ no tangent map at $x$ is }k+1\text{-symmetric}\}$ is the $k^{th}$-stratum of the singular set of $f$, then it is well known that $\dim S^k\leq k$, however little else about the structure of $S^k(f)$ is understood in any generality.  Our first result is for a general stationary harmonic map, where we prove that $S^k(f)$ is $k$-rectifiable.  In fact, we prove for $k$-a.e. point $x\in S^k(f)$ that there exists a unique $k$-plane $V^k\subseteq T_xM$ such that {\it every} tangent map at $x$ is $k$-symmetric with respect to $V$.  

In the case of minimizing harmonic maps we go further, and prove that the singular set $S(f)$, which is well known to satisfy $\dim S(f)\leq n-3$, is in fact $n-3$-rectifiable with uniformly {\it finite} $n-3$-measure.  An effective version of this allows us to prove that $|\nabla f|$ has estimates in $L^3_{weak}$, an estimate which is sharp as $|\nabla f|$ may not live in $L^3$.  More generally, we show that the regularity scale $r_f$ also has $L^3_{weak}$ estimates.

The above results are in fact just applications of a new class of estimates we prove on the {\it quantitative} stratifications $S^k_{\epsilon,r}(f)$ and $S^k_{\epsilon}(f)\equiv S^k_{\epsilon,0}(f)$.  Roughly, $S^k_{\epsilon}\subseteq M$ is the collection of points $x\in M$ for which no ball $B_r(x)$ is $\epsilon$-close to being $k+1$-symmetric.  We show that $S^k_\epsilon$ is $k$-rectifiable and satisfies the Minkowski estimate $\Vol(B_r\,S_\epsilon^k)\leq C r^{n-k}$.  

The proofs require a new $L^2$-subspace approximation theorem for stationary harmonic maps, as well as new $W^{1,p}$-Reifenberg and rectifiable-Reifenberg type theorems.  These results are generalizations of the classical Reifenberg, and give checkable criteria to determine when a set is $k$-rectifiable with uniform measure estimates.  The new Reifenberg type theorems may be of some independent interest.  The $L^2$-subspace approximation theorem we prove is then used to help break down the quantitative stratifications into pieces which satisfy these criteria.
\end{abstract}

\maketitle

\tableofcontents

\section{Introduction}

Let $f:B_2(p)\subseteq M\to N$ be a stationary harmonic map between two Riemannian manifolds with $N$ compact without boundary.  We recall that the singular set of $f$ is defined as the complement of the regular set, i.e., 
\begin{gather}
 \Sing(f)=\cur{x\in M \ \ s.t. \ \ \exists r>0 \ \ s.t. \ \ u|_{\B r x} \ \ \text{ is continuous }}^C=\cur{x\in M \ \ s.t. \ \ \exists r>0 \ \ s.t. \ \ u|_{\B r x} \ \ \text{ is smooth }}^C\, .
\end{gather}
One can use $f$ to give a stratification of $M$ by the singular sets
\begin{align}\label{e:stratification}
S^0(f)\subseteq\cdots\subseteq S^k(f)\subseteq \cdots \Sing(f)\subseteq M\, , 
\end{align}
defined by
\begin{align}\label{e:sing_set}
S^k(f)\equiv\{x\in M: \text{ no tangent map at $x$ is }k+1\text{-symmetric}\}\, ,
\end{align}
see Definition \ref{d:stratification} for a precise definition and more detailed discussion.  A key result of \cite{ScUh_RegHarm} is that we have the Hausdorff dimension estimate
\begin{align}\label{e:sing_set_haus_est}
\dim S^k\leq k\, ,
\end{align}
However, little else is known about the structure of the singular sets $S^k(f)$.  In the stationary case, nothing is known in general.  In the case of a minimizing harmonic map, \cite{ScUh_RegHarm} proves that $S(f)=S^{n-3}(f)$, and, under the additional assumption that $N$ is analytic, it follows from the wonderful work of \cite{Simon_RegMin} that the top stratum $S^{n-3}$ is rectifiable.  In low dimensions and when the targets are spheres, these results have been further refined in \cite{linwang} and \cite{HL42}.\\


The goal of this paper is to study additional properties of the stratification and the associated quantitative stratification for stationary and minimizing harmonic maps.  That is, a first goal is to see for a stationary harmonic map with a general compact target space that $S^k(f)$ is $k$-rectifiable {\it for all} $k$.  In fact, the proof allows us to see the following stronger result.  For $k$-a.e. $x\in S^k$ there will exist a unique $k$-plane $V^k\subseteq T_xM$ such that {\it every} tangent map at $x$ will be $k$-symmetric with respect to $V$.  Let us observe a subtlety to this statement.  We are not claiming that the tangent maps at such $x$ are unique.  We are simply claiming that for $k$-a.e. point there is a maximal dimensional $k$-plane of symmetry which exists, and that this plane of symmetry is unique independent of which tangent map is chosen.  Theorem \ref{t:main_stationary} gives the precise results for the classical stratification of a stationary harmonic map.\\

For minimizing harmonic maps we can improve on this result in two ways.  First, we can show that the top stratum is not only $n-3$ rectifiable, but has an {\it a priori} bound on its $n-3$ measure.  That is, if 
\begin{align}
\fint_{B_2(p)}|\nabla f|^2\leq \Lambda\, ,
\end{align}
is the $L^2$ bound on the energy, then we have the $n-3$-Hausdorff measure estimate
\begin{align}
\lambda^{n-3}(S^{n-3}(f)\cap B_1)\leq C\, ,
\end{align}
where $C$ depends on $M$, $N$ and $\Lambda$ (see Sections \ref{ss:stationary_results} and following for precise statements).

Indeed, we show a much stronger Minkowski version of this estimate, see Theorem \ref{t:main_min_finite_measure} for the complete statement.  In fact, we can prove what turns out to be a much more effective analytic version of the above.  Namely, in Theorem \ref{t:main_min_weak_L3} we show that $|\nabla f|$, and in fact the regularity scale $r_f$, have {\it a priori} weak $L^3$ bounds.  That is,
\begin{align}
\Vol\ton{\cur{|\nabla f|>r^{-1}}\cap B_1}\leq \Vol\ton{B_r\cur{|\nabla f|>r^{-1}}\cap B_1} \leq C r^3\, ,
\end{align}
where $\B r E$ stands for the tubular neighborhood of radius $r$ around the set $E$, i.e.,
\begin{gather}
 \B r E = \cur{x \ \ s.t. \ \ d(x,E)<r} = \bigcup_{x\in E} \B r x\, .
\end{gather}

Notice that this is indeed sharp, in that there are counterexamples showing that $|\nabla f|\not\in L^3$, see Example \ref{sss:Lp_sharp_example}.  Let us also point out that this sharpens results from \cite{ChNa2}, where it was proven for minimizing harmonic maps that $|\nabla f|\in L^p$ for all $p<3$.  We refer the reader to Section \ref{ss:minimizing_results} for the precise and most general statements.\\

Now the techniques of this paper all center around estimating the {\it quantitative} stratification, not the standard stratification itself.  In fact, it is for the quantitative stratification that the most important results of the paper hold, everything else can be seen to be corollaries of these statements.  

A first version of the quantitative stratification can be found in \cite[section 2.25]{almgren_big}. This concept was later developed in \cite{ChNa1} with the goal of proving new estimates on noncollapsed manifolds with Ricci curvature bounded below and in particular Einstein manifolds, and then extended in \cite{ChNa2} to give effective and $L^p$ estimates on minimal and stationary harmonic maps and minimal submanifolds. It has since been used in \cite{ChNaHa1}, \cite{ChNaHa2}, \cite{ChNaVa}, \cite{FoMaSpa}, \cite{brelamm} to prove similar results in the areas of mean curvature flow, critical sets of elliptic equations, harmonic map flow, and biharmonic maps.  

Before describing the results in this paper on the quantitative stratification, let us give more precise definitions of everything.  To begin with, to describe the stratification and quantitative stratification we need to discuss the notion of symmetry associated to a solution.  Specifically:

\begin{definition}
We define the following:
\begin{enumerate}
\item A mapping $f:\dR^n\to N$ is called $k$-symmetric if $f(\lambda x)=f(x)$ $\forall \lambda>0$ and if there exists a $k$-plane $V^k\subseteq \dR^n$ such that for each $x\in \dR^n$ and $y\in V^k$ we have that $f(x+y)=f(x)$. Note that a $0$-symmetric function is simply a homogeneous function of degree $0$.
\item Given a mapping $f:M\to N$ and $\epsilon>0$, we say a ball $B_r(x)\subseteq M$ is $(k,\epsilon)$-symmetric if $r<$inj$(x)$ and there exists a $k$-symmetric mapping $\tilde f:T_xM\to N$ such that $\fint_{B_r(x)}|f-\tilde f|^2 dv_g<\epsilon$, where we have used the exponential map to identify $\tilde f$ as a function on $M$.
\end{enumerate}
\end{definition}

Thus, a function is $k$-symmetric if it only depends on $n-k$ variables and is radially invariant.  A $k$-symmetric function can therefore be naturally identified with a function on a $n-k-1$ sphere.  A function is $(k,\epsilon)$-symmetric on a ball $B_r(x)$ if it is $L^2$ close to a $k$-symmetric function on this ball.\\

With the notion of symmetry in hand we can define precisely the quantitative stratification associated to a solution.  The idea is to group points together based on the amount of symmetry that balls centered at those points contain.  In fact, there are several variants which will play a role for us, all of which are important for the applications to the standard singular set.  Let us introduce them and briefly discuss them:

\begin{definition}\label{d:stratification}
For a stationary harmonic map $f:B_2(p)\to N$ we make the following definitions:
\begin{enumerate}
\item For $\epsilon,r>0$ we define the $k^{th}$ $(\epsilon,r)$-stratification $S^k_{\epsilon,r}(f)$ by
\begin{align}
S^k_{\epsilon,r}(f)\equiv\{x\in B_1(p):\text{ for no }r\leq s<1\text{ is $B_{s}(x)$ a }(k+1,\epsilon)\text{-symmetric ball}\}.
\end{align}
\item For $\epsilon>0$ we define the $k^{th}$ $\epsilon$-stratification $S^k_{\epsilon}(f)$ by
\begin{align}
S^k_{\epsilon}(f)=\bigcap_{r>0} S^k_{\epsilon,r}(f)\equiv\{x\in B_1(p):\text{ for no }0< r<1\text{ is $B_{r}(x)$ a }(k+1,\epsilon)\text{-symmetric ball}\}.
\end{align}
\item We define the $k^{th}$-stratification $S^k(f)$ by
\begin{align}
S^k(f)=\bigcup_{\epsilon>0} S^k_{\epsilon}=\{x\in B_1(p):\text{ no tangent cone at $x$ is $k+1$-symmetric}\}.
\end{align}
\end{enumerate}
\end{definition}
\begin{remark}
It is a small but important exercise to check that the standard stratification $S^k(f)$ as defined in \eqref{e:stratification} agrees with the set $\bigcup_{\epsilon>0} S^k_{\epsilon}$.  We do this carefully in Section \ref{ss:proof_thm_main_stationary}.
\end{remark}

Let us discuss in words the meaning of the quantitative stratification, and how it relates to the standard stratification.  As discussed at the beginning of the section, the stratification $S^k(f)$ of $M$ is built by separating points of $M$ based on the infinitesimal symmetries of $f$ at those points.  The quantitative stratifications $S^k_{\epsilon}(f)$ and $S^k_{\epsilon,r}(f)$ are, on the other hand, instead built by separating points of $M$ based on how many symmetries exist on balls of definite size around the points.  In practice, the quantitative stratification has two advantages to the standard stratification.  First, for applications to minimizing harmonic maps the quantitative stratification allows one to prove effective estimates.  In particular, in \cite{ChNa2} the weaker $L^p$ estimates 
\begin{align}\label{e:sing_set_Lp_est}
\fint_{B_1(p)}|\nabla f|^{3-\delta},\, \fint_{B_1(p)}|\nabla^2 f|^{\frac{3-\delta}{2}}<C_\delta \text{  }\forall\, \delta>0\, ,
\end{align}
on solutions were obtained by exploiting this fact.  The second advantage is that the estimates on the quantitative stratification are much stronger than those on the standard stratification.  Namely, in \cite{ChNa2} the Hausdorff dimension estimate \eqref{e:sing_set_haus_est} on $S^k(f)$ was improved to the Minkowski content estimate 
\begin{align}\label{e:sing_set_mink_delta_est}
\Vol(B_r\, S^k_{\epsilon,r})\leq C_\delta r^{n-k-\delta}\text{  }\forall \delta>0\, .
\end{align}
One of the key technical estimates of this paper is that in Theorem \ref{t:main_quant_strat_stationary} we drop the $\delta$ from the above estimate and obtain an estimate of the form
\begin{align}\label{e:quant_strat_mink_r}
\Vol(B_r\,S^k_{\epsilon,r})\leq C r^{n-k}\, .
\end{align}
From this we are able conclude in Theorem \ref{t:main_eps_stationary} an estimate on $S^k_\epsilon$ of the form
\begin{align}\label{e:quant_strat_mink}
\Vol(B_r\,S^k_{\epsilon})\leq C r^{n-k}\, .
\end{align}
In particular, this estimate allows us to conclude that $S^k_{\epsilon}$ has uniformly finite $k$-dimensional measure.  In fact, the techniques will prove much more for us.  They will show us that $S^k_\epsilon$ is $k$-rectifiable, and that for $k$-a.e. point $x\in S^k_\epsilon$ there is a unique $k$-plane $V^k\subseteq T_xM$ such that {\it every} tangent map at $x$ is $k$-symmetric with respect to $V$.  By observing that $S^k(f)=\bigcup S^k_\epsilon(f)$, this allows us to prove in Theorem \ref{t:main_stationary} our main results on the classical stratification.  This decomposition of $S^k$ into the pieces $S^k_\epsilon$ is crucial for the proof.

On the other hand, \eqref{e:quant_strat_mink_r}, combined with the $\epsilon$-regularity theorems of \cite{ScUh_RegHarm},\cite{ChNa2}, allow us to conclude in the minimizing case both the weak $L^3$ estimate on $|\nabla f|$, and the $n-3$-finiteness of the singular set of $f$.  Thus we will see that Theorems \ref{t:main_min_finite_measure} and \ref{t:main_min_weak_L3} are fairly quick consequences of \eqref{e:quant_strat_mink}.\\

Thus we have seen that \eqref{e:quant_strat_mink_r} and \eqref{e:quant_strat_mink}, and more generally Theorem \ref{t:main_quant_strat_stationary} and Theorem \ref{t:main_eps_stationary}, are the main challenges of the paper.  We will give a more complete outline of the proof in Section \ref{ss:outline_proof}, however let us mention for the moment that two of the new ingredients to the proof are a new $L^2$-subspace approximation theorem for stationary harmonic maps, proved in Section \ref{s:best_approx}, and new $W^{1,p}$-Reifenberg and rectifiable-Reifenberg theorems, proved in Section \ref{s:bi-Lipschitz_reifenberg}.  The $L^2$-approximation result roughly tells us that the $L^2$-distance of a measure from being contained in a $k$-dimensional subspace may be estimated by integrating the energy drop of a stationary harmonic map over the measure.  To exploit the estimate we prove a new $W^{1,p}$-Reifenberg type theorem.  The classical Reifenberg theorem states that if we have a set $S$ which is $L^\infty$-approximated by a subspace at every point and scale, then $S$ is bi-\hol to a manifold, see theorem \ref{t:classic_reifenberg} for a precise statement.  It is important for us to improve on this bi-\hol estimate, at least enough that we are able to control the $k^{th}$-dimensional measure of the set and prove rectifiability.  In particular, we want to improve the $C^\alpha$-maps to $W^{1,p}$-maps for $p>k$, and we will want to do it using a condition which is integral in nature.  More precisely, we will only require a form of summable $L^2$-closeness of the subset $S$ to the approximating subspaces.  We will see in Theorem \ref{t:best_approximation} that by using the $L^2$-subspace approximation theorem that the conditions of this new rectifiable-Reifenberg are in fact controllable for the quantitative stratifications $S^k_\epsilon$, as least after we break it up into appropriate pieces.

\subsection{Results for Stationary Harmonic Maps}\label{ss:stationary_results}

We now turn our attention to giving precise statements of the main results of this paper.  In this subsection we focus on those concerning the singular structure of stationary harmonic maps.  That is, we will be considering stationary harmonic maps $f:B_2(p)\subseteq M\to N$.  In order to write the constants involved in some explicit form, let us choose $K_M,K_N,n>0$ to be the smallest number such that
\begin{align}\label{e:manifold_bounds}
&|\sec_{B_2(p)}|\leq K_M,\,\, \inj(B_2(p))\geq K_M^{-1}\, ,\notag\\
&|\sec_N|\leq K_N,\,\, \inj(N)\geq K_N^{-1},\,\,\diam(N)\leq K_N\, ,\notag\\
&\dim(M),\, \dim(N)\leq n\, .\\ \notag
\end{align}

Now let us begin by discussing our main theorem for the quantitative stratifications $S^k_{\epsilon,r}(f)$:\\

\begin{theorem}[$(\epsilon,r)$-Stratification of Stationary Harmonic Maps]\label{t:main_quant_strat_stationary}
Let $f:B_2(p)\subseteq M\to N$ be a stationary harmonic mapping satisfying \eqref{e:manifold_bounds} with  $\fint_{B_2(p)}|\nabla f|^2 \leq \Lambda$.  Then for each $\epsilon>0$ there exists $C_\epsilon(n,K_M,K_N,\Lambda,\epsilon)$ such that
\begin{align}\label{e:main_eps_r_stationary:mink}
\Vol\Big(B_r\Big(S^k_{\epsilon,r}(f)\Big)\Big)\leq C_\epsilon r^{n-k}\, .
\end{align}
\end{theorem}
\vspace{.5 cm}

When we study the stratum $S^k_{\epsilon}(f)$, we can refine the above to prove structure theorems on the set itself.  For the definition of $k$-rectifiability, we refer the reader to the standard reference \cite{Fed}.\\

\begin{theorem}[$\epsilon$-Stratification of Stationary Harmonic Maps]\label{t:main_eps_stationary}
Let $f:B_2(p)\subseteq M\to N$ be a stationary harmonic mapping satisfying \eqref{e:manifold_bounds} with $\fint_{B_2(p)}|\nabla f|^2 \leq \Lambda$.  Then for each $\epsilon>0$ there exists $C_\epsilon(n,K_M,K_N,\Lambda,\epsilon)$ such that
\begin{align}\label{e:main_eps_stationary:mink}
\Vol\Big(B_r\big(S^k_{\epsilon}(f)\big)\Big)\leq C_\epsilon r^{n-k}\, .
\end{align}
In particular, we have the $k$-dimensional Hausdorff measure estimate $\lambda^{k}(S^k_\epsilon(f))\leq C_\epsilon$.  Further, $S^k_\epsilon(f)$ is $k$-rectifiable, and for $k$-a.e. $x\in S^k_\epsilon$ there exists a {\it unique} $k$-plane $V^k\subseteq T_xM$ such that {\it every} tangent cone of $x$ is $k$-symmetric with respect to $V^k$.
\end{theorem}
\begin{remark}
In fact, the techniques will prove an estimate even stronger than the Minkowski estimate of \eqref{e:main_eps_r_stationary:mink} and \eqref{e:main_eps_stationary:mink}.  That is, one can prove a uniform $k$-dimensional packing content estimate.  More precisely, let $\{B_{r_i}(x_i)\}$ be any collection of disjoint balls such that $x_i\in S^k_{\epsilon}$,  then we have the content estimate $\sum r_i^k \leq C_\epsilon$.
\end{remark}
\vspace{.5 cm}

Finally, we end this subsection by stating our main results when it comes to the classical stratification $S^k(f)$.  The following may be proved from the previous Theorem in only a few lines given the formula $S^k(f)=\bigcup S^k_\epsilon(f)$:\\

\begin{theorem}[Stratification of Stationary Harmonic Maps]\label{t:main_stationary}
Let $f:B_2(p)\subseteq M\to N$ be a stationary harmonic mapping satisfying \eqref{e:manifold_bounds} with $\fint_{B_2(p)}|\nabla f|^2 \leq \Lambda$.  Then for each $k$ we have that $S^k(f)$ is $k$-rectifiable.  Further, for $k$-a.e. $x\in S^k(f)$ there exists a {\it unique} $k$-plane $V^k\subseteq T_xM$ such that {\it every} tangent cone of $x$ is $k$-symmetric with respect to $V^k$.
\end{theorem}
\vspace{1cm}

\subsection{Results for Minimizing Harmonic Maps}\label{ss:minimizing_results}

In this section we record our main results for minimizing harmonic maps.  Most of the results of this section follow quickly by combining the quantitative stratification results of Section \ref{ss:stationary_results} with the $\epsilon$-regularity of \cite{ScUh_RegHarm,ChNa2}, see Section \ref{ss:eps_reg} for a review of these points.  \\

Our first estimate is on the singular set $\Sing(f)$ of a minimizing harmonic map.  Recall that $\Sing(f)$ is the set of points where $f$ is not smooth.  Our first estimate on the singular structure of a minimizing harmonic map is the following:\\

\begin{theorem}[Structure of Singular Set]\label{t:main_min_finite_measure}
Let $f:B_2(p)\subseteq M\to N$ be a minimizing harmonic mapping satisfying \eqref{e:manifold_bounds} with $\fint_{B_2(p)}|\nabla f|^2 \leq \Lambda$.  Then $\Sing(f)$ is $n-3$-rectifiable and there exists $C(n,K_M,K_N,\Lambda)$ such that
\begin{align}
\Vol\Big(B_r\big(\Sing(f)\big)\cap B_1(p)\Big)\leq Cr^3\, .
\end{align}
In particular, $\lambda^{n-3}(\Sing(f))\leq C$.
\end{theorem}
\begin{remark}
As in Theorem \ref{t:main_eps_stationary}, the same techniques prove a packing content estimate on $\Sing(f)$.  That is, if $\{B_{r_i}(x_i)\}$ is any collection of disjoint balls such that $x_i\in \Sing(f)\cap B_1$ , then we have the content estimate $\sum r_i^k \leq C$.
\end{remark}

\vspace{.5 cm}
The above can be extended to effective Schauder estimates on $f$.  To state the results in full strength let us recall the notion of the regularity scale associated to a function.  Namely:

\begin{definition}\label{d:regularity_scale}
Let $f:B_2(p)\to N$ be a mapping between Riemannian manifolds.  For $x\in B_1(p)$ we define the regularity scale $r_f(x)$ by
\begin{align}
r_f(x)\equiv \max\{0\leq r\leq 1: \sup_{B_r(x)}|\nabla f|\leq r^{-1}\}\, .
\end{align}
By definition, $r_f(x)\equiv 0$ if $f$ is not Lipschitz in a neighborhood of $x$.
\end{definition}
\begin{remark}
The regularity scale has nice scaling property. Indeed, if we define $T_{X,\rho}:T_x M\to N$ by
\begin{gather}
 T_{x,\rho}(y)= f(\exp_x(\rho y))\, ,
\end{gather}
then $r_f(x)= \rho r_{T_{x,\rho}}(0)$. In other words, if $r\equiv r_f(x)$ and we rescale $B_r(x)\to B_1(0)$, then on the rescaled ball we will have that $|\nabla T_{x,r}|\leq 1$ on $B_1(0)$.
\end{remark}
\begin{remark}
We have the easy estimate $|\nabla f|(x)\leq r_f(x)^{-1}$.  However, a lower bound on $r_f(x)$ is in principle {\it much} stronger than an upper bound on $|\nabla f|(x)$.
\end{remark}
\begin{remark}
Notice that the regularity scale is a Lipschitz function with $|\nabla r_f(x)|\leq 1$.
\end{remark}
\begin{remark}
If $f$ satisfies an elliptic equation, e.g. is weakly harmonic, then we have the estimate
\begin{align}
\sup_{B_{r_f/2}(x)} |\nabla^k f|\leq C_k \,r_f(x)^{-k}\, .
\end{align}
In particular, control on $r_f$ gives control on all higher order derivatives.
\end{remark}
\vspace{.5 cm}
Now let us state our main estimates for minimizing harmonic maps:\\

\begin{theorem}[Estimates on Minimizing Harmonic Maps]\label{t:main_min_weak_L3}
Let $f:B_2(p)\subseteq M\to N$ be a minimizing harmonic mapping satisfying \eqref{e:manifold_bounds} with $\fint_{B_2(p)}|\nabla f|^2 \leq \Lambda$.  Then there exists $C(n,K_M,K_N,\Lambda)$ such that
\begin{align}\label{e:main_min_weak_L3:1}
\Vol\Big(\{x\in B_1(p): |\nabla f|> r^{-1}\}\Big)\leq \Vol\Big(\{x\in B_1(p): r_f(x)< r\}\Big)\leq Cr^3\, .
\end{align}
In particular, both $|\nabla f|$ and $r_f^{-1}$ have uniform bounds in $L^3_{weak}\Big(B_1(p)\Big)$, the space of weakly $L^3$ functions.  In fact, we also have the Hessian estimate
\begin{align}\label{e:main_min_weak_L3:2}
&\Vol\Big(\{x\in B_1(p): |\nabla^2 f|> r^{-2}\}\Big)\leq Cr^3\, ,
\end{align}
which in particular gives us uniform bounds for $|\nabla^2 f|$ in $L^{3/2}_{weak}(B_1(p))$, the space of weakly  $L^{3/2}$ functions.
\end{theorem}
\begin{remark}
The $L^3_{weak}$ estimates are sharp in that there exist examples for which $|\nabla f|$ does not live in $L^3$, see Section \ref{sss:Lp_sharp_example}.
\end{remark}

\vspace{.5 cm}

\subsection{Results under Additional Hypothesis}\label{ss:additional_results}

The previous two subsections have focused on results for completely general stationary or minimizing harmonic maps.  In this subsection we would like to see how these results may be improved under further assumptions.  Specifically, for stationary harmonic maps $f:M\to N$ we would like to see that the regularity results may be improved to match those of minimizing harmonic maps if we assume there are no smooth harmonic maps from $S^2$ into $N$.  The idea behind this follows that of \cite{lin_stat},\cite{ChNaHa2}.  Additionally, though the regularity results of Sections \ref{ss:stationary_results} and \ref{ss:minimizing_results} are sharp in complete generality, they may be improved if we assume that $N$ has no other stationary or minimizing harmonic maps from $S^k$ into $N$.  Precisely, the main result of this subsection for stationary harmonic maps is the following:\\

\begin{theorem}[Improved Estimates for Stationary Maps]\label{t:improved_results_stationary}
Let $f:B_2(p)\subseteq M\to N$ be a stationary harmonic mapping satisfying \eqref{e:manifold_bounds} with $\fint_{B_2(p)}|\nabla f|^2 \leq \Lambda$.  Assume further that for some $k\geq 2$ there exists no smooth nonconstant stationary harmonic maps $S^\ell\to N$ for all $\ell\leq k$.  Then there exists $C(n,K_M,N,\Lambda)$ such that
\begin{align}
\Vol\Big(\{x\in B_1(p): |\nabla f|> r^{-1}\}\Big)\leq \Vol\Big(\{x\in B_1(p): r_f(x)< r\}\Big)\leq Cr^{2+k}\, .
\end{align}
In particular, both $|\nabla f|$ and $r_f^{-1}$ have uniform bounds in $L^{2+k}_{weak}\Big(B_1(p)\Big)$, the space of weakly $L^{2+k}$ functions.  
\end{theorem}
\begin{remark}
The proof of Theorem \ref{t:improved_results_stationary} follows verbatim the proof of Theorem \ref{t:main_min_weak_L3}, except one replaces the $\epsilon$-regularity of theorem \ref{t:eps_reg} with the $\epsilon$-regularity of theorem \ref{t:eps_reg_improved_stationary}.
\end{remark}
\vspace{1cm}

In the context where $f$ is minimizing we have a similar improvement, though in this case we only need to assume there exists no minimizing harmonic maps from $S^k\to N$:\\

\begin{theorem}[Improved Estimates for Minimizing Maps]\label{t:improved_results_minimizing}
Let $f:B_2(p)\subseteq M\to N$ be a minimizing harmonic mapping satisfying \eqref{e:manifold_bounds} with $\fint_{B_2(p)}|\nabla f|^2 \leq \Lambda$.  Assume that for some $k\geq 2$ there exists no smooth nonconstant minimizing harmonic maps $S^\ell\to N$ for all $\ell\leq k$.  Then there exists $C(n,K_M,N,\Lambda)$ such that
\begin{align}
\Vol\Big(\{x\in B_1(p): |\nabla f|> r^{-1}\}\Big)\leq \Vol\Big(\{x\in B_1(p): r_f(x)< r\}\Big)\leq Cr^{2+k}\, .
\end{align}
In particular, both $|\nabla f|$ and $r_f^{-1}$ have uniform bounds in $L^{2+k}_{weak}\Big(B_1(p)\Big)$, the space of weakly $L^{2+k}$ functions.  
\end{theorem}
\begin{remark}
The proof of Theorem \ref{t:improved_results_minimizing} follows verbatim the proof of Theorem \ref{t:main_min_weak_L3}, except one replaces the $\epsilon$-regularity of theorem \ref{t:eps_reg} with the $\epsilon$-regularity of theorem \ref{t:eps_reg_improved_minimizing}.
\end{remark}
\vspace{1cm}

\subsection{Outline of Proofs and Techniques}\label{ss:outline_proof}

In this subsection we give a brief outline of the proof of the main theorems.  To describe the new ingredients involved it will be helpful to give a comparison to the proofs of previous results in the area, in particular the dimension estimate \eqref{e:sing_set_haus_est} of \cite{ScUh_RegHarm} (which are similar in spirit to the estimates for minimal surfaces in \cite{almgren_exreg} and \cite[theorem 2.26]{almgren_big}) and the Minkowski and $L^p$ estimates \eqref{e:sing_set_Lp_est} of \cite{ChNa2}.

Indeed, the starting point for the study of singular sets for solutions of geometric equations typically looks the same, that is, one needs a monotone quantity.  In the case of harmonic maps $f:M\to N$ between Riemannian manifolds we consider the normalized Dirichlet energy
\begin{align}
\theta_r(x) \equiv r^{2-n}\int_{B_r(x)}|\nabla f|^2\, . 
\end{align}
For simplicity sake let us take $M\equiv \dR^n$ in this discussion, which is really of no loss except for some small technical work (it is the nonlinearity of $N$ that is the difficulty).  Then $\frac{d}{dr}\theta_r\geq 0$, 
and $\theta_r(x)$ is independent of $r$ if and only if $f$ is $0$-symmetric, see Section \ref{ss:monotonicity} for more on this.  Interestingly, this is the only information one requires to prove the dimension estimate \eqref{e:sing_set_haus_est}.  Namely, since $\theta_r(x)$ is monotone and bounded, it must converge as $r$ tends to zero.  In particular, if we consider the sequence of scales $r_\alpha=2^{-\alpha}$ then we have for each $x$ that
\begin{align}\label{e:convergence}
\lim_{\alpha\to \infty} \Big|\theta_{r_\alpha}(x)-\theta_{r_{\alpha+1}}(x)\Big|\to 0\, .
\end{align}
From this one can conclude that {\it every} tangent map of $f$ is $0$-symmetric.  This fact combined with some very general dimension reduction arguments originating with Federer from geometric measure theory \cite{simon_stat}, which do not depend on the harmonic behavior of $f$ at all, yield the dimension estimate \eqref{e:sing_set_haus_est} from \cite{ScUh_RegHarm}.

The improvement in \cite{ChNa2} of the Hausdorff dimension estimate \eqref{e:sing_set_haus_est} to the Minkowski content estimate \eqref{e:sing_set_mink_delta_est}, and therefore the $L^p$ estimate of \eqref{e:sing_set_Lp_est}, requires exploiting more about the monotone quantity $\theta_r(x)$ than that it limits as $r$ tends to zero.  Indeed, an effective version of \eqref{e:convergence} says that for each $\delta>0$ there exists $N(\Lambda,\delta)>0$ such that 
\begin{align}\label{e:convergence_effective}
\Big|\theta_{r_\alpha}(x)-\theta_{r_{\alpha+1}}(x)\Big|<\delta\, ,
\end{align}
holds for all except for at most $N$ scales $\alpha\in \{\alpha_1,\ldots,\alpha_N\}\subseteq\dN$.  These {\it bad} scales where \eqref{e:convergence_effective} fails may differ from point to point, but the number of such scales is uniformly bound.  This allows one to conclude that for all but at most $N$-scales that $B_{r_{\alpha}}(x)$ is $(0,\epsilon)$-symmetric, see Section \ref{ss:cone_splitting}.  To exploit this information a new technique other than dimension reduction was required in \cite{ChNa2}.  Indeed, in \cite{ChNa2} the quantitative $0$-symmetry of \eqref{e:convergence_effective} was instead combined with the notion of cone splitting and an energy decomposition in order to conclude the estimates \eqref{e:sing_set_Lp_est},\eqref{e:sing_set_mink_delta_est}.  Since we will use them in this paper, the quantitative $0$-symmetry and cone splitting will be reviewed further in Section \ref{ss:cone_splitting}.\\

Now let us begin to discuss the results of this paper.  The most challenging aspect of this paper is the proof of the estimates on the quantitative stratifications of Theorems \ref{t:main_quant_strat_stationary} and \ref{t:main_eps_stationary}, and so we will focus on these in our outline.  Let us first observe that it might be advantageous to replace \eqref{e:convergence_effective} with a version that forces an actual rate of convergence, see for instance \cite{Simon_RegMin}.  More generally, if one is in a context where an effective version of tangent cone uniqueness can be proved then this may be exploited.  In fact, in the context of critical sets of elliptic equations one can follow exactly this approach, see the authors work \cite{NaVa} where versions of Theorems \ref{t:main_quant_strat_stationary} and \ref{t:main_eps_stationary} were first proved in this context.  However, in the general context of this paper such an approach fails, as tangent cone uniqueness is not available, and potentially not correct.\\

Instead, we will first replace \eqref{e:convergence_effective} with the following relatively simple observation.  Namely, for each $x$ there exists $N(\Lambda,\delta)$ and a finite number of scales $\{\alpha_1,\ldots,\alpha_N\}\subseteq \dN$ such that
\begin{align}\label{e:convergence_effective2}
\sum_{\alpha_{j}< \alpha<\alpha_{j+1}} \Big|\theta_{r_{\alpha}}(x)-\theta_{r_{\alpha+1}}(x)\Big| < \delta\, .
\end{align}
That is, not only does the energy drop by less than $\delta$ between these scales, but the sum of all the energy drops is less than $\delta$ between these scales.\\

Unfortunately, exploiting \eqref{e:convergence_effective2} turns out to be {\it substantially} harder to use than exploiting either \eqref{e:convergence} or even \eqref{e:convergence_effective}.  In essence, this is because it is not a local assumption in terms of scale, and one needs estimates which can see many scales simultaneously, but which do not require any form of tangent cone uniqueness statements.  Accomplishing this requires two new ingredients, a new rectifiable-Reifenberg type theorem, and a new $L^2$-best subspace approximation theorem for stationary harmonic maps, which will allow us to apply the rectifiable-Reifenberg.  Let us discuss these two ingredients separately.\\

We begin by discussing the new $W^{1,p}$-Reifenberg and rectifiable-Reifenberg theorems, which are introduced and proved in Section \ref{s:bi-Lipschitz_reifenberg}.  Recall that the classical Reifenberg theorem, reviewed in Section \ref{ss:reifenberg}, gives criteria under which a set becomes $C^\alpha$-H\"older equivalent to a ball $B_1(0^k)$ in Euclidean space.  In the context of this paper, it is important to improve on this result so that we have gradient and volume control of our set.  Let us remark that there have been many generalizations of the classical Reifenberg theorem in the literature, see for instance \cite{Toro_reif} \cite{davidtoro}, however those results have hypotheses which are much too strong for the purposes of this paper.  Instead, we will focus on improving the $C^\alpha$-equivalence to a $W^{1,p}$-equivalence.  This is strictly stronger by Sobolev embedding, and if $p>k$ then this results in volume estimates and a rectifiable structure for the set.  More generally, we will require a version of the theorem which allows for more degenerate structural behavior, namely a rectifiable-Reifenberg theorem.  In this case, the assumptions will conclude that a set $S$ is rectifiable with volume estimates.  Of course, what is key about this result is that the criteria will be checkable for our quantitative stratifications, thus let us discuss these criteria briefly.  Roughly, if $S\subseteq B_1(0^n)$ is a subset equipped with the $k$-dimensional Hausdorff measure $\lambda^k$, then let us define the $k$-dimensional distortion of $S$ by
\begin{align}
D^k_S(y,s)\equiv s^{-2}\inf_{L^k}s^{-k}\int_{S\cap B_{s}(y)}d^2(z,L^k)\,d\lambda^k(z)\, ,
\end{align}
where the $\inf$ is taken over all $k$-dimensional affine subspaces of $\dR^n$.  That is, $D^k$ measures how far $S$ is from being contained in a $k$-dimensional subspace.

Our rectifiable-Reifenberg then requires this be small on $S$ in an integral sense, more precisely that
\begin{align}\label{e:reifenberg_integral_estimate}
&r^{-k}\int_{S\cap B_r(x)}\sum_{r_\alpha\leq r}  D^k(y,r_\alpha) d\lambda^k(y)<\delta^2\, .
\end{align}
For $\delta$ sufficiently small, the conclusions of the rectifiable-Reifenberg Theorem \ref{t:reifenberg_W1p_holes} are that the set $S$ is rectifiable with effective bounds on the $k$-dimensional measure.  Let us remark that one cannot possibly conclude better than rectifiable under this assumption, see for instance the Examples of Section \ref{ss:examples}.

Thus, in order to prove the quantitative stratification estimates of Theorems \ref{t:main_quant_strat_stationary} and \ref{t:main_eps_stationary}, we will need to verify that the integral conditions \eqref{e:reifenberg_integral_estimate} hold for the quantitative stratifications $S^k_\epsilon(f)$, $S^k_{\epsilon,r}(f)$ on all balls $\B{r}{x}$.  In actuality the proof is more complicated.  We will need to apply a discrete version of the rectifiable-Reifenberg, which will allow us to build an iterative covering of the quantitative stratifications, and each of these will satisfy \eqref{e:reifenberg_integral_estimate}.  This will allow us to keep effective track of all the estimates involved.  However, let us for the moment just focus on the main estimates which allows us to turn \eqref{e:reifenberg_integral_estimate} into information about our harmonic maps, without worrying about such details.\\  

Namely, in Section \ref{s:best_approx} we prove a new and very general approximation theorem for stationary harmonic maps.  As always in this outline, let us assume $M\equiv \dR^n$, the general case is no harder and we simply work on an injectivity radius ball.  Thus we consider a stationary harmonic map $f:B_{16}(0^n)\to N$, as well as an arbitrary measure $\mu$ which is supported on $B_1(0^n)$.  We would like to study how closely the support of $\mu$ can be approximated by a $k$-dimensional affine subspace $L^k\subseteq \dR^n$, in the appropriate sense, and we would like to estimate this distance by using properties of $f$.  Indeed, if we assume that $B_8(0^n)$ is {\it not} $(k+1,\epsilon)$-symmetric with respect to $f$, then for an arbitrary $\mu$ we will prove in Theorem \ref{t:best_approximation} that
\begin{align}\label{e:measure_L2_dist}
\inf_{L^k\subseteq \dR^n}\int d^2(x,L^k)\,d\mu\leq C\int |\theta_8(x)-\theta_{1}(x)|\,d\mu \, ,
\end{align}
where $C$ will depend on $\epsilon$, energy bound of $f$, and the geometry of $M$ and $N$.  That is, if $f$ does not have $k+1$ degrees of symmetry, then how closely the support of an arbitrary measure $\mu$ can be approximated by a $k$-dimensional subspace can be estimated by looking at the energy drop of $f$ along $\mu$.  In applications, $\mu$ will be the restriction to $B_1$ of some discrete approximation of the $k$-dimensional Hausdorff measure on $S^k_\epsilon$, and thus the symmetry assumption on $f$ will hold for all balls centered on the support of $\mu$.\\

In practice, applying \eqref{e:measure_L2_dist} to \eqref{e:reifenberg_integral_estimate} is subtle and must be done inductively on scale.  Additionally, in order to prove the effective Hausdorff estimates $\lambda^k(S^k_\epsilon\cap B_r)\leq Cr^{k}$ we will need to use the Covering Lemma \ref{l:covering} to break up the quantitative stratification into appropriate pieces, and we will apply the estimates to these.  This decomposition is based on a covering scheme first introduced by the authors in \cite{NaVa}.  Thus for the purposes of our outline, let us assume we have already proved the Hausdorff estimate $\lambda^k(S^k_\epsilon\cap B_r)\leq Cr^{k}$, and use this to be able to apply the rectifiable-Reifenberg in order to conclude the rectifiability of the singular set.  This may feel like a large assumption, however it turns out the proof of the Hausdorff estimate will itself be proved by a very similar tactic, though will require an inductive argument on scale, and will use the discrete rectifiable-Reifenberg of Theorem \ref{t:reifenberg_W1p_discrete} in place of the rectifiable-Reifenberg of Theorem \ref{t:reifenberg_W1p_holes}.\\

Thus let us choose a ball $B_r$ and let $E\equiv \sup_{B_{r}} \theta_r(y)$.  Let us consider the subset $\tilde S^k_{\epsilon}\subseteq S^k_\epsilon\cap B_r$ defined by
\begin{align}
\tilde S^k_{\epsilon}\equiv\{y\in S^k_\epsilon\cap B_r: \theta_0(y)>E-\eta\}\, ,
\end{align}
where $\eta=\eta(n,K_M,K_N,\Lambda,\epsilon)$ will be chosen appropriately later.  We will show now that $\tilde S^k_{\epsilon}$ is rectifiable.  Since $\eta$ is fixed and the ball $B_r$ is arbitrary, the rectifiability of all of $S^k_\epsilon$ follows quickly from a covering argument.  Thus, let us estimate \eqref{e:reifenberg_integral_estimate} by plugging in \eqref{e:measure_L2_dist} and the Hausdorff estimate to conclude:
\begin{align}\label{e:outline:2}
r^{-k}\int_{\tilde S^k_\epsilon}&\sum_{r_\alpha\leq r} D^k(x,r_\alpha)\,d\lambda^k \notag\\
&= r^{-k}\int_{\tilde S^k_\epsilon}\sum_{r_\alpha\leq r}\Big(\inf_{L^k}r_\alpha^{-2-k}\int_{\tilde S^k_\epsilon\cap B_{r_\alpha}(x)}d^2(y,L^k)d\lambda^k(y) \Big) d\lambda^k(x)\notag\\
&\leq Cr^{-k}\int_{\tilde S^k_\epsilon}\sum_{r_\alpha\leq r}\Big(r_\alpha^{-k}\int_{\tilde S^k_\epsilon\cap B_{r_\alpha}(x)}|\theta_{8r_\alpha}(y)-\theta_{r_\alpha}(y)|d\lambda^k(y) \Big) d\lambda^k(x)\notag
\end{align}
\begin{align}
&= Cr^{-k}\sum_{r_\alpha\leq r}r_\alpha^{-k}\int_{\tilde S^k_\epsilon}\lambda^k(\tilde S^k_\epsilon\cap B_{r_\alpha}(y))|\theta_{8r_\alpha}(y)-\theta_{r_\alpha}(y)|\,d\lambda^k(y)\, ,\notag\\
&\leq Cr^{-k}\int_{\tilde S^k_\epsilon}\sum_{r_\alpha\leq r}|\theta_{8r_\alpha}(y)-\theta_{r_\alpha}(y)|\, d\lambda^k(y)\, ,\notag\\
&\leq Cr^{-k}\int_{\tilde S^k_\epsilon}|\theta_{8r}(y)-\theta_{0}(y)|\, d\lambda^k(y)\notag\\
&\leq Cr^{-k}\lambda_k(\tilde S^k_\epsilon)\cdot \eta\notag\\
&<\delta^2\, .
\end{align}
where in the last line we have chosen $\eta=\eta(n,K_M,K_N,\Lambda,\epsilon)$ so that the estimate is less than the required $\delta$ from the rectifiable-Reifenberg.  Thus we can apply the rectifiable-Reifenberg of Theorem \ref{t:reifenberg_W1p_holes} in order to conclude the rectifiability of the set $\tilde S^k_\epsilon$, which in particular proves that $S^k_\epsilon$ is itself rectifiable, as claimed.


\vspace{1cm}
  
\section{Preliminaries}\label{s:Prelim}

\subsection{Stationary Harmonic Maps and Monotonicity}\label{ss:monotonicity}

The key technical tool available to a stationary harmonic map $f:B_2(p)\subseteq M\to N$, which is not available to a weakly harmonic map, is that of a monotone quantity.  Namely, given $x\in B_1(p)$ and $r\leq 1$ we can consider the normalized Dirichlet energy defined by
\begin{align}\label{e:normalized_Dirichlet}
\theta_r(x)\equiv r^{2-n}\int_{B_r(x)}|\nabla f|^2\, . 
\end{align}

If $M\equiv \dR^n$ 
, then for each $x$ fixed we have that $\theta_r(x)$ is exactly a monotone increasing function of $r$.  More precisely, we have that
\begin{align}
\frac{d}{dr}\theta_r(x) = 2r^{2-n}\int_{\partial B_r(x)} \Big|\frac{\partial f}{\partial r}\Big|^2 = 2r^{2-n}\int_{\partial B_r(x)} |\langle \nabla f,n_x\rangle|^2\, .
\end{align}
Hence, we see that if $\theta_r(x)$ is independent of $r$ then we have that $f$ is $0$-symmetric with respect to $x$.  More generally, if $\theta_s(x)=\theta_r(x)$ then $f$ is radially symmetric on the annulus $A_{s,r}(x)=\B{r}{x}\setminus \B{s}{x}$.  Motivated by this, we see that what we are really interested in is the amount the Dirichlet energy drops from one scale to the next, and thus we define
\begin{align}
W_{s,r}(x) \equiv \theta_r(x)-\theta_s(x)\geq 0\, .
\end{align}
Oftentimes we will want to enumerate our choice of scale, so let us define the scales $r_\alpha\equiv 2^{-\alpha}$ for $\alpha\geq 0$, and the corresponding Dirichlet energy drop:
\begin{align}
W_{\alpha}(x) \equiv W_{r_{\alpha},r_{\alpha-3}}(x)\equiv \theta_{r_{\alpha-3}}(x)-\theta_{r_{\alpha}}(x)\geq 0\, .
\end{align}

In the general case, i.e., when $M\neq \dR^n$, essentially the same statements may be made, however $\theta_r(x)$ is now only {\it almost} monotone, meaning that $e^{Cr}\theta_r(x)$ is monotone for some constant which depends only on the geometry of $M$.  From this one can prove that at every point, every tangent map is $0$-symmetric, which is the starting point for the dimension estimate \eqref{e:sing_set_haus_est} of \cite{ScUh_RegHarm}.  
In Section \ref{ss:cone_splitting} we will discuss quantitative versions of this point, first introduced in \cite{ChNa2} and used in this paper as well, and also generalizations which involve higher degrees of symmetry.  These points were first used in \cite{ChNa2} to prove Minkowski estimates on the quantitative stratification of a stationary harmonic map.  They will also play a role in our arguments, though in a different manner.\\

\subsection{\texorpdfstring{Quantitative $0$-Symmetry and Cone Splitting}{Quantitative 0-Symmetry and Cone Splitting}}\label{ss:cone_splitting}

In this subsection we review some of the quantitative symmetry and splitting results of \cite{ChNa2}, in particular those which will play a role in this paper.  

The first result we will discuss acts as an {\it effective} formulation of the fact that every tangent map is $0$-symmetric.  Namely, the quantitative $0$-symmetry of \cite{ChNa2} states that for each $\epsilon>0$ and point $x\in B_1$, all but a finite number of the balls $\cur{B_{r_\alpha}(x)}_{\alpha\in \N}$ are $(0,\epsilon)$-symmetric, where $r_\alpha\equiv 2^{-\alpha}$.  Precisely:\\

\begin{theorem}[Quantitative $0$-Symmetry \cite{ChNa2}]\label{t:quantitative_0_symmetry}
Let $f:B_2(p)\subseteq M\to N$ be a stationary harmonic map satisfying \eqref{e:manifold_bounds} with $\fint_{B_2(p)}|\nabla f|^2\leq \Lambda$ and let $\epsilon>0$ be fixed.  Then the following hold:
\begin{enumerate}
\item There exists $\delta(n,K_N,\Lambda,\epsilon)>0$ such that for each $x\in B_1(p)$ if $|\theta_{r}(x)-\theta_{\delta r}(x)|<\delta$ and $K_M<\delta$, then $B_r(x)$ is $(0,\epsilon)$-symmetric.
\item For each $x\in B_1(p)$ there exists a finite number of scales $\{\alpha_1,\ldots,\alpha_K\}\subseteq \dN$ with $K\leq K(n,K_M,K_N,\Lambda,\epsilon)$ such that for $r\not\in (r_{\alpha_j}/2,2r_{\alpha_j})$ we have that $B_{r}(x)$ is $(0,\epsilon)$-symmetric.
\end{enumerate}
\end{theorem}
\begin{remark}
In \cite{ChNa2} it is stated that the constant depends on the manifold $N$, not just $K_N$, however it is easy to check that only $K_N$ is important in the proof.  See the proof of Lemma \ref{l:best_subspace:energy_lower_bound} for a relevant argument.
\end{remark}
\begin{remark}
The assumption $K_M<\delta$ is of little consequence, since this just means focusing the estimates on balls of sufficiently small radius after rescaling.
\end{remark}

Another technical tool that played an important role in \cite{ChNa2} was that of cone splitting.  This will be used in this paper when proving the existence of unique tangent planes of symmetry for the singular set, so we will discuss it here.  In short, cone splitting is the idea that multiple $0$-symmetries add to give rise to a $k$-symmetry.  To state it precisely let us give a careful definition of the notion of independence of a collection of points:

\begin{definition}\label{d:independent_points}
We say a collection of points $\{x_1,\ldots,x_\ell\}\in \dR^n$ is independent if they are linearly independent.  We say the collection is $\tau$-independent if $d(x_{k+1},\text{span}\{x_1,\ldots,x_k\})>\tau$ for each $k$.  If $\{x_1,\ldots,x_\ell\}\in M$ then we say the collection is $\tau$-independent with respect to $x$ if $d(x,x_j)<\inj(x)$ and the collection is $\tau$-independent when written in exponential coordinates at $x$.
\end{definition}
\vspace{.5 cm}

Now we are in a position to state the effective cone splitting of \cite{ChNa2}:

\begin{theorem}[Cone Splitting \cite{ChNa2}]\label{t:con_splitting}
Let $f:B_3(p)\subseteq M\to N$ be a stationary harmonic map satisfying \eqref{e:manifold_bounds} with $\fint_{B_2(p)}|\nabla f|^2\leq \Lambda$ and let $\epsilon,\tau >0$ be fixed.  Then there exists $\delta(n,K_N,\Lambda,\epsilon,\tau)>0$ such that if $K_M<\delta$ and $x_0,\ldots,x_{k}\in B_1(p)$ are such that
\begin{enumerate}
\item $B_2(x_j)$ are $(0,\delta)$-symmetric,
\item $\{x_0,\ldots,x_k\}$ are $\tau$-independent at $p$,
\end{enumerate}
then $B_1(p)$ is $(k,\epsilon)$-symmetric.
\end{theorem}
\vspace{.5 cm}

We end with the following, which one can view as a quantitative form of dimension reduction.  The proof is standard and can be accomplished by a contradiction argument:

\begin{theorem}[Quantitative Dimension Reduction]\label{t:quant_dim_red}
Let $f:B_2(p)\subseteq M\to N$ be a stationary harmonic map satisfying \eqref{e:manifold_bounds} with $\fint_{B_2(p)}|\nabla f|^2\leq \Lambda$.  Then for each $\epsilon>0$ there exists $\delta(n,K_M,K_N,\Lambda,\epsilon), r(n,K_M,K_N,\Lambda,\epsilon)>0$ such that if $B_2(p)$ is $(k,\delta)$-symmetric with respect to some $k$-plane $V^k$, then for each $x\in B_1(p)\setminus B_\epsilon(V^k)$ we have that $B_r(x)$ is $(k+1,\epsilon)$-symmetric.
\end{theorem}
\vspace{.5 cm}

\subsection{Defect Measure}\label{ss:defect_measure}

Since it will play a role for us in one technical aspect of this paper, let us recall in this section the defect measure of a sequence of stationary harmonic maps.  We begin with a definition:

\begin{definition}\label{d:defect}
Let $f_i:B_2(p)\to N$ be a sequence of stationary harmonic maps satisfying \eqref{e:manifold_bounds} with $\fint_{B_2(p)}|\nabla f_i|^2\leq \Lambda$.  Then, after possibly passing to a subsequence, $f_i$ has a weak $W^{1,2}$ limit $f$, and we can consider the measure
\begin{align}
|\nabla f_i|^2 dv_g \wto \abs{\nabla f}^2 dv_g + \nu\, ,
\end{align}
where this last convergence is intended in the weak sense of Radon measures. The measure $\nu$ is nonnegative by Fatou's lemma, and it is called the defect measure associated to $f_i\wto f$.
\end{definition}


\begin{remark}
We can also allow $f_i:B_2(p_i)\to N$ to be defined on different manifolds.  In practice, this will occur under the assumption $K_{M_i}\to 0$, so that $\nu$ becomes a measure on $B_2(0^n)$.
\end{remark}

\vspace{.5 cm}

The following theorem is one of the main accomplishments of \cite{lin_stat}:

\begin{theorem}[Rectifiability of Defect Measure]\cite[lemma 1.7]{lin_stat}\label{t:defect:rectifiable}
If $\nu$ is a defect measure as in definition \ref{d:defect}, then it is $n-2$ rectifiable.
\end{theorem}
\vspace{.5 cm}

A key tool in the proof of the above, which will be useful in this paper as well, is the following $0$-symmetry result:

\begin{theorem}\cite[lemma 1.7 (ii)]{lin_stat}\label{t:defect:0symmetry}
Let $f_i:B_2(p_i)\subseteq M_i\to N$ be a sequence of stationary harmonic maps satisfying \eqref{e:manifold_bounds} with $\fint_{B_2(p)}|\nabla f_i|^2\leq \Lambda$.  Assume $K_{M_i}<\delta_i\to 0$ and that $\big|\theta_{i,2}(p_i)-\theta_{i,\delta_i}\big|<\delta_i\to 0$.  Then the defect measure $|\nabla f_i|^2 dv_g\wto \abs{\nabla f}^2dv_g + \nu$ is $0$-symmetric, that is, both the function $f$ and $\nu$ are invariant under dilation around the origin.
\end{theorem}
\vspace{.5 cm}

\subsection{\texorpdfstring{$\epsilon$-regularity for Minimizers}{epsilon-regularity for Minimizers}}\label{ss:eps_reg}

In this subsection we quickly review the $\epsilon$-regularity theorems of \cite{ChNa2}, which themselves build on the difficult work of \cite{ScUh_RegHarm}.  This will be our primary technical tool in upgrading the structural results on stationary harmonic maps to the regularity results for minimizing harmonic maps.  Recall the definition of the regularity scale given in Definition \ref{d:regularity_scale}, then the main theorem of this subsection is the following:\\

\begin{theorem}[Minimizing $\epsilon$-Regularity \cite{ChNa2}]\label{t:eps_reg}
Let $f:B_2(p)\subseteq M\to N$ be a minimizing harmonic map satisfying \eqref{e:manifold_bounds} with $\fint_{B_2(p)}|\nabla f|^2\leq \Lambda$.  Then there exists $\epsilon(n,K_N,\Lambda)>0$ such that if $K_M<\epsilon$ and $B_2(p)$ is $(n-2,\epsilon)$-symmetric, then $$r_f(p)\geq 1\, .$$
\end{theorem}
\begin{remark}
As previously remarked, the assumption $K_M<\epsilon$ is of little consequence, since this just means focusing the estimates on balls of sufficiently small radius after rescaling.
\end{remark}
\vspace{.5cm}

The following stationary version was proved in \cite{ChNaHa2}, and is essentially just a combination of the defect measure ideas of \cite{lin_stat}, the $\epsilon$-regularity of \cite{beth}, and a contradiction argument:\\

\begin{theorem}[Stationary $\epsilon$-Regularity \cite{ChNaHa2}]
Let $f:B_2(p)\subseteq M\to N$ be a minimizing harmonic map satisfying \eqref{e:manifold_bounds} with $\fint_{B_2(p)}|\nabla f|^2\leq \Lambda$.  Assume that there exists no nonconstant stationary harmonic map $S^2\to N$. Then there exists $\epsilon(n,N,\Lambda)>0$ such that if $K_M<\epsilon$ and $B_1(p)$ is $(n-2,\epsilon)$-symmetric, then $$r_f(p)\geq \frac 1 2\, .$$
\end{theorem}
\vspace{1cm}

Let us also discuss some improvements of the above $\epsilon$-regularity theorems in the case where there are no nonconstant harmonic maps from $S^k$ into $N$.  These are the key $\epsilon$-regularity results needed for Theorems \ref{t:improved_results_stationary} and \ref{t:improved_results_minimizing}.  We begin with the minimizing case:\\

\begin{theorem}[Improved Minimizing $\epsilon$-Regularity \cite{ChNaHa2}]\label{t:eps_reg_improved_minimizing}
Let $f:B_2(p)\subseteq M\to N$ be a minimizing harmonic mapping satisfying \eqref{e:manifold_bounds} with $\fint_{B_2(p)}|\nabla f|^2 \leq \Lambda$.  Assume further that for some $k\geq 2$ there exists no nonconstant minimizing harmonic maps $S^\ell\to N$ for all $\ell\leq k$.  Then there exists $\epsilon(n,N,\Lambda)$ such that if $K_M<\epsilon$ and $B_1(p)$ is $(n-k-1,\epsilon)$-symmetric, then $$r_f(p)\geq \frac 1 2\, .$$
\end{theorem}
\vspace{.5cm}

Finally, we end with the stationary version of the above theorem:\\

\begin{theorem}[Improved Stationary $\epsilon$-Regularity \cite{ChNa2}]\label{t:eps_reg_improved_stationary}
Let $f:B_2(p)\subseteq M\to N$ be a stationary harmonic mapping satisfying \eqref{e:manifold_bounds} with $\fint_{B_2(p)}|\nabla f|^2 \leq \Lambda$.  Assume further that for some $k\geq 2$ there exists no nonconstant stationary harmonic maps $S^\ell\to N$ for all $\ell\leq k$.  Then there exists $\epsilon(n,N,\Lambda)$ such that if $K_M<\epsilon$ and $B_1(p)$ is $(n-k-1,\epsilon)$-symmetric, then $$r_f(p)\geq \frac 1 2\, .$$
\end{theorem}
\begin{remark}
 Note that in both theorems above, nothing would change if we only required that $N$ admits no \textit{continuous} (or even \textit{smooth}) nonconstant stationary (or minimizing) harmonic map $S^\ell\to N$ for all $\ell\leq k$. Indeed, the tangent maps at singular points of non-smooth stationary (or minimizing) harmonic maps give rise to smooth stationary (or minimizing) harmonic maps from lower dimensional spheres.
\end{remark}

\vspace{.5cm}

\subsection{Hausdorff, Minkowski, and packing Content}\label{ss:haus_mink}

In this subsection we give a brief review of the notions of Hausdorff, Minkowski, and packing content.  We will also use this to recall the definition of Hausdorff measure.  The results of this subsection are completely standard, but this gives us an opportunity to introduce some notation for the paper.  For a more detailed reference, we refer the reader to \cite{mattila,Fed}. Let us begin with the notions of content:

\begin{definition}\label{d:content}
Given a set $S\subseteq \dR^n$ and $r>0$ we define the following:
\begin{enumerate}
\item The $k$-dimensional Hausdorff $r$-content of $S$ is given by
\begin{align}
\lambda^k_r(S)\equiv \inf\Big\{\sum \omega_k r_i^k : S\subseteq \bigcup B_{r_i}(x_i) \text{ and }r_i\leq r\Big\}\, .
\end{align}
\item The $k$-dimensional Minkowski $r$-content of $S$ is given by
\begin{align}
m^k_r(S)\equiv (2r)^{s-n}\Vol \ton{\B r S}\, .
\end{align}
\item The $k$-dimensional packing $r$-content of $S$ is given by
\begin{align}
p^k_r(S)\equiv \sup\Big\{\sum \omega_k r_i^k : \, x_i\in S\text{ and }\{B_{r_i}(x_i)\} \text{ are disjoint}\text{ and }r_i\leq r\Big\}\, .
\end{align}
\end{enumerate}
\end{definition}

These definitions make sense for any $k\in [0,\infty)$, though in this paper we will be particularly interested in integer valued $k$.  Notice that if $S$ is a compact set then $\lambda^k_r(S),m^k_r(S)<\infty$ for any $r>0$, and that we always have the relations
\begin{align}
\lambda^k_r(S) \lesssim m^k_r(S) \lesssim p^k_r(S)\, .
\end{align}
In particular, bounding the Hausdorff content is less powerful than bounding the Minkowski content, which is itself less powerful than bounding the packing content.  

Primarily in this paper we will be mostly interested in content estimates, because these are the most effective estimates.  However, since it is classical, let us go ahead and use the Hausdorff content to define a measure.  To accomplish this, let us more generally observe that if $r\leq r'$ then $\lambda^k_r(S)\geq \lambda^{k'}_r(S)$. In particular, we can define the limit
\begin{align}
\lambda^k_0(S)\equiv \lim_{r\to 0} \lambda^k_r(S)=\sup_{r>0}\lambda^k_r(S)\, .\notag
\end{align}

It turns out that $\lambda^k_0$ is a genuine measure.  

\begin{definition}\label{d:haus_meas}
Given a set $S\subseteq \dR^n$ we define its $k$-dimensional Hausdorff measure by $\lambda^k(S)\equiv \lambda^k_0(S)$.
\end{definition}

Similar constructions can be carried out for the Minkowski and packing content. In particular, we can define 
\begin{gather}
 \overline{m}^k_0 (S)\equiv \limsup_{r\to 0} m^k_r(S)\, ,\quad \underline{m}^k_0 (S)\equiv \liminf_{r\to 0} m^k_r(S)\, ,\\
 p^k_0(S)=\lim_{r\to 0}p^k_r(S)=\inf_{r>0} p^k_r(S)\, . 
\end{gather}

\begin{definition}\label{d:dimension}
Given a set $S\subseteq \dR^n$ we define its Hausdorff and Minkowski dimension (or box-dimension) by
\begin{align}
\dim_H S\equiv \inf\Big\{k\geq 0: \lambda^k_0(S)=0\Big\}\, ,\notag\\
\dim_M S\equiv \inf\Big\{k\geq 0: \overline m^k_0(S)=0\Big\}\, .
\end{align}
\end{definition}
\begin{remark}
 Note that we could define an upper and lower Minkowski dimension by
 \begin{gather}
 \overline{\dim}_M S\equiv \inf\Big\{k\geq 0: \overline m^k_0(S)=0\Big\}\, ,\quad  \underline{\dim}_M S\equiv \inf\Big\{k\geq 0: \underline m^k_0(S)=0\Big\}\, .
 \end{gather}
In general, $\underline{\dim}_M S\leq \overline{\dim}_M S$, where the inequality may be strict. However, for the purposes of this paper we will only be interested in the \textit{upper} Minkowski dimension.
\end{remark}

As an easy example consider the rationals $\dQ^n\subseteq \dR^n$.  Then it is a worthwhile exercise to check that $\dim_H \dQ^n = 0$, while $\dim_M \dQ^n = n$.  \\

A very important notion related to measures is the \textit{density} at a point. Although this is standard, for completeness we briefly recall the definition of Hausdorff density, and refer the reader to \cite[chapter 6]{mattila} for more on this subject.

\begin{definition}
 Given a set $S\subset \R^n$ which is $\lambda^k$-measurable, and $x\in \R^n$, we define the $k$-dimensional upper and lower density of $S$ at $x$ by
 \begin{gather}
  \theta^{\star k}(S,x)  = \limsup_{r\to 0} \frac{\lambda^k(S\cap \B{r}{x})}{\omega_k r^k}\, ,\quad 
  \theta^{k}_\star (S,x) = \liminf_{r\to 0} \frac{\lambda^k(S\cap \B{r}{x})}{\omega_k r^k}\, .
 \end{gather}
\end{definition}
In the following, we will use the fact that for almost any point in a set with finite $\lambda^k$-measure, the density is bounded from above and below.
\begin{proposition}[ \cite{mattila}]\label{prop_dens}
 Let $S\subset \R^n$ be a set with $\lambda^k(S)<\infty$. Then for $k$-a.e. $x\in S$:
 \begin{gather}
  2^{-k}\leq \theta^{\star k}(S,x)\leq 1\, ,
 \end{gather}
while for $k$-a.e. $x\in \R^n \setminus S$
\begin{gather}
 \theta^{\star k}(S,x)=0\, .
\end{gather}
\end{proposition}
\vspace{.5 cm}








\subsection{The Classical Reifenberg Theorem}\label{ss:reifenberg}

In this Section we recall the classical Reifenberg Theorem, as well as some more recent generalizations.  The Reifenberg theorem gives criteria on a closed subset $S\subseteq B_2\subseteq \dR^n$ which determine when $S\cap B_1$ is bi-H\"older to a ball $B_1(0^k)$ in a smaller dimensional Euclidean space.  The criteria itself is based on the existence of {\it good} best approximating subspaces at each scale.  We start by recalling the Hausdorff distance.
\begin{definition}\label{d:haus_dist}
Given two sets $A,B\subseteq \R^n$, we define the Hausdorff distance between these two by
\begin{gather}
 d_H(A,B)=\inf \cur{r\geq 0 \ \ s.t. \ \ A\subset \B{r}{B} \ \ \text{and} \ \ B\subset \B{r}{A} }\, .
\end{gather}
Recall that $d_H$ is a distance on closed sets, meaning that $d_H(A,B)=0$ implies $\overline A = \overline B$.
\end{definition}
\vspace{.5cm}

The classical Reifenberg theorem says the following:\\

\begin{theorem}[Reifenberg Theorem \cite{reif_orig,simon_reif}]\label{t:classic_reifenberg}
For each $0<\alpha<1$ and $\epsilon>0$ there exists $\delta(n,\alpha,\epsilon)>0$ such that the following holds.  Assume $0^n\in S\subseteq B_2\subseteq \dR^n$ is a closed subset, and that for each $B_r(x)\subseteq B_1$ with $x\in S$ we have
\begin{align}\label{e:reifenberg_L_infty_orig}
\inf_{L^k} d_H\big( S\cap B_r(x),L^k\cap B_r(x)\big)<\delta\, r\, ,
\end{align}
where the $\inf$ is taken over all $k$-dimensional affine subspaces $L^k\subseteq \dR^n$.  Then there exists $\phi:B_1(0^k)\to S$ which is a $C^\alpha$ bi-H\"older homeomorphism onto its image with $[\phi]_{C^\alpha},[\phi^{-1}]_{C^\alpha}<1+\epsilon$ and $S\cap B_1\subseteq \phi(B_1(0^k))$.

\end{theorem}

\begin{remark}\label{rem:classic_reifenberg_+}
 In fact, one can prove a little more. In particular, under the hypothesis of the previous theorem, there exists a closed subset $S'\subset \R^n$ such that $S'\cap \B 1 0 = S \cap \B 1 0$ and which is homeomorphic to a $k$-dimensional subspace $0^n\in T_0\subseteq \R^n$ via the $C^\alpha$ bi-H\"older homeomorphism $\phi:T_0\to S'$. Moreover, $\abs{\phi(x)-x}\leq C(n)\delta$ for all $x\in T_0$ and $\phi(x)=x$ for all $x\in T_0\setminus \B 2 0$.
\end{remark}

One can paraphrase the above to say that if $S$ can be well approximated on every ball by a subspace in the $L^\infty$-sense, then $S$ must be bi-\hol to a ball in Euclidean space.\\

Let us also mention that there are several more recent generalizations of the classic Reifenberg theorem.  In \cite{Toro_reif}, the author proves a strengthened version of \eqref{e:reifenberg_L_infty_orig} that allows one to improve bi-\hol to bi-Lipschitz.  Unfortunately, for the applications of this paper the hypotheses of \cite{Toro_reif} are much too restrictive.  We will require a weaker condition than in \cite{Toro_reif}, which is more integral in nature, see Theorem \ref{t:reifenberg_W1p}.  In exchange, we will improve the bi-\hol of the classical Reifenberg to $W^{1,p}$.  \\

We will also need a version of the classical Reifenberg which only assumes that the subset $S$ is contained near a subspace, not conversely that the subspace is also contained near $S$.  In exchange, we will only conclude the set is rectifiable.  A result in this direction was first proved in \cite{davidtoro}, but again the hypotheses are too restrictive for the applications of this paper, and additionally there is a topological assumption necessary for the results of \cite{davidtoro}, which is not reasonable in the context in this paper.  We will see how to appropriately drop this assumption in Theorem \ref{t:reifenberg_W1p_holes}. \\

\subsection{\texorpdfstring{$W^{1,p}$ maps and rectifiability}{W-(1,p) maps and rectifiability}}
In the paper we will be using the structure of $W^{1,p}$ maps for $p>k$ in order to conclude rectifiable structures on sets.  For the reader's convenience, here we recall a standard result about rectifiability and $W^{1,p}$ maps:\\ 
\begin{lemma}\label{lemma_w1p_rec}
 Let $\Omega\subset \R^k$ be an open domain, and let $f:\Omega\to \R^n$ be a $W^{1,p}$ map with $p>k$. Then for all $K\subset \Omega$, $f(K)$ is a $k$-rectifiable set.
\end{lemma}

In order to prove this we need, by definition, to show that there exists a countable sequence of Lipschitz maps $f_i:\R^k\to \R^n$ such that
 \begin{gather}
  \lambda^k\ton{f(K)\setminus \bigcup_i f_i(\R^k) } =0\, .
 \end{gather}
By a classical result, it is possible for $p>k$ to approximate all $W^{1,p}$ maps with Lipschitz maps. More specifically, for all $\epsilon>0$, there exists a Lipschitz map $f_\epsilon$ which coincides with $f$ up to a set $E_\epsilon$ of small $k$-dimensional measure in $\R^k$. In particular, we have
\begin{gather}
 E_\epsilon \equiv \cur{x\in \R^k \ \ s.t. \ \ f(x)\neq f_\epsilon(x)} \, , \quad \lambda^k(E_\epsilon)<\epsilon\, .
\end{gather}
For a detailed reference, see for example \cite[theorem 3, sec 6.6.3]{EG} or \cite[section 3.10]{ziemer}). The only thing left to prove is that $f(E_\epsilon)$ has small measure:
\begin{lemma}\label{lemma_monti}
 Let $f$ be the continuous representative of a $W^{1,p}$ map as above with $p>k$. Then for all measurable subsets $E\subset \Omega$ we have
 \begin{gather}
  \lambda^k(f(E))\leq C(k,n,p) \norm{\nabla f}_{L^p(\Omega)}^{k} \ton{\lambda^k(E)}^{1-\frac{k}{p}}\, .
 \end{gather}
\end{lemma}
This lemma follows from standard Morrey-type estimates, see for example \cite[proposition 2.4]{monti}.
\vspace{.5cm}

\section{\texorpdfstring{The $W^{1,p}$-Reifenberg and rectifiable-Reifenberg Theorems}{The W 1p Reifenberg and rectifiable-Reifenberg Theorems}}\label{s:bi-Lipschitz_reifenberg}

In this Section we state very general Reifenberg type theorems for subsets of Euclidean spaces, their proofs are carried out in Sections \ref{sec_proof_main} and \ref{sec_proofII}.  Although this is a key ingredient of the proof of our main theorems, the results of this Section are of an independent nature, and may be of some separate interest.  Therefore, we will attempt to write them so that this Section may be read independently of the rest of the paper, with the exception of the appropriate Preliminary Sections. \\

As described in Section \ref{ss:reifenberg}, the classical Reifenberg theorem tells us that if a closed subset $S\subseteq B_2\subseteq \dR^n$ can be well approximated on every ball by some $k$-dimensional affine space in the $L^\infty$ sense, then $S\cap B_1$ must itself be homeomorphic, in fact bi-H\"older, to a ball $B_1(0^k)$ in a smaller dimensional Euclidean space.\\

For the applications of this paper, we will need to improve on the classical Reifenberg theorem in several ways.  At the most basic level, $C^\alpha$-equivalence is not strong enough.  We will require control on the gradient of the mappings, since this is the essence of rectifiability and volume control.  For this purpose we will obtain $W^{1,p}$-control on our mappings for $k<p<\infty$.  We will deal with this first in Theorem \ref{t:reifenberg_W1p}, where we will also replace the $L^\infty$ closeness of the Reifenberg theorem with an $L^2$ closeness condition, which turns out to be more natural in the applications.\\

A second manner in which we will need improvement is that we will need to allow for the existence of holes.  That is, in applications our set $S$ may be rectifiable, but it might not be homeomorphic to a ball, this is much too strong.  To deal with this we can weaken the hypothesis of the Reifenberg theorem and only assume that on each ball $B_r(x)$ there is a $k$-dimensional subspace $V$ such that each point of our set $S$ is close to $V$, but not conversely that each point of the subspace $V$ is close to $S$.  We study this case in Theorem \ref{t:reifenberg_W1p_holes}, where we will show under natural conditions that such a set $S$ is rectifiable with volume bounds, which is easily seen to be a sharp result under the assumptions of Theorem \ref{t:reifenberg_W1p_holes}.\\

The last version of the Reifenberg theorem that we prove has a more discrete nature, and it is the one we will use to obtain the volume bounds in the rest of the paper.  In this case, our set $S=\{x_i\}$ will be a discrete set of points paired with radii $r_i>0$ such that $\{B_{r_i}(x_i)\}$ are disjoint balls.  We will associate to this collection the measure $\mu=\sum \omega_k r_i^k \delta_{x_i}$, and under the appropriate assumptions show in Theorem \ref{t:reifenberg_W1p_discrete} that the volume of $\mu$ continues to enjoy the packing upper bound $\mu(B_r(x))<Cr^k$ for all radii $r>r_i$.  In the applications, we will prove the main theorems of this paper by building a series of covers of the quantitative stratifications $S^k_{\epsilon}$, and we will apply the discrete $W^{1,p}$-Reifenberg to each of this to obtain volume bounds.  Only at the last stage will we apply Theorem \ref{t:reifenberg_W1p_holes} in order to obtain the rectifiability of the sets as well.\\

It is worth mentioning that some interesting generalizations of Reifenberg's theorem have been proved in numorous sources.  Generalizations in the spirit of this paper are explored in \cite{david_semmes}, \cite{Toro_reif}, and \cite{davidtoro}.  However, in each of these cases the requirements of these theorems are too stringent to be applicable in our situation.  Very recently, and using techniques independent from this work, the authors in \cite{AzzTol},\cite{Tol} proved necessary and sufficient conditions for the $k$-rectifiability of a set which are closely related to the results of this paper.  However, these results lack any {\it apriori} control over the volume of the sets in question, which is fundamental in this paper to the applications.


\subsection{\texorpdfstring{Statement of Main $W^{1,p}$-Reifenberg and rectifiable-Reifenberg Results}{Statement of Main W 1p - Reifenberg and rectifiable-Reifenberg Results}}\label{ss:reifenberg_statements}


There are three main results we wish to discuss in this subsection, each more general than the last.  In order to keep the statements as clean and intuitive as possible we will introduce the following definitions, and discuss them briefly.



\begin{definition}\label{deph_D}
Let $\mu$ be a measure on $B_2$ with $r>0$ and $k\in \dN$.  Then we define the $k$-dimensional displacement by
\begin{align}\label{eq_deph_D}
D^{k}_\mu(x,r)\equiv \inf_{L^k\subseteq \R^n}r^{-(k+2)}\int_{B_{r}(x)}d^2(y,L^k)\,d\mu(y)\, ,
\end{align}
if $\mu(B_r(x))\geq\epsilon_n r^k\equiv (1000n)^{-7n^2}  r^k$, and $D^{k}_\mu(x,r)\equiv 0$ otherwise, where the $\inf$'s are taken over all $k$-dimensional affine subspaces $L^k\subseteq \dR^n$.  If $S\subseteq B_2$, then we can define its $k$-displacement $D^{k}_S(x,r)$ by associating to $S$ the $k$-dimensional Hausdorff measure $\lambda^k_S$ restricted to $S$. 
\end{definition}
\begin{remark} Note that in literature $D$ is usually referred to as the Jones' number $\beta_2$. 
\end{remark}
\begin{remark}
One can replace $\epsilon_n$ by any smaller lower bound, zero included, and the proofs and statements will all continue to hold.  Our particular choice is based on constants which will be obtained in Section \ref{ss:best_comparison}.
\end{remark}

\begin{remark}
Notice that the definitions are scale invariant.  In particular, if we rescale $B_r\to B_1$ and let $\tilde S$ be the induced set, then $D^{k}_S(x,r)\to D^{k}_{\tilde S}(x,1)$.
\end{remark}
\begin{remark}\label{rem_mon}
Notice the monotonicity given by the following:  If $S'\subseteq S$, then $D^k_{S'}(x,r)\leq D^k_S(x,r)$.
\end{remark}

\begin{remark}\label{r:D_scale_control}
 It is easily seen from the definition that $D^k_\mu(x,r)$ is controlled above and below by $D^k_\mu(x,r/2)$ and $D^k_\mu(x,2r)$, as long as $\mu$ is not too small on $\B r x$.  In particular, if $\mu(B_{r}(x))\geq 2^k\epsilon_n r^k$, then for all $y\in S\cap \B r x$,  $D^k_\mu(x,r)\leq 2^{k+2} D^k_\mu(y,2r)$. As a corollary we have the estimate
 \begin{align}\label{eq_estDint}
  \mu(B_{r}(x))\geq 2^k\epsilon_n r^k \quad \Longrightarrow \quad &D^k_\mu(x,r)\leq 2^{k+2} \fint_{\B r x } D^k_\mu(y,2r)d\mu (y)\, .
 \end{align}

\end{remark}

\vspace{.5cm}

%
%
%
%
%
%
%
%

The first of the main results of this part of the paper is not strictly used elsewhere in the paper, but it is a natural generalization of the Reifenberg and gives intuition and motivation for the rest of the statements, which are essentially more complicated versions of it.  Thus, we have decided it is worth discussing independently.  The first theorem of this section is the following:\\

\begin{theorem}[$W^{1,p}$-Reifenberg]\label{t:reifenberg_W1p}
For each $\epsilon>0$ and $p\in [1,\infty)$ there exists $\delta(n,\epsilon,p)>0$ such that the following holds.  Let $S\subseteq B_2\subseteq \dR^n$ be a closed subset with $0^n\in S$, and assume for each $x\in S\cap B_1$ and $B_r(x)\subseteq B_2$ that
\begin{align}
&\inf_{L^k} d_H\big( S\cap B_r(x),L^k\cap B_r(x)\big)<\delta r \, , \label{e:reifenberg_L_infty} \\
&\int_{S\cap B_r(x)}\,\ton{\int_0^r D^k_S(x,s)\,\frac{ds}{s}}\, d\lambda^k(x)<\delta^2r^{k}\, . \label{e:reifenberg_excess_L2}
\end{align}
Then the following hold:
\begin{enumerate}
\item there exists a mapping $\phi:\R^k\to \R^n$ which is a $1+\epsilon$ bi-$W^{1,p}$ map onto its image and such that $S\cap \B 1 {0^n} = \phi(B_1(0^k))\cap \B 1 {0^n}$. In particular
\begin{gather}
 \fint_{\B {1}{0^k}}\abs{\nabla \phi}^p \leq 1+\epsilon\, , \quad \fint_{S\cap \B 1 {0^n}} \abs{\nabla \phi^{-1}}^{p}\leq 1+\epsilon\, .
\end{gather}
\item $S\cap B_1(0^n)$ is rectifiable.
\item For each ball $B_r(x)\subseteq B_1$ with $x\in S$ we have 
\begin{gather}\label{eq_lambda_lower_upper}
(1-\epsilon)\omega_k r^k\leq \lambda^k(S\cap B_r(x))\leq (1+\epsilon)\omega_k r^k\, .
\end{gather}
\end{enumerate}
\end{theorem}
\begin{remark}
The result $(2)$ follows from $(1)$. Indeed by a standard result, if $\phi$ is a $W^{1,p}$ map with $p>k$, then its image is a $k$-rectifiable set (see Lemma \ref{lemma_w1p_rec}).
We get $(3)$ by applying the result of $(1)$ to all smaller balls $B_r(x)\subseteq B_1$, since the assumptions of the theorem hold on these balls as well.  

Note that the $W^{1,p}$ estimates for $\phi^{-1}$ are justified in Lemma \ref{lemma_w1p_inv}.
\end{remark}
\begin{remark}
Note that for $p>k$, by Sobolev embeddings a $W^{1,p}$ map is also a $C^\alpha$ map with $\alpha = 1-\frac k p$.
\end{remark}

\begin{remark}
 As it is easily seen, the requirement that $S$ is closed is essential for this theorem, and in particular for the lower bound on the Hausdorff measure. As an example, consider any set $S\subseteq \R^k$ which is dense but has zero Hausdorff measure. In the following theorems, we will not be concerned with lower bounds on the measure, and we will be able to drop the closed assumption.
\end{remark}

\vspace{.5cm}

Let us now consider the case when we drop the assumption \eqref{e:reifenberg_L_infty} from the result.  The key distinction now is that $S$ only needs to be locally near a piece of a $k$-dimensional subspace, but not conversely.  Thus, we cannot hope to obtain topological information about the set $S$, see Section \ref{sss:reifenberg}.  The precise statement is the following:\\

\begin{theorem}[Rectifiable-Reifenberg]\label{t:reifenberg_W1p_holes}
For every $\epsilon>0$, there exists $\delta(n,\epsilon)>0$ such that the following holds.  Let $S\subseteq B_2\subseteq \dR^n$ be a $\lambda^k$-measurable subset, and assume for each $B_r(x)\subseteq B_2$  with $\lambda^k(S\cap B_r(x))\geq \epsilon_n r^k$ that
\begin{align}\label{e:reifenberg_displacement_L2}
\int_{S\cap B_r(x)}\,\ton{\int_0^r D^k_S(x,s)\,\frac{ds}{s}}\, d\lambda^k(x)<&\delta^2r^{k}\, . 
\end{align}
Then the following hold:
\begin{enumerate}
\item For each ball $B_r(x)\subseteq B_1$ with $x\in S$ we have 
\begin{gather}\label{eq_lambda_obj}
\lambda^k(S\cap B_r(x))\leq (1+\epsilon)\omega_k r^k \, .
\end{gather}
\item $S\cap B_1(0^n)$ is $k$-rectifiable.
\end{enumerate}
\end{theorem}
\begin{remark}
Notice that for the statement of the theorem we do not need control over balls which already have small measure.  This will be quite convenient for the applications.
\end{remark}
\begin{remark}
Instead of \eqref{e:reifenberg_displacement_L2} we may assume the essentially equivalent estimate
\begin{align}\label{eq_sum2^alpha}
\sum_{r_\alpha\leq 4r}\int_{S\cap B_r(x)}D^k_S(x,r_\alpha)\, d\lambda^k(x)<\delta^2 r^k\, .
\end{align}
In the applications, this will be the more convenient phrasing.
\end{remark}
\vspace{1cm}

Finally, we end by stating a version of the above theorem which is more discrete in nature. This result will be particularly important in the proof of the main theorems of this paper:\\

\begin{theorem}[Discrete Rectifiable-Reifenberg]\label{t:reifenberg_W1p_discrete}
There exists $\delta(n)>0$ and $D(n)$ such that the following holds.  Let $\{B_{r_j}(x_j)\}_{x_j\in S}\subseteq B_2$ be a collection of disjoint balls, and let $\mu\equiv \sum_{x_j\in S}\omega_k r^k_j \delta_{x_j}$ be the associated measure.  Assume that for each $B_r(x)\subseteq B_2$ with $\mu(\B{r}{x})\geq \epsilon_n r^k$ we have
\begin{align}\label{e:reifenberg_L_2_discrete}
&\int_{B_r(x)}\ton{\int_0^r D^k_\mu(x,s)\,{\frac{ds}{s}}}\, d\mu(x)<\delta^2 r^{k}\, .
\end{align}
Then we have the estimate 
\begin{align}
\sum_{x_j\in B_1} r_j^k<D(n)\, .
\end{align}
\end{theorem}
\begin{remark}
As in Theorem \ref{t:reifenberg_W1p_holes}, instead of \eqref{e:reifenberg_displacement_L2} we may assume the estimate
\begin{align}
\sum_{r_\alpha\leq 2r}\int_{B_r(x)}D^k_\mu(x,r_\alpha)\, d\mu(x)<\delta^2 r^k\, .
\end{align}
In the applications, this will be the more convenient phrasing.
\end{remark}

For many of the applications of this paper, it is this version of the Reifenberg which will be most important.  The reasoning is that to keep uniform control on all estimates our inductive covering will need to cover at all scales.  It is only at the last scale that we begin to cover the singular sets $S^k$ on sets of positive measure.

\vspace{1cm}

\section{Technical Constructions toward New Reifenberg Results}\label{ss:best_comparison}
In this section, we prove some technical lemmas needed for dealing with the relation between best $L^2$ subspaces.  These elementary results will be used in many of the estimates of subsequent sections.
\subsection{Hausdorff distance and subspaces}
We start by recalling some standard facts about affine subspaces in $\R^n$ and Hausdorff distance.

\begin{definition}
 Given two linear subspaces $L,V\subseteq \R^n$, we define the Grassmannian distance between these two as
 \begin{gather}
  d_G(L,V)= d_H(L\cap \B 1 0, V\cap \B 1 0 )=d_H\ton{L\cap \overline{\B 1 0}, V\cap \overline{\B 1 0} }\, ,
 \end{gather}
Note that if $\dim(L)\neq \dim(V)$, then $d_G(L,V)=1$.
\end{definition}
\vspace{.3 cm}

For general subsets in $\R^n$, it is evident that $A\subseteq \B{\delta}{B}$ does not imply $B\subseteq \B{c\delta}{A}$. However, if $A$ and $B$ are affine spaces with the same dimension, then it is not difficult to see that this property holds.  More precisely:

 \begin{lemma}\label{lemma_hdv}
 Let $V,\, W$ be two $k$-dimensional affine subspaces in $\R^n$, and suppose that $V\cap \B {1/2}{0}\neq \emptyset$. There exists a constant $c(k,n)$ such that if $V\cap \B 1 0\subseteq \B{\delta}{W\cap \B 1 0}$, then $W\cap \B 1 0\subseteq \B{c\delta}{V\cap \B 1 0}$. Thus in particular $d_H(V\cap \B 1 0,W\cap \B 1 0)\leq c\delta$.
\end{lemma}
\begin{proof}
 The proof relies on the fact that $V$ and $W$ have the same dimension. Let $x_0\in V$ be the point of minimal distance from the origin. By assumption, we have that $\norm{x_0}\leq 1/2$. Let $x_1,\cdots,x_k\in V\cap \overline{\B 1 0}$ be a sequence of points such that
 \begin{gather}
  \norm{x_i-x_0}\geq 1/2\, \quad \text{ and for }\, i\neq j\, ,  \quad \ps{x_i-x_0}{x_j-x_0}=0\, .
 \end{gather}
In other words, $\cur{x_i-x_0}_{i=1}^k$ is an affine base for $V$. Let $\cur{y_i}_{i=0}^k\subseteq W\cap \overline{\B 1 0}$ be such that $d(x_i,y_i)\leq \delta$. Then
 \begin{gather}
  \norm{y_i-y_0}\geq 1/2-2\delta\, \quad \text{ and for }\, i\neq j\, ,  \quad \abs{\ps{y_i-y_0}{y_j-y_0}}\leq 2\delta+4\delta^2\, .
 \end{gather}
This implies that for $\delta\leq\delta_0(n)$, $\cur{y_i-y_0}_{i=1}^k$ is an affine base for $W$ and for all $y\in W$
\begin{gather}
 y=y_0+ \sum_{i=1}^k \alpha_i (y_i-y_0)\, , \quad \abs{\alpha_i}\leq 10 \norm{y-y_0}\, .
\end{gather}
Now let $y\in W\cap \overline{\B 1 0}$ be the point of maximum distance from $V$, and let $\pi$ be the projection onto $V$ and $\pi^\perp$ the projection onto $V^\perp$. Then
\begin{gather}
 d(y,V)= d(y,\pi(y))=\norm{\pi^\perp (y)}\leq \sum_{i=1}^k \abs{\alpha_i}\norm{\pi^\perp (y_i-y_0)}\leq c'(n,k)\delta\, .
\end{gather}
Since $y\in \overline{\B 1 0}$, then $\pi(y)\in V \cap \B{1+c'\delta}{0}$, and thus $d(y,V\cap B_1(0))\leq 2c'\delta\equiv c\delta$, which proves the claim. 
 \end{proof}
\vspace{.5 cm}
Next we will see that the Grassmannian distance between two subspaces is enough to control the projections with respect to these planes. In order to do so, we recall a standard estimate.
\begin{lemma}
 Let $V,W$ be linear subspaces of a Hilbert space. Then $d_G(V,W)=d_G\ton{V^\perp,W^\perp}$.
\end{lemma}
\begin{proof}
We will prove that $d_G\ton{V^\perp,W^\perp}\leq d_G\ton{V,W}$. By symmetry, this is sufficient. 

Take $x\in V^\perp$ such that $\norm{x}=1$, and consider that $d(x,W^\perp)=\norm{\pi_W(x)}$. Let $z=\pi_W(x)$ and $y=\pi_V(z)$. We want to show that if $d_G(V,W)\leq \epsilon<1$, then $\norm{z}\leq \epsilon$. We can limit our study to the space spanned by $x,y,z$, and assume wlog that $x=(1,0,0)$, $y=(0,b,0)$ and $z=(a,b,c)$. By orthogonality between $z$ and $z-x$, we have
\begin{gather}
 a^2+b^2+c^2 +(1-a)^2+b^2+c^2=1 \, \quad \Longrightarrow \quad a=a^2+b^2+c^2 \, ,
\end{gather}
and since $z\in W$, we also have $\norm{z-y}\leq \epsilon \norm{z}$, which implies
\begin{gather}
 a^2+c^2\leq \epsilon^2 \ton{a^2+b^2+c^2} \, \quad \Longrightarrow \quad a^2+c^2 \leq \frac{\epsilon^2}{1-\epsilon^2} b^2\, .
\end{gather}
Since the function $f(x)=x^2/(1-x^2)$ is monotone increasing for $x\geq 0$, we can define $\alpha\geq 0$ in such a way that
\begin{gather}
 a^2+c^2 = \frac{\alpha^2}{1-\alpha^2} b^2\, , \quad a=a^2+b^2+c^2 = \frac{1}{1-\alpha^2} b^2\, .
\end{gather}
Note that necessarily we will have $\alpha\leq \epsilon$. Now we have
\begin{gather}
 \frac{1}{(1-\alpha^2)^2} b^4 =a^2\leq \frac{\alpha^2}{1-\alpha^2} b^2 \quad \Longrightarrow \quad b^2\leq \alpha^2\ton{1-\alpha^2} \quad \Longrightarrow \quad \norm z^2 = a^2+b^2+c^2\leq \alpha^2\leq \epsilon^2\, .
\end{gather}
This proves that $V^\perp\cap\B 1 0 \subset \B{\epsilon}{W^\perp}$. In a similar way, one proves the opposite direction.
\end{proof}

As a corollary, we prove that the Grassmannian distance $d_G(V,W)$ is equivalent to the distance given by $\norm{\pi_V-\pi_W}$.
\begin{lemma}\label{lemma_epsproj}
Let $V,W$ be linear subspaces of $\R^n$. Then for every $x\in \R^n$, 
 \begin{gather}
  \norm{\pi_V(x)-\pi_W(x)}\leq 2d_G(V,W)\norm x\, .
 \end{gather}
In particular, if $x\in W^\perp$, then $\norm{\pi_V(x)}\leq 2d_G(V,W)\norm x$.

Conversely, we have
\begin{gather}
 d_G(V,W)\leq \sup_{x\in \R^n\setminus \{0\}} \cur{\frac{\norm{\pi_V(x)-\pi_W(x)}}{\norm x}}\, .
\end{gather}
\end{lemma}
 \begin{proof}
The proof is just a corollary of the previous lemma. Assume wlog that $\norm{x}=1$, and let $x=y+z$ where $y=\pi_V(x)$ and $z=\pi_{V^\perp}(x)$. Then
\begin{gather}
 \norm{\pi_V(x)-\pi_W(x)}=\norm{y-\pi_W(y)-\pi_W(z)} \leq \norm{y-\pi_W(y)}+\norm{z-\pi_{W^\perp}(z)} = d(y,W)+d(z,W^\perp)\, .
\end{gather}
Since $\norm y^2 + \norm z^2 =\norm x^2= 1$, by the previous lemma we get the first estimate.

The reverse estimate is an immediate consequence of the definition. 
\end{proof}
\vspace{.5 cm}

\subsection{\texorpdfstring{Distance between $L^2$ best planes}{Distance between L2 best planes}}
Here we study the distance between best approximating subspaces for our measure $\mu$ on different balls. Let us begin by fixing our notation for this subsection, and pointing out the interdependencies of the constants chosen here.  Throughout this subsection, our choice of scale $\rho=\rho(n,M)>0$ is a constant which will eventually be fixed according to Lemma \ref{lemma_alpharho}.  For applications to future sections, it is sufficient to know that we can take $\rho(n,M)= 10^{-10} (100n)^{-n} M^{-1}$. We also point out that in Section \ref{sec_proof_main}, we will fix $M=40^n$, and so $\rho$ will be a constant depending only on $n$. In particular, we can use the very coarse estimate
\begin{gather}\label{eq_rho_rough}
\rho=10^{-10}(100n)^{-3n}\, .
\end{gather}
We will also introduce a threshold value $\gamma_k= \omega_k 40^{-k}$.  The dimensional constant $\gamma_k$ is chosen simply to be much smaller than any covering errors which will appear.

We will consider a positive Radon measure $\mu$ supported on $S\subseteq \B 1 0$, and use $D(x,r)\equiv D^{k}_{\mu}(x,r)$ to bound the distances between best $L^2$ planes at different points and scales. By definition let us denote by $V(x,r)$ a best $k$-dimensional plane on $\B r x $, i.e., a $k$-dimensional affine subspaces minimizing $\int_{\B r x} d(x,V)^2 d\mu$. Note that, in general, this subspace may not be unique. We want to prove that, under reasonable hypothesis, the distance between $V(x,r)$ and $V(y,r')$ is small if $d(x,y)\sim r$ and $r'\sim r$.\\

In order to achieve this, we will need to understand some minimal properties of $\mu$.  First, we need to understand how concentrated $\mu$ is on any given ball. For this reason, for some $\rho>0$ and all $x\in \B 1 0$ we will want to consider the upper mass bound 

\begin{gather}\label{eq_rho}
 \mu(\B \rho x )\leq M \rho^k\,\,\, \forall x\in B_1(0)\, .
\end{gather}

However, an upper bound on the measure is not enough to guarantee best $L^2$-planes are close, as the following example shows:
\begin{example}
 Let $V,V'$ be $k$-dimensional subspaces, $0\in V\cap V'$, and set $S=\ton{V\cap \B 1 0 \setminus \B{1/10}{0}}\cup S'$, where $S'\subseteq V'\cap \B {1/10}{0}$ and $\mu=\lambda^k|_S$. Then evidently $D(0,1)\leq \lambda^k(S')$ and $D(0,1/10)=0$, independently of $V$ and $V'$. However, $V(0,1)$ will be close to $V$, while $V(0,1/10)=V'$. Thus, in general, we cannot expect $V(0,1)$ and $V(0,1/10)$ to be close if $\mu(\B {1/10}{0})$ is too small.
\end{example}

Thus, in order to prove that the best planes are close, we need to have some definite amount of measure on the set, in such a way that $S$ ``effectively spans'' a $k$-dimensional subspace, where by effectively span we mean the following:
\begin{definition}
 Given a sequence of points $p_i\in \R^n$, we say that $\cur{p_i}_{i=0}^k$ $\alpha-$effectively span a $k$-dimensional affine subspace if for all $i=1,\cdots,k$
 \begin{gather}
  \norm{p_i-p_0}\leq \alpha^{-1}\, , \quad p_i\not \in \B{\alpha}{p_0+\operatorname{span}\cur{p_1-p_0,\cdots,p_{i-1}-p_0}}\, .
 \end{gather}
\end{definition}
 This implies that the vectors $p_i-p_0$ are linearly independent in a quantitative way. In particular, we obtain immediately that
\begin{lemma}\label{lemma_effspa}
 If $\cur{p_i}_{i=0}^k$ $\alpha$-effectively span the $k$-dimensional affine subspace $$V=p_0+\operatorname{span}\cur{p_1-p_0,\cdots,p_k-p_0}\, ,$$ then for all $x\in V$ there exists a \textit{unique} set $\cur{\alpha_i}_{i=1}^k$ such that
\begin{gather}
 x=p_0+\sum_{i=0}^k \alpha_i (p_i-p_0)\, , \quad \abs{\alpha_i}\leq c(n,\alpha)\norm{x-p_0}\, .
\end{gather}
\end{lemma}
\begin{proof}
 The proof is quite straightforward. Since $\cur{p_i-p_0}_{i=1}^k$ are linearly independent, we can apply the Gram-Schmidt orthonormalization process to obtain an orthonormal basis $e_1,\cdots,e_k$ for the linear space $\operatorname{span} \cur{p_i-p_0}_{i=1}^k$. By induction, it is easy to check that for all $i$
 \begin{gather}
  e_i = \sum_{j=1}^i \alpha'_j (p_j-p_0)\, , \quad \abs{\alpha'_j}\leq c(n,\alpha)\, .
 \end{gather}
Now the estimate follows from the fact that for all $x\in V$
\begin{gather}
 x=p_0+\sum_{i=1}^k \ps{x-p_0}{e_i} e_i\, .
\end{gather}

\end{proof}

With these definitions, we are ready to prove that in case $\mu$ is not too small, then its support must effectively span something $k$-dimensional.

\begin{lemma}\label{lemma_alpharho}
Let $\gamma_k=\omega_k 40^{-k}$. There exists a $\rho_0(n,\gamma_k,M)=\rho_0(n,M)$ such that if \eqref{eq_rho} holds for some $\rho\leq \rho_0$ and if $\mu(\B 1 0)\geq \gamma_k$, then for every affine subspace $V\subseteq \R^n$ of dimension $\leq k-1$, there exists an $x\in S$ such that $\B{\rho}{x}\subseteq \B 1 0$, $\B{10\rho}{x}\cap V = \emptyset$ and $\mu\ton{\B{\rho}{x}}\geq c(n,\rho) = c(n)\rho^n>0$. 
 \end{lemma}
\begin{proof}
Let $V$ be any $k-1$-dimensional subspace, and consider the set $\B {11\rho}{V}$. Let $B_i=\B{\rho}{x_i}$ be a sequence of balls that cover the set $\B{11\rho}{V}\cap \B 1 0$ and such that $B_i/2\equiv \B {\rho /2}{x_i}$ are disjoint. If $N$ is the number of these balls, then a standard covering argument gives
\begin{align}
 &N\omega_n \rho^n/2^n\leq \omega_{k-1} (1+\rho)^{k-1} \omega_{n-k+1} (12\rho)^{n-k+1}\leq 24^n \omega_{k-1} \omega_{n-k+1} \rho^{n-k+1}\, \quad \notag\\
 \Longrightarrow \quad &N\leq 48^{n} \frac{\omega_{k-1} \omega_{n-k+1}}{\omega_n} \rho^{1-k}\, .
\end{align}
By \eqref{eq_rho}, the measure of the set $\B{11\rho}{V}$ is bounded by
\begin{gather}
 \mu(\B{11\rho}{V})\leq \sum_i \mu\ton{B_i}\leq MN \omega_k \rho^k = 48^n \frac{\omega_k \omega_{k-1} \omega_{n-k+1}}{\omega_n} M \rho \leq 10^{5} (50n)^n M \rho =  c(n)M\rho .
\end{gather}
Thus if 
\begin{gather}
 \rho\leq 10^{-5}(50n)^{-n}\gamma_k/(4M)\, , 
\end{gather}
then $\mu(\B{11\rho}{V})\leq \gamma_k/4$.  In particular, we get that there must be some point of $S$ not in $\B{11\rho}{V}$.  More effectively, let us consider the set $\B 1 0 \setminus \B {11\rho}{V}$. This set can be covered by at most $c(n,\rho)=4^{n}\rho^{-n}$ balls of radius $\rho$, and we also see that
\begin{gather}
 \mu\ton{\B {1}{0}\setminus \B{11\rho}{V}}\geq \frac{3\gamma_k}{4}\, .
\end{gather}
Thus, there must exist at least one ball of radius $\rho$ which is disjoint from $\B{10\rho}{V}$ and such that
\begin{gather}
 \mu\ton{\B{\rho}{x}}\geq \frac{3\gamma_k}{4} 4^{-n}\rho^n\geq c(n)\rho^n\, .
\end{gather}
\end{proof}
\vspace{.5 cm}

%
%

Now if at two consecutive scales there are some balls on which the measure $\mu$ {\it effectively spans} $k$-dimensional subspaces, we show that these subspaces have to be close together.

\begin{lemma}\label{lemma_vw}
 Let $\mu$ be a positive Radon measure and assume $\mu(\B 10 )\geq \gamma_k $ and that for each $y\in \B 1 0$ we have $\mu(\B {\rho^2} y )\leq M \rho^{2k}$, where $\rho\leq \rho_0$.  Additionally, let $B_\rho(x)\subset \B 1 0$ be a ball such that $\mu(\B \rho x )\geq \gamma_k \rho^k$.  Then if $A=V(0,1)\cap \B \rho x$ and $B=V(x,\rho)\cap \B \rho x$ are $L^2$-best subspace approximations of $\mu$ with $d(x,A)<\rho/2$, then 
 \begin{gather}\label{eq_distD}
  d_H(A,B)^2 \leq c(n,\rho,M) \ton{D^k_\mu(x,\rho)+D^k_\mu(0,1)}\, .
 \end{gather}
\end{lemma}
\begin{proof}
Let us begin by observing that if $c(n,\rho,M)>2\delta^{-1}(n,\rho,M)$, which will be chosen later, then we may assume without loss of generality that 
\begin{align}\label{e:lemma_vw:1}
D^k_\mu(x,\rho)+D^k_\mu(0,1)\leq \delta = \delta(n,\rho,M)\, ,
\end{align}
since otherwise \eqref{eq_distD} is trivially satisfied. Moreover, note that $\gamma_k >> \epsilon_n$, so equation \eqref{eq_deph_D} is valid on $\B \rho x $ and on $\B 1 0$. \\

We will estimate the distance $d_H(A,B)$ by finding $k+1$ balls $\B{\rho^2}{y_i}$ which have enough mass and effectively span in the appropriate sense $V(x,\rho)$. Given the upper bounds on $D^k_\mu$, we will then be in a position to prove our estimate.

\vspace{5mm}

Consider any $\B {\rho^2} {y}\subseteq \B 1 0$ and let $p(y)\in \B{\rho^2}{y}$ be the center of mass of $\mu$ restricted to this ball.
Let also $\pi(p)$ be the orthogonal projection of $p$ onto $V(x,\rho)$.  By Jensen's inequality:
\begin{gather}\label{eq_vw1}
 d(p(y),V(x,\rho))^2=d(p(y),\pi(p(y)))^2=d\ton{\fint_{B_{\rho^2}(y)}z\,d\mu(z),V(x,\rho)}^2 \leq \frac{1}{\mu(\B{\rho^2}{y})}\int_{\B {\rho^2}{y}} d(z,V(x,\rho))^2 d\mu(z)\, . 
\end{gather}

Using this estimate and Lemma \ref{lemma_alpharho}, we want to prove that there exists a sequence of $k+1$ balls $\B{\rho^2}{y_i}\subseteq \B{\rho}{x}$ such that
\begin{enumerate}
\def\theenumi{\roman{enumi}}
 \item\label{it_1} $\mu\ton{\B{\rho^2}{y_i}}\geq c(n,\rho,M)>0 $
 \item\label{it_2} $\cur{\pi(p(y_i))}_{i=0}^k\equiv\cur{\pi_i}_{i=0}^k$ effectively spans $V(x,\rho)$. In other words for all $i=1,\cdots,k$, $\pi_i\in V(x,\rho)$ and 
 \begin{gather}\label{eq_vw2}
\pi_i\not \in \B{5\rho^2}{\pi_0+\operatorname{span}\ton{\pi_1-\pi_0,\cdots,\pi_{i-1}-\pi_0}}\, .
 \end{gather}
\end{enumerate}
We prove this statement by induction on $i=0,\cdots,k$. For $i=0$, the statement is trivially true since $\mu(\B{\rho}{x})\geq \gamma_k \rho^k$. In order to find $y_{i+1}$, consider the $i$-dimensional affine subspace $V^{(i)} = \operatorname{span}\ton{\pi_0,\cdots,\pi_{i-1}}\leq V(x,\rho)$. By Lemma \ref{lemma_alpharho} applied to the ball $\B{\rho}{x}$, there exists some $\B{\rho^2}{y_{i+1}}$ such that $\mu\ton{\B{\rho^2}{y_{i+1}}}\geq c(n,\rho,M)>0 $ and 
 \begin{gather}
 y_{i+1}\not \in \B{10\rho^2}{\operatorname{span}\ton{\pi_0,\cdots,\pi_{i-1}}}\, .
 \end{gather}
By definition of center of mass, it is clear that $d(y_i,p(y_i))\leq \rho^2$. Moreover, by item \eqref{it_1} and equation \eqref{eq_vw1}, we get
\begin{gather}
 d(p(y_{i+1}),V(x,\rho))^2\leq c \int_{\B {\rho^2}{y}} d(z,V(x,\rho))^2 d\mu(z) \leq c D^k_\mu(x,\rho)\leq c\delta\, .
\end{gather}
Thus by the triangle inequality we have $d(y_i,\pi_i)\leq 2\rho^2$ if $\delta\leq \delta_0(n,\rho,M)$ is small enough. This implies \eqref{eq_vw2}.
Using similar estimates, we also prove $d(p(y_i),V(0,1))^2\leq c' D^k_\mu(0,1)$ for all $i=0,\cdots,k$. Thus by the triangle inequality
\begin{gather}
 d(\pi_i,V(0,1))\leq d(\pi_i,p(y_i))+d(p(y_i),V(0,1))\leq c(n,\rho,M)\ton{D^k_\mu(x,\rho)+D^k_\mu(0,1)}^{1/2}\, .
\end{gather}

\vspace{5mm}

Now consider any $y\in V(x,\rho)$. By item \eqref{it_2} and Lemma \ref{lemma_effspa}, there exists a unique set $\cur{\beta_i}_{i=1}^k$ such that
\begin{gather}
 y=\pi_0+\sum_{i=0}^k \beta_i (\pi_i-\pi_0)\, , \quad \abs{\beta_i}\leq c(n,\rho)\norm{y-\pi_0}\, .
\end{gather}
Hence for all $y\in V(x,\rho)\cap \B{\rho}{x_i}$, we have
\begin{gather}
 d(y,V(0,1))\leq d(\pi_0,V(0,1)) + \sum_i \abs{\beta_i} d(\pi_i-\pi_0,V(0,1))\leq c(n,\rho,M)\ton{D^k_\mu(x,\rho)+D^k_\mu(0,1)}^{1/2}\, .
\end{gather}
By Lemma \ref{lemma_hdv}, this completes the proof of \eqref{eq_distD}. 
\end{proof}
\vspace{.5 cm}

\subsection{\texorpdfstring{Comparison between $L^2$ and $L^\infty$ planes}{Comparison between L2 and L-infinity planes}} Given $\B r x $, we denote as before by $V(x,r)$ one of the $k$-dimensional subspace minimizing $\int_{\B r x } d(y,V)^2 d\mu$. Suppose that the support of $\mu$ satisfies a uniform one-sided Reifenberg condition, i.e. suppose that there exists a $k$-dimensional plane $L(x,r)$ such that $x\in L(x,r)$ and
\begin{gather}\label{eq__}
\supp{\mu} \cap \B{r}x \subseteq \B{\delta r}{L(x,r)}\, . 
\end{gather}
Then, by the same technique used in Lemma \ref{lemma_vw}, we can prove that
\begin{lemma}\label{lemma_LV}
  Let $\mu$ be a positive Radon measure with $\mu\ton{\B 1 0}\geq \gamma_k$ and such that for all $\B \rho y \subseteq \B 1 0$ we have $\mu(\B \rho y )\leq M \rho^k$ and \eqref{eq__}. Then
 \begin{gather}
  d_H(L(0,1)\cap \B 1 0, V(0,1)\cap \B 1 0)^2 \leq c(n,\rho,M)\ton{\delta^2+D^k_\mu(0,1)}\, .
 \end{gather}
\end{lemma}
\vspace{.5 cm}

\subsection{bi-Lipschitz equivalences}
In this subsection, we study a particular class of maps with nice local properties. These maps are a slightly modified version of the maps which are usually exploited to prove Reifenberg's theorem, see for example \cite{reif_orig,Toro_reif,davidtoro}, \cite[section 10.5]{morrey} or \cite{simon_reif}. The estimates in this section are standard in literature.

We start by defining the functions $\sigma$. For some $r>0$, let $\cur{x_i}$ be an $r$-separated subset of $\R^n$, i.e., 
\begin{enumerate}
\def\theenumi{\roman{enumi}}
 \item \label{item_1}$d(x_i,x_j)\geq r$.
\end{enumerate}
Let also $p_i$ be a points in $\R^n$ with
\begin{enumerate}
\def\theenumi{\roman{enumi}}\setcounter{enumi}{1}
 \item $p_i\in \B{10r}{x_i}$
\end{enumerate}
and let $V_i$ be a sequence of $k$-dimensional linear subspaces.

By standard theory, it is easy to find a locally finite smooth partition of unity $\lambda_i:\dR^n\to [0,1]$ such that
\begin{enumerate}\label{deph_sigma_0}
\def\theenumi{\roman{enumi}}\setcounter{enumi}{2}
 \item \label{item_3r}$\supp{\lambda_i}\subseteq \B {3r}{x_i}$ for all $i$,
 \item for all $x\in \bigcup_{i} \B{2 r}{x_i}$, $\sum_i \lambda_i(x)=1$ ,
 \item $\sup_i \norm{\nabla \lambda_i}_\infty \leq c(n)/r$ ,
 \item \label{item_last} if we set $1-\psi(x)=\sum_i \lambda_i(x)$, then $\psi$ is a smooth function with $\norm{\nabla \psi}_\infty \leq c(n)/r$ .
\end{enumerate}
Note that by \eqref{item_3r}, and since $x_i$ is $r$-separated, there exists a constant $c(n)$ such that for all $x$, $\lambda_i(x)>0$ for at most $c(n)$ different indexes.

For convenience of notation, set $\pi_V(v)$ to be the orthogonal projection onto the linear subspace $V$ of the free vector $v$, and set
\begin{gather}
\pi_{p_i,V_i}(x)=p_i+\pi_{V_i}(x-p_i)\, .
\end{gather}
In other words, $\pi_{p_i,V_i}$ is the affine projection onto the affine subspace $p_i+V_i$. Recall that $\pi_{V_i}$ is a linear map, and so the gradients of $\pi_{V_i}$ and of $\pi_{p_i,V_i}$ at every point are equal to $\pi_{V_i}$. 

\begin{definition}\label{deph_sigma}
 Given $\cur{x_i,p_i,\lambda_i}$ satisfying \eqref{item_1} to \eqref{item_last}, and given a family of linear $k$-dimensional spaces $V_i$, we define a smooth function $\sigma:\R^n\to \R^n$ by
 \begin{gather}  
  \sigma(x)= x+\sum_i \lambda_i(x) \pi_{V_i^\perp}\ton{p_i-x} = \psi(x) x +\sum_i \lambda_{i}(x)\pi_{p_i,V_i}\ton{x}\, .
 \end{gather}
\end{definition}
By local finiteness, it is evident that $\sigma$ is smooth. Moreover, if $\psi(x)=1$, then $\sigma(x)=x$. It is clear that philosophically $\sigma$ is a form of ``smooth interpolation'' between the identity and the projections onto the subspaces $V_i$. It stands to reason that if $V_i$ are all close together, then this map $\sigma$ is close to being an orthogonal projection in the region $ \bigcup_{i} \B{2 r}{x_i}$. 

\begin{lemma}\label{lemma_sigma_simon}
 Suppose that there exists a $k$-dimensional linear subspace $V\subseteq \R^n$ and a point $p\in \R^n$ such that for all $i$
 \begin{gather}\label{eq_sigmadelta}
  d_G(V_i,V)\leq \delta\, , \quad d(p_i,p+V)\leq \delta\, .
 \end{gather}
Then the map $\sigma$ restricted to the set $U=\psi^{-1}(0)=\ton{\sum_i \lambda_i}^{-1}(1)$ can be written as
\begin{gather}
 \sigma(x)=\pi_{p,V}(x)+e(x)\, ,
\end{gather}
and $e(x)$ is a smooth function with 
\begin{gather}
 \norm{e}_\infty + \norm{\nabla e}_\infty\leq c(n)\delta/r=c(n,r)\delta\, .
\end{gather}
\end{lemma}
\begin{remark}
Thus, on $U$ we have that $\sigma$ is the affine projection onto $V$ plus an error which is small in $C^1$.
\end{remark}

\begin{proof}
 On the set $U$, we can define 
 \begin{gather}
  e(x)=\sigma(x)-\pi_{p,V}(x)=-\pi_{p,V}(x)+ \sum_i \lambda_i(x) \cdot \ton{\pi_{p_i,V_i}(x)}=\notag\\
  =\sum_i \lambda_i(x) \cdot \ton{p_i-p-\pi_V(p_i-p) +\pi_V(p_i) - \pi_{V_i}(p_i) + \pi_{V_i}(x)  -\pi_V(x)}\, .
 \end{gather}
 By \eqref{eq_sigmadelta} and Lemma \ref{lemma_epsproj}, we have the estimates
 \begin{gather}\label{eq_222}
  \norm{p_i-p-\pi_V(p_i-p)}<\delta\, , \quad \norm{\pi_V(x-p_i)-\pi_{V_i}(x-p_i)}\leq c(n)\delta \norm{x-p_i}\leq 13 c(n) r \delta\, . 
 \end{gather}
 This implies
 \begin{gather}
  \norm{e}_{L^\infty(U)} \leq c(n)(1+13 r)\delta \leq c(n)\delta\, .
 \end{gather}

As for $\nabla e$, we have
\begin{gather}
 \nabla e = \sum_i \nabla \lambda_i(x) \cdot \ton{p_i-p-\pi_V(p_i-p) +\pi_V(p_i) - \pi_{V_i}(p_i) + \pi_{V_i}(x)  -\pi_V(x)} + \sum_i \lambda_i(x) \nabla \ton{\pi_{V_i}(x)-\pi_V(x)}\, .
\end{gather}
The first sum is easily estimated, and since $\ps{\nabla (\pi_{W})|_x }{w}=\pi_{W}(w)$, we can still apply Lemma \ref{lemma_epsproj} and conclude:
\begin{gather}
 \norm{\nabla e}_{L^\infty(U)} \leq \frac{c(n)}{r} \delta\, .
\end{gather}
\end{proof}
\vspace{.5 cm}
As we have seen, $\sigma$ is in some sense close to the affine projection to $p+V$. In the next lemma, which is similar in spirit to \cite[squash lemma]{simon_reif}, we prove that the image through $\sigma$ of a graph over $V$ is again a graph over $V$ with nice bounds. 
\begin{lemma}[squash lemma]\label{lemma_squash}
 Fix $\rho\leq 1$ and some $\B{r/\rho}{y}\subseteq \R^n$, let $I=\cur{x_i}\cap \B{5r/\rho}{y}$ be an $r$-separated set and define $\sigma$ as in Definition \ref{deph_sigma}. Suppose that there exists a $k$-dimensional subspace $V$ and some $p\in \R^n$ such that $d(y,p+V)\leq \delta r$ and for all $i$:
 \begin{gather}\label{eq_111}
  d(p_i,p+V)\leq \delta r\, \quad \text{ and } \quad d_G(V_i,V)\leq \delta\, .
 \end{gather}
 Suppose also that there exists a $C^1$ function $g:V\to V^\perp$ such that $G\subseteq \R^n$ is the graph $$G=\cur{p+x+g(x)\, \ \ \text{for } \ \ x\in V}\cap \B{r/\rho}{y}\, ,$$ and $r^{-1}\norm{g}_\infty + \norm{\nabla g}_\infty \leq \delta'$. There exists a $\delta_0(n)>0$ sufficiently small such that if $\delta \leq \delta_0\rho $ and $\delta'\leq 1$, then
 \begin{enumerate}
 \def\theenumi{\roman{enumi}}
  \item \label{it_s1} $\forall z\in G$, $ r^{-1}\abs{\sigma(z)-z}\leq c(n)(\delta +\delta')\rho^{-1}$, and $\sigma$ is a $C^1$ diffeomorphism from $G$ to its image,
  \item \label{it_s2} the set $\sigma(G)$ is contained in a $C^1$ graph $\cur{p+x+\tilde g(x)\, , \ \ x\in V}$ with
  \begin{gather}
   r^{-1}\norm{\tilde g}_\infty + \norm{\nabla \tilde g}_\infty \leq c(n) (\delta+\delta')\rho^{-1}\, .
  \end{gather}
  \item \label{it_s3} moreover, if $U'$ is such that $\B{c(\delta +\delta')\rho^{-1}}{U'}\subseteq \psi^{-1}(0)$, then the previous bound is \textit{independent} of $\delta'$, in the sense that
  \begin{gather}
   r^{-1}\norm{\tilde g}_{L^\infty(U'\cap V)} + \norm{\nabla \tilde g}_{L^\infty(U'\cap V)}\leq c(n) \delta\rho^{-1}\, .
  \end{gather}
  For example, if $\delta'\leq \delta_0(n)\rho^{-1}$, we can take $U'=\bigcup_{i} \B{1.5 r}{x_i}$.
  \item \label{it_s4} the map $\sigma$ is a bi-Lipschitz equivalence between $G$ and $\sigma(G)$ with bi-Lipschitz\newline constant $\leq 1+c(n)(\delta+\delta')^2\rho^{-2}$.
 \end{enumerate}
\end{lemma}

\begin{proof}
For convenience, we fix $r=1$ and $p=0$. By notation, given any map $f:\R^n\to \R^m$, $p\in \R^n$ and $w\in T_p(\R^n)=\R^n$, we will denote by $\nabla|_p f[w]$ the gradient of $f$ evaluated at $p$ and applied to the vector $w$.

Recall that 
 \begin{gather}
  \sigma(x+g(x))=\psi(z) (x+g(x)) + \sum_{x_i\in I} \lambda_i(z) \ton{\pi_{p_i,V_i}(x+g(x))}\, , \quad 1-\psi(x)=\sum_{x_i\in I} \lambda_i(x)\, ,
 \end{gather}
where we have set for convenience $z=z(x)=x+g(x)$. Define $h(x)$ by 
\begin{gather}
 (1-\psi(z))x + h(x) \equiv \sum_{i} \lambda_i(z) \ton{\pi_{p_i,V_i}(x+g(x))}\, .
\end{gather}
Set also $h^T(x)=\pi_V(h(x))$ and $h^\perp(x)=\pi_{V\perp}(h(x))$. By projecting the function $\sigma(x+g(x))$ onto $V$ and its orthogonal complement we obtain
\begin{gather}
 \sigma(x+g(x))\equiv \sigma^T(x)+\sigma^\perp(x) \, ,\notag\\
 \sigma^T(x)= x + h^T(x)\, , \quad \sigma^\perp(x)=\psi(z)g(x)+h^\perp(x)\, .
\end{gather}

We claim that if $\delta'\leq 1$, then
\begin{gather}\label{eq_hT}
 \abs{h^T(x)}+\abs{\nabla {h^T(x)}}\leq \frac{c\delta}{\rho}\, ,
\end{gather}
where this bound is \textit{independent} of $\delta'$ as long as $\delta'\leq 1$. Indeed, for all $x\in V$ we have
\begin{gather}
 h^T(x)=\pi_V \qua{\sum_i \lambda_i(z) \ton{\pi_{p_i,V_i}(x+g(x))  - x}  }=\sum_i \lambda_i(z) \pi_V \qua{\ton{\pi_{p_i,V_i}(x)  - \pi_{V}(x) } +\pi_{p_i,V_i}(g(x))}
\end{gather}
Given \eqref{eq_111} and Lemma \ref{lemma_epsproj}, with computations similar to \eqref{eq_222}, we get $\abs{h^T(x)}\leq c\delta (1+\rho^{-1})\leq c\delta \rho^{-1}$. As for the gradient, we get for any vector $w\in V$
\begin{gather}
 \nabla h^T|_x [w] = \pi_V \qua{ \sum_i \nabla\lambda_i |_z \qua{w+\nabla g|_x [w]} \ton{\pi_{p_i,V_i}(x+g(x))  - x} +\sum_i \lambda_i(z)\ton{\pi_{V_i}\ton{w+\nabla g[w]} -w }}\, ,
\end{gather}
In particular, we obtain
\begin{gather}
 \abs{\nabla h^T|_x [w]}\leq  \sum_i \abs{\nabla\lambda_i} \ton{1+\abs{\nabla g}}{\abs w}\abs{\pi_{p_i,V_i}(x+g(x))  - x} +\sum_i \lambda_i(z)\ton{\abs{\pi_{V_i}\ton{w}-w} + \abs{\pi_{V_i}\ton{\nabla g[w]}}}\, .
\end{gather}
For the first term, we can estimate
\begin{gather}
 \abs{\nabla \lambda_i}\leq c(n)\, , \quad \abs{\nabla g}\leq \delta'\leq 1\, , \quad \abs{\pi_{p_i,V_i}(x+g(x))  - x}\leq \abs{\pi_{p_i,V_i}(x)-x}+\abs{\pi_{V_i}(g(x))}\, .
\end{gather}
Since $x\in V$ with $\abs x \leq \rho^{-1}$, and $g(x)\in V^\perp$, by \eqref{eq_111} and Lemma \ref{lemma_epsproj} we obtain
\begin{gather}
 \abs{\pi_{p_i,V_i}(x+g(x))  - x}\leq c\delta \rho^{-1}\, .
\end{gather}
As for the second term, we have
\begin{gather}
 \abs{\abs{\pi_{V_i}\ton{w}-w}\leq  c\delta \abs w\, , \quad \abs{\pi_{V_i}\ton{\nabla g[w]}}}\leq c\delta \delta' \abs w \leq c \delta \abs w\, .
\end{gather}
Summing all the contributions, we obtain \eqref{eq_hT} as wanted.

Thus we can apply the inverse function theorem on the function $\sigma^T(x):V\to V$ and obtain a $C^1$ inverse $Q$ such that for all $x\in V$, $\abs{Q(x)-x}+ \abs{\nabla Q-id }\leq c(n)\delta \rho^{-1}$ , and if $\psi(x+g(x))=1$, then $Q(x)=x$ . So we can write that for all $x\in V$
\begin{gather}
 \sigma(x+g(x))=\sigma^T(x)+\tilde g(\sigma^T(x))\, \quad \text{where} \quad \tilde g(x)= \sigma^\perp (Q(x))=h^\perp(Q(x))+\psi\ton{z(Q(x))}g(Q(x)) \, .
\end{gather}

Arguing as above, we see that $h^\perp(x)$ is a $C^1$ function with 
\begin{gather}\label{eq_hp}
 \abs{h^\perp(x)}+\abs{\nabla h^\perp(x)}\leq \frac{c\delta}{\rho}\, ,
\end{gather}
and this bound is independent of $\delta'$ (as long as $\delta'\leq 1$).

Thus the function $\tilde g:V\to V^\perp$ satisfies for all $x$ in its domain 
\begin{gather}
\abs{\tilde g(x)}+ \abs{\nabla \tilde g(x)}\leq c(n)(\delta+\delta') \rho^{-1}\, ,
\end{gather}
Moreover, for those $x$ such that $\psi(Q(x)+g(Q(x)))=0$, the estimates on $\tilde g$ are independent of $\delta'$, in the sense that $\abs{\tilde g(x)}+ \abs{\nabla \tilde g(x)}\leq c(n)\delta \rho^{-1}$ . Note that by the previous bounds we have
\begin{gather}
 \abs{Q(x)+g(Q(x)) -x}\leq c (\delta+\delta')\rho^{-1}\, ,
\end{gather}
and so if $\B {c(\delta+\delta')\rho^{-1}}{U'}\subset \psi^{-1}(0)$, then for all $x\in U'\cap V$, $\psi(Q(x)+g(Q(x)))=0$. This proves items \eqref{it_s2}, \eqref{it_s3}. As for item \eqref{it_s1}, it is an easy consequence of the estimates in \eqref{eq_hT}, \eqref{eq_hp}.

Now since both $G$ and $\sigma(G)$ are Lipschitz graphs over $V$, it is clear that the bi-Lipschitz map induced by $\pi_V$ would have the right bi-Lipschitz estimate. Since $\sigma$ is close to $\pi_V$, it stands to reason that this property remains true. In order to check the estimates, we need to be a bit careful about the horizontal displacement of $\sigma$.

\paragraph{bi-Lipschitz estimates}
In order to prove the estimate in \eqref{it_s4}, we show that for all $z=x+g(x)\in G$ and for all unit vectors $w\in T_z(G)\subset \R^n$, we have
\begin{gather}\label{eq_nabla_2lip}
 \abs{ \abs{\nabla \sigma|_z [w]}^2-1}\leq c(\delta+\delta')^2\, .
\end{gather}
First of all, note that if $\psi(z)=1$, then $\sigma$ is the identity, and there's nothing to prove.

In general, we have that
\begin{gather}
 \nabla \sigma|_z [w]= \ton{\psi(z) w + \sum_i \lambda_i (z) \pi_{V_i}[w]} + \ton{ z \nabla \psi [w]+ \sum_i \pi_{p_i,V_i} (z) \nabla \lambda_i [w]} \equiv A+B\, .
\end{gather}
Since $\psi(z)+\sum_i \lambda_i(z)=1$ everywhere by definition, we have
\begin{gather}
 \abs{B} =\abs{ \sum_i (\pi_{p_i,V_i} (z) -z) \nabla \lambda_i [w]}\leq c \sup_i \cur{\abs{\pi_{p_i,V_i} (z) -z}}\leq c\delta'\, .
\end{gather}
This last estimate comes from the fact that $G$ is the graph of $g$ over $V$ with $\norm{g}_\infty \leq \delta'$. Moreover, we can easily improve the estimate for $B$ in the horizontal direction using Lemma \ref{lemma_epsproj}. Indeed, since $\pi_{p_i,V}(z)-z=\pi_{V_i^\perp}(z)$, we have
\begin{gather}
 \abs{\pi_V B} =\abs{ \sum_i \pi_V \ton{\pi_{p_i,V_i} (z) -z} \nabla \lambda_i [w]}\leq c \sup_i \cur{\abs{\pi_V\ton{\pi_{p_i,V_i} (z) -z}}}\leq c\delta'\delta\, .
\end{gather}
As for $A$, by adapting the proof of Lemma \ref{lemma_sigma_simon}, we get $\abs{A-\pi_V[w]}\leq c\delta$. Moreover, also in this case we get better estimates for $A$ in the horizontal direction. Indeed, we have
\begin{gather}
 \abs{\pi_V(A) -\pi_V[w]} = \abs{\psi(z) \pi_V [w] +\sum_i \ton{ \lambda_i(z) \pi_V[\pi_{V_i}[w]]} -\pi_V[w]} = \abs{\sum_i  \lambda_i(z) \ton{\pi_V[\pi_{V_i}[w] -\pi_V[w]]}}\, .
\end{gather}
Now let $w=\pi_V[w]+\pi_{V^\perp}[w]=w_V+w_{V^\perp}$. Then we have
\begin{gather}
 \abs{\pi_V(A) -\pi_V[w]}\leq \sum_i  \lambda_i(z) \ton{\abs{\pi_V[\pi_{V_i}[w_V] -w_V]} + \abs{\pi_V[\pi_{V_i}[w_{V^\perp}]]}}=\sum_i  \lambda_i(z) \ton{\abs{\pi_V[\pi_{V_i^\perp}[w_V]]} + \abs{\pi_V[\pi_{V_i}[w_{V^\perp}]]}}\, .
\end{gather}
Since $G$ is the Lipschitz graph of $g$ over $V$ with $\norm{\nabla g}\leq c\delta'$, then $\norm{\pi_{V^\perp}[w]}\leq c\delta'$. Then, by Lemma \ref{lemma_epsproj}, we have
\begin{gather}
 \abs{\pi_V(A) -\pi_V[w]}\leq c\sum_i  \lambda_i(z) \ton{\delta^2+\delta\delta'}\, .
\end{gather}
Summing up, since $\abs{\pi_V[w]}\leq \abs w =1$, we obtain that
\begin{gather}
 \abs{\abs{\nabla\sigma|_z[w]}^2 -1} = \abs{\abs{\pi_{V^\perp}\nabla\sigma|_z[w]}^2+\abs{\ton{\pi_{V}\nabla\sigma|_z[w] - \pi_V[w]}+\pi_V[w]}^2  -1} \leq\\
 \leq c(\delta+\delta')^2+\abs{\abs{\pi_V[w]}^2 -1}=c(\delta+\delta')^2+\abs{\pi_{V^\perp}[w]}^2\leq c(\delta+\delta')^2\, .
\end{gather}
\end{proof}

\vspace{.5 cm}

\subsection{\texorpdfstring{Pointwise Estimates on $D$}{Pointwise Estimates on D}}  We wish to see in this subsection how \eqref{e:reifenberg_excess_L2} implies pointwise estimates on $D$, which will be convenient in the proof of the generalized Reifenberg results.  Indeed, the following is an almost immediate consequence of Remark \ref{r:D_scale_control}:

\begin{lemma}\label{l:D_pointwise_bound}
Assume $B_r(x)\subseteq B_2(0)$ is such that $\mu(B_r(x))\geq 4^k\epsilon_n r^k$ and $\int_{B_{2r}(x)}D^k_\mu(y,2r)\,d\mu(y)<\delta^2 (2r)^k$.  Then there exists $c(n)$ such that $D(x,r)<c\delta^2$.  In particular, if \eqref{e:reifenberg_excess_L2} holds then for every $B_r(x)\subseteq B_2(0)$ such that $\mu(B_r(x))\geq 4^k\epsilon_n r^k$ we have that $D(x,r)<c\delta^2$.
\end{lemma}

\vspace{.5 cm}

\section{Proof of Theorem \ref{t:reifenberg_W1p_holes}: The Rectifiable-Reifenberg}\label{sec_proof_main}
Here we carry out the proof of Theorem \ref{t:reifenberg_W1p_holes}.

In the proof, we will fix the constant $M=C_1(n)\leq 40^n$ and therefore a positive scale $\rho(n,C_1(n))=\rho(n)<1$ according to Lemma \ref{lemma_alpharho}.  For convenience, we will assume that $\rho=2^{-q}$, $q\in \N$, so that we will be able to use the sum bounds \eqref{eq_sum2^alpha} more easily.

The constant $C_1(n)$ will be defined by the end of the proof, however it is enough to know that it is can be taken to be $C_1(n)=40^n$.  Thus, the value of the parameter $\rho$ will depend only on the dimension $n$, as already pointed out in \eqref{eq_rho_rough}. Let us also define the scales $r_j=\rho^j$.


\subsection{\texorpdfstring{Weak Upper bound $\lambda^k(S\cap B_r(x))\leq C(n) r^k$}{Weak Upper bound on the measure}}\label{sss:weak_upper_bound}

We start by proving a uniform upper bound weaker than \eqref{eq_lambda_obj}, in particular we want to show that for all $x\in \R^n$ and $r>0$:
\begin{gather}\label{eq_lambda_obj2}
\lambda^k(S\cap B_r(x))\leq C(n) r^k \, .
\end{gather}
Once we have obtained this estimate, the stronger upper bound and the rectifiability will be almost corollaries of this proof.

Given the scale invariance of the quantities involved, we do not lose generality if we prove \eqref{eq_lambda_obj2} only for $x=0$ and $r=1$. The strategy for the proof is the following: first we prove that $S$ has $\sigma$-finite $k$-dimensional Hausdorff measure, and use this information to build a suitable covering by balls with controlled $\lambda^k$-measure. Then we fix any $A\in \N$ and $r_A=\rho^A$ and show by induction on $j=A,\cdots,0$ that the measure of quasi-balls $\tB{r_j}{y}$, defined in \eqref{eq_tB}, is bounded above as in \eqref{eq_lambda_obj2}.  The definition of a quasi-ball will be such that for $r=1$ the quasi-ball $\tB{1}{x}$ agrees with the set $S\cap B_1(x)$, up to a set of measure zero, which will prove \eqref{eq_lambda_obj2}.  In order to prove the measure statement on the quasi-balls, we will use a second downward induction on $i=j,\cdots,A$, which is the technical heart of the construction.

\paragraph{$\sigma$-finiteness of the measure} As a first step towards the proof, we remark that $S$ must have $\sigma$-finite $k$-dimensional Hausdorff measure.  We will use this to reduce the proof to the case when $\lambda^k(S)<\infty$. Indeed, in order for \eqref{e:reifenberg_displacement_L2} to be true, we need in particular
 \begin{gather}
  \int_{S\cap \B 1 0}D^k_S(x,1)\, d\lambda^k(x)\,ds<c \delta^2\, .
 \end{gather}
 Define for $a>0$ the sets $S_a=\cur{x\in S \ \ s.t. \ \ D^k_S(x,1)\geq a}$ and $S_0=\cur{x\in S \ \ s.t. \ \ D^k_S(x,1)=0}$. Then evidently for all $a>0$, $\lambda^k(S_a)<c\delta^2/a<\infty$. Moreover, if $x\in S_0$, then either $\lambda^k(S\cap \B 1 x)<\epsilon_n\leq 1 $, or up to sets of $k$-measure zero, $S\cap \B 1 x $ is contained in a $k$-dimensional plane, and therefore $\lambda^k(S_0)<\infty$. Since $S=S_0\cup_{i=1}^\infty S_{i^{-1}}$, $S$ has $\sigma$-finite $k$-dimensional measure, as claimed.
 
 Now in order to prove the uniform bounds on $\lambda^k(S)$, we can therefore assume without loss of generality that $\lambda^k(S)<\infty$. Indeed, by the monotonicity of Remark \ref{rem_mon}, all subsets of $S$ satisfy \eqref{e:reifenberg_displacement_L2}. Therefore if we show that \eqref{eq_lambda_obj2} holds for all subsets of finite $k$-dimensional measure, then the estimate will hold also on $S$.  Thus, we will assume throughout the remainder of the proof that $\lambda^k(S)<\infty$.

 \vspace{.5 cm} 

\paragraph{Covering of the set $S^\star$} Let $S^\star\subseteq S\cap \B 1 0$ be the set of points with controlled upper density, i.e.
 \begin{gather}
  S^\star=\cur{x\in S\cap \B 1 0 \ \ s.t. \ \ 2^{-k}\leq \theta^{\star k}(S,x)\leq 1}\, .
 \end{gather}
 By proposition \ref{prop_dens}, $S\setminus S^\star$ has zero $k$-dimensional measure, thus it is sufficient to give bounds on $S^\star$.

 We want to cover the set $S^\star$ by balls which have uniform upper and lower bounds on their Hausdorff measure and such that their best $L^2$ plane is not too far away from their center. In order to achieve this, for all $x\in S^\star$ let $r_x$ be such that
 \begin{align}
 r_x=\rho^{n_x}\, , \quad &\text{where} \ \ n_x\in \N \, , \ \ n_x\geq 2\, \label{eq_rho2} ,\\
 \lambda^{k}\ton{\B {\rho r_x} x \cap S}&\geq \frac{1}{2^{k+1}} \omega_k \rho^k \ton{\rho r_x}^k\, , \notag\\
\lambda^k \ton{\B {r} x \cap S}&\leq 2\omega_k r^k\, ,  \quad   \forall r\leq r_x\,  \label{eq_lambda_dens}\ \  
 \end{align}

Let $p_x$ be the center of mass of $\B {\rho r_x} x \cap S$ with respect to $\lambda^k|_{\B {\rho r_x} x \cap S}$.  In particular, we have that $p_x\in \B {\rho r_x} x$.
By Jensen's inequality, we have
\begin{gather}
 \lambda^k\ton{\B {\rho r_x} x \cap S} d(p_x,V(x,r_x))^2\leq \int_{\B {\rho r_x} x \cap S} d(y,V(x,r_x))^2\leq  r_x^{k+2} D(x,r_x)\, ,
\end{gather}
where the last inequality comes from Definition \ref{deph_D}, \eqref{eq_rho_rough} and $\epsilon_n = (1000n)^{-7n^2}$. Using Lemma \ref{l:D_pointwise_bound} and \eqref{e:reifenberg_displacement_L2} we therefore get that
\begin{align}\label{eq_yfinal}
\lambda^k\ton{\B {\rho r_x} x \cap S} d(p_x,V(x,r_x))^2\leq c r_x^{k+2} \delta^2\, .
\end{align}

In particular, using our lower bound on $\lambda^k\ton{\B {r_x/8} x \cap S}$ we obtain that $d(p_x,V(x,r_x))\leq c\delta r_x$, which implies for $\delta$ small enough that
\begin{gather}\label{eq_dfcm}
d(x,V(x,r_x))\leq (\rho+c\delta)r_x\leq r_x/100\, .
\end{gather}

Now consider the open covering of $S^\star$ given by $\cup_{x\in S^\star} \B{r_x/5}{x}$, and extract a countable Vitali subcovering. Thus
\begin{gather}
 S^\star \subseteq \bigcup_{x\in \tilde S} \B{r_x}{x}\cap S^\star\, ,
\end{gather}
where if $x, \ y\in \tilde S, \ x\neq y$, then $ \B{r_x/5}{x}\cap \B{r_y/5}{y}=\emptyset$.

\vspace{3mm}

Fix any $A\in \N$ and $\bar r=\rho^A$, and define the sets $\tilde S _{\bar r}^r = \cur{x\in \tilde S \ \ s.t. \ \ \bar r\leq r_x\leq r}$,  $\tSr=\tSr^{\rho^2}$ and 
\begin{gather}
S_{\bar r}=\bigcup_{x\in \tSr}\B{r_x}{x} \cap S^\star\, .
\end{gather}

It is clear that $S_{\bar r}\nearrow S^\star$ as $\bar r \to 0$, so if we have bounds on $\lambda^k(S_{\bar r})$ which are independent of $\bar r$, we are done. Thus from here on we will consider $\bar r=\rho^A$ to be positive and fixed.

\subsection{First induction: upwards} We are going to prove inductively on $j=A,\cdots,0$ that for all $x\in \R^n$ and $r_j=\rho^j\leq 1$ the measure of
\begin{gather}\label{eq_tB}
 \tB{r}{x}= \bigcup_{y\in \tSr^{\rho r}\cap{\B{r}{x}}}\B{r_y}{y} \cap S^\star \, ,
\end{gather}
is bounded by 
\begin{gather}\label{eq_srx}
 \lambda^k\ton{\tB {r_j} x } \leq C_1(n) r_j^k\leq 40^n r_j^k\, .
\end{gather}
Note that for $j=A$ this bound follows from the definition of $r_x$ and \eqref{eq_lambda_dens}. 
We emphasize that by construction $\B {r_y}{y}\cap S^\star $ appears in the union in \eqref{eq_tB} only if $\bar r \leq r_y\leq \rho r$. However, given \eqref{eq_rho2}, we have the inclusion $S_{\bar r} \cap \B 1 0 \subseteq \tB{1}{0}$. Note also that $\tB{r}{x}\subseteq \B{r(1+\rho)}{x}$. .

The reason why we have to introduce and estimate $\tB{r_j}{x}$ instead of $\B{r_j}{x}$ is that we have no a priori control of what happens inside any of the balls $\B{r_y}{y}$. The bounds \eqref{eq_lambda_dens} are valid only on each ball as a whole.  However, since our primary goal is to estimate $S_{\bar r}  \cap \B 1 0$, and we have the inclusion $S_{\bar r} \cap \B 1 0 \subseteq \tB{1}{0}$, there is no loss in this restriction.

\paragraph{Rough estimate} Fix some $j$, and suppose that \eqref{eq_srx} holds for all $\bar r \leq r_k \leq r_{j+1}$.
 
Let us first observe that we can easily obtain a bad upper bound on $\lambda^k\ton{\tB{\chi r_{j+1}}{x}}$ for some fixed $\chi>1$. Consider the points $y\in \tSr\cap{\B{\chi r_{j+1}}{x}}$, and divide them into two groups: the ones with $r_y\leq r_{j+2}$ and the ones with $r_y\geq r_{j+1}$.

For the first group, cover them by balls $\B{r_{j+1}}{z_i}$ such that $\B{r_{j+1}/2}{z_i}$ are disjoint. Since there can be at most $c(n,\chi)$ balls of this form, and for all of these balls the upper bound \eqref{eq_srx} holds, we have an induced upper bound on the measure of this set.

As for the points with $r_y> r_{j+2}$, by construction there can be only $c(n,\chi)$ many of them, and we also have the bound $r_y\leq \rho\chi r_{j+1}$, which by \eqref{eq_lambda_dens} implies $\lambda^k\ton{\tB{r_y}{y}}\leq c(n,\chi) r_{j+1}^k$. Summing up the two contributions, we get the very rough estimate
\begin{gather}\label{eq_lambdarough}
 \lambda^k(\tB{\chi r_{j+1}}{x})\leq C_2(n,\chi) r_{j+1}^k\, ,
\end{gather}
where $C_2>>C_1$.  Note that, as long as the inductive hypothesis holds, $C_1$ and $C_2$ are independent of $j$.  However, it is clear that successive repetitions of the above estimate will not lead to \eqref{eq_srx}.  Our goal therefore is to push down this estimate to arrive at the better constant of $C_1$, however it will be technically very convenient when applying the tools of Section \ref{ss:best_comparison} that we may assume the worse bound in the process.
\vspace{.5 cm}

\subsection{Second induction: downwards: outline}
Suppose that \eqref{eq_srx} is true for all $x\in \B 1 0$ and $i=j+1,\cdots,A$. Fix $x\in \R^n$, and consider the set $$\tilde B=\tB{r_j}{x}\subseteq \B{(1+\rho)r_j}{x}\, .$$  
We are going to build by induction on $i\geq j$ a sequence of smooth maps $\sigma_i:\R^n\to \R^n$ and smooth $k$-dimensional manifolds $T_{j,i}=T_i$ which satisfy nine properties, which we will use to eventually conclude the proof of \eqref{eq_srx} for $\tilde B$.  Let us outline the inductive procedure now, and introduce all the relevant terminology.  Everything described in the remainder of this subsection will be discussed more precisely over the coming pages.  To begin with, we will have at the first step that
\begin{align}
 &\sigma_j=id,\, \notag\\
 &T_{j,j}=T_j=\tilde V(x,r_j)\cap \B {2r_j}{x}\subseteq \R^n\, ,
\end{align}
where $\tilde V(x,r_j)$ is one of the $k$-dimensional affine subspaces which minimizes $\int_{\tB{r_j}{x}} d^2(y,\tilde V)\,d\lambda^k$.  Thus, the first manifold $T_j$ is a $k$-dimensional affine subspace which best approximates $\tB{r_j}{x}$.  At future steps we can recover $T_i$ from $T_{i-1}$ and $\sigma_i$ from the simple relation
\begin{align}
 &T_i=\sigma_i(T_{i-1})\, .
\end{align}
We will see that $\sigma_i$ is a diffeomorphism when restricted to $T_{i-1}$, and thus each additional submanifold $T_i$ is also diffeomorphic to $\dR^k$.  As part of our inductive construction we will build at each stage a covering of $T_i$ given by
\begin{align}\label{e:downward_induction:1}
\B{r_j}{T_i}\cap \tB{r_j}{x}\sim \bigcup_{s=j}^i\ton{\bigcup_{y\in I_b^s}\tB {r_s}{y} \bigcup_{y\in I_f^s} \ton{\B{r_s}{y}\cap S_{\bar r}}}\bigcup_{y\in I_g^i} \tB{r_i}{y}\, ,
\end{align}
where given any two distinct balls $B_1$ and $B_2$ in the covering, $B_1/5$ and $B_2/5$ are disjoint. Here $I_g$, $I_b$, and $I_f$ represent the {\it good}, {\it bad}, and {\it final} balls in the covering. A final ball $\tB{r_i}{y}$ with $y\in I^i_f$ is a ball such that $y\in \tilde S$ and $r_y=r_i$, and the other balls in the covering are characterized as good or bad according to how much measure they carry. Good balls are those with large measure, bad balls the ones with small measure. More precisely, we have
\begin{align}\label{e:downward_induction:2}
&\lambda^k\big(\tB {r_i}{y}\big)\geq \gamma_k r_i^k\, ,\quad\text{if}\quad y\in I^i_g\, ,\notag\\
&\lambda^k\big(\tB {r_i}{y}\big)<\gamma_k r_i^k\, ,\quad\text{if}\quad y\in I^i_b\, ,
\end{align}
where $\gamma_k=\omega_k 40^{-k}$. We will see that, over each ball $\tB{r_s}{y}$ in this covering, $T_i$ can be written as a graph over the best approximating subspace $\tilde V(y,r_s)$ with good estimates. \\

Our goal in these constructions is the proof of \eqref{eq_srx} for the ball $\tilde B=\tB{r_j}{x}$, and thus we will need to relate the submanifolds $T_i$, and more importantly the covering \eqref{e:downward_induction:1}, to the set $\tilde B$.  Indeed, this covering of $T_i$ almost covers the set $\tilde B$, at least up to an excess set $\tilde E_{i-1}$.  That is,
\begin{align}
\tilde B\subseteq \tilde E_{i-1}\cup \bigcup_{s=j}^i\ton{\bigcup_{y\in I_b^s}\tB {r_s}{y} \bigcup_{y\in I_f^s} \ton{\B{r_s}{y}\cap S_{\bar r}}}\bigcup_{y\in I_g^i} \tB{r_i}{y}\, .
\end{align}
We will see that the set $\tilde E_{i-1}$ consists of those points of $\tilde B$ which do not satisfy a uniform Reifenberg condition.  Thus in order to prove \eqref{eq_srx} we will need to estimate the covering \eqref{e:downward_induction:1}, as well as the excess set $\tilde E_{i-1}$.\\

Let us now outline the main properties used in the inductive construction of the mapping $\sigma_{j+1}:\dR^n\to \dR^n$, and hence $T_{j+1}=\sigma_{j+1}(T_j)$.  As is suggested in \eqref{e:downward_induction:1}, it is the good balls and not the bad and final balls which are subdivided at further steps of the induction procedure.  In order to better understand this construction let us begin by analyzing the good balls $\tB{r_i}{y}$ more carefully.  On each such ball we may consider the best approximating $k$-dimensional subspace $\tilde V(y,r_i)$.  Since $\tB{r_i}{y}$ is a good ball, one can check that most of $\tilde B\cap \tB{r_i}{y}$ must satisfy a uniform Reifenberg and reside in a small neighborhood of $\tilde V(y,r_i)$.  We denote those points which don't by $\tilde E(y,r_i)$, see \eqref{eq_exrough} for the precise definition.  Then we can define the next step of the excess set by
\begin{align}
\tilde E_i = \tilde E_{i-1}\bigcup_{y\in I^i_g} \tilde E(y,r_i)\, . 
\end{align}
Thus our excess set represents all those points which do not lie in an appropriately small neighborhood of the submanifolds $T_i$.  With this in hand we can then find a submanifold $T'_i\subseteq T_i$, which is roughly defined by 
\begin{align}
T'_i \approx T_i\cap \bigcup_{y\in I_g^i} \tB{r_i}{y}\, ,
\end{align}
see \eqref{e:downward_induction:map_manifold:1} for the precise inductive definition, such that
\begin{align}
\tilde B \subseteq \tilde E_{i}\cup \bigcup_{s=j}^i\ton{\bigcup_{y\in I_b^s}\tB {r_s}{y} \bigcup_{y\in I_f^s} \tB{r_s}{y}} \bigcup B_{r_{i+1}/4}\big(T'_i\big) \equiv \tilde R_i\bigcup B_{r_{i+1}/4}\big(T'_i\big)\, ,
\end{align}
where $\tilde R_i$ represents our remainder term, and consists of those balls and sets which will not be further subdivided at the next stage of the induction.  Now in order to finish the inductive step of the construction, we can cover $B_{r_{i+1}/4}\big(T'_i\big)$ by some Vitali set
\begin{align}
B_{r_{i+1}/4}\big(T'_i\big)\subseteq \bigcup_{y\in I} \tB{r_{i+1}}{y}\, ,
\end{align}
where $y\in I\subseteq T'_i$.  We may then decompose the ball centers 
\begin{align}
I^{i+1}=I^{i+1}_g\cup I^{i+1}_b\cup I^{i+1}_f\, ,
\end{align}
based on \eqref{e:downward_induction:2}.  Now we will use Definition \ref{deph_sigma} and the best approximating subspaces $V(y,r_{i+1})$ to build $\sigma_{i+1}:\dR^n\to \dR^n$ such that
\begin{align}
\text{supp}\{\sigma_{i+1}-Id\}\subseteq \bigcup_{y\in I^{i+1}_g}\B{3r_{i+1}}{y}\, .
\end{align}
This completes the outline of the inductive construction.

\subsection{Second induction: downwards: detailed proof}\label{sec_IIind_pre}
Let us now describe precisely the proof of this inductive construction which will lead to \eqref{eq_srx}. For $j\leq i\leq A$, we will define a sequence of approximating manifolds $T_i$ for the set $S$ and a sequence of smooth maps $\sigma_i$ such that
\begin{enumerate}
 \def\theenumi{\roman{enumi}}
 \item \label{it_i1} $\sigma_j=id$, $T_{j,j}=T_j=\tilde V(x,r_j)\cap \B {2r_j}{x}\subseteq \R^n$,
 \item \label{it_i2} $T_i=\sigma_i(T_{i-1})$,
 \item \label{it_i3} for $y\in T_{i}$,
 \begin{gather}\label{eq_distdelta}
  d(\sigma_{i+1}(y),y)\leq c\delta r_{i+1}\, ,
 \end{gather}
 and $\sigma_{i+1}|_{T_{i}}$ is a diffeomorphism,
 \item \label{it_i5} for every $y\in T_i$, $T_i\cap \B{2 r_i}{y}$ is the graph over some $k$-dimensional affine subspace of a smooth function $f$ satisfying
  \begin{gather}\label{eq_lipTi}
  \frac{\norm{f}_{\infty}}{r_i} + \norm{\nabla f}_\infty \leq c\delta\, ,
 \end{gather}
\end{enumerate}
As outlined before, the manifolds $T_i$ will be good approximations of the set $S$ up to some ``excess'' set of small measure. Moreover, we will also introduce the concept of \textit{good, bad} and \textit{final} balls (whose centers will be in the sets $I_g^i$, $I_b^i$ and $I_f^i$), a \textit{remainder set} $R^i$, and the manifolds $T_i'\subseteq T_i$. Before giving the precise definitions (which are in equations \eqref{eq_deph_Igb}, \eqref{eq_deph_If}, \eqref{eq_R} and \eqref{e:downward_induction:map_manifold:1} respectively), let us group here all the properties that we will need (and prove) for these objects, so that the reader can always come back to this page to have a clear picture of what are the objectives of the proof.
\begin{enumerate}
 \setcounter{enumi}{4} \def\theenumi{\roman{enumi}}
 \item\label{it_i4} for every $i\geq j+1$ and $y\in I_g^i$, $d(y,\tilde V(y,r_i))\leq c\delta r_i$, the set $T_i\cap \B{1.5 r_i}{y}$ is the graph over $\tilde V(y,r_i)$ of a smooth function $f$ satisfying \eqref{eq_lipTi}, where $\tilde V(x,r)$ is one of the $k$-dimensional affine subspaces minimizing $\int_{\tB{r}{x}} d^2(y,\tilde V)\,d\lambda^k$,
 \item \label{it_i6} for all $i$, we have the inclusion
 \begin{gather}\label{eq_sub}
  \tilde B \subseteq \B{r_{i}/2}{T_{i-i}'}\bigcup \tilde R^{i-1}\, ,
 \end{gather}
 \item \label{it_i7} for every $i\geq j+1$ and all $y\in I_g^i$, the set $T_{i-1}'\cap \B{2 r_i}{y}$ is a Lipschitz graph over the plane $\tilde V(y,r_{i})$ with $d(y,\tilde V(y,r_{i}))\leq c\delta r_{i}$ ,
\end{enumerate}
The last two properties needed are the key for the final volume estimates:
\begin{enumerate}
 \setcounter{enumi}{7} \def\theenumi{\roman{enumi}}
\item \label{it_i8} we can estimate
 \begin{gather}\label{eq_volFB}
  \lambda^k(\sigma_i^{-1}(T_i'))+ \qua{\#\ton{I_b^i}+\#\ton{I_f^i}} \omega_k (r_i/10)^k \leq \lambda^k(T_{i-1}')\, ,
 \end{gather}
 \item\label{it_i9} we can estimate the excess set by
 \begin{align}
 \lambda^k\big(\tilde E(y,r_i)\big) r_{i+1}^2\leq C(n) D^k_S(y,2r_i))\, .
 \end{align}
\end{enumerate}
\vspace{.5 cm}

\paragraph{Inductive definitions} First of all, note that we can assume without loss of generality that
\begin{gather}
 \lambda^k(\tilde B_{r_j}(x))\geq \gamma_k r_j^k\, ,
\end{gather}
otherwise there's nothing to prove. With this hypothesis, we start our inductive construction by setting
\begin{gather}
 I_g^j = \cur{x}\, , \quad I_f^j=I_b^j = \emptyset\, , \quad T^j = \tilde V(x,r_j)\, \quad T_j'=T_j\, , \quad \sigma_j=id.
\end{gather}

\paragraph{Excess set.}
Let us begin by describing the construction of the excess set. Fix any $y$ and $r_i\geq \rho^{-1}\bar r=\rho^{A-1}$, and assume that $\tB{r_i}{y}$ satisfies $\lambda^k(\tilde B_{r_i}(y))\geq \gamma_k r_i^k$.\\

Thus define $\tilde V(y,r_i)$ to be (one of) the $k$-dimensional plane minimizing $\int_{\tB{r_i}{y}} d(y,\tilde V)^2 d\lambda^k$, and define also the excess set to be the set of points which are some definite amount away from the best plane $\tilde V$.  Precisely, 
\begin{gather}\label{eq_E}
 E(y,r_i)= \B{r_i}{y} \setminus \B{r_{i+1}/4}{\tilde V}\, , \quad \tilde E(y,r_i) =  \bigcup_{y\in \tSr^{r_{i+2}}\cap E(y,r_i)}\B{r_y}{y} \bigcap S^\star \, .
\end{gather}
The points in $\tilde E$ are in some sense what prevents the set $S$ from satisfying a uniform Reifenberg condition on $\B {r_i}{y}$. By construction, all points in $\tilde E$ have a uniform lower bound on the distance from $\tilde V$, so that if we assume $\lambda^k(\tilde B_{r_i}(y))\geq \gamma_k r_i^k$, i.e. $\B{r_i}{y}$ is a good ball, then we can estimate
\begin{align}\label{eq_exrough}
\int_{\tB{r_i}{y}\setminus \tilde E(y,r_i)} &d(y,\tilde V(y,r_i))^2 \ d\lambda^k(y)+\lambda^k\big(\tilde E(y,r_i)\big) (r_{i+1}/5)^2\leq \int_{\tB{r_i}{y}} d(y,\tilde V(y,r_i))^2 \ d\lambda^k(y)\notag \\
&\leq \int_{\tB{r_i}{y}} d(y,V(y,r_i))^2 \ d\lambda^k(y)
\leq \int_{\B{(1+\rho)r_i}{y}} d(y,V(y,(1+\rho)r_i))^2\notag\\
&= ((1+\rho)r_i)^{k+2} D^k_S(y,(1+\rho)r_i)\leq C(n) r_i^{k+2} D_S^k(y,2r_i)\, .
\end{align}
\vspace{.5 cm}

\paragraph{Good, bad and final balls, and remainder set} Inductively, let us define the remainder set to be the union of all the previous bad balls, final balls, and the excess sets:
\begin{gather}\label{eq_R}
R^i=\bigcup_{s=j}^i\ton{\bigcup_{y\in I_b^s}\B {r_s}{y} \bigcup_{y\in I_f^s} \B{r_{s}}{y} \bigcup_{y\in I_g^s}E(y,r_s) }\, ,\quad \tilde R^i=\bigcup_{s=j}^i\ton{\bigcup_{y\in I_b^s}\tB {r_s}{y} \bigcup_{y\in I_f^s} \ton{\B{r_s}{y}\cap S_{\bar r}} \bigcup_{y\in I_g^s}\tilde E(y,r_s) }\, . 
\end{gather} 
The set $\tilde R^i$ represents everything we want to throw out at the inductive stage of the proof.  We will see later in the proof how to estimate this remainder set itself.  

Note that for $i=j$, i.e. at the first step of the induction, we have
\begin{gather}
R^j=E(x,r_j)\, ,\quad \tilde R^j=\tilde E(x,r_j) \, . 
\end{gather}

Now consider the balls in the covering outside the remainder set, and separate the balls with radius $\geq r_{i+1}$ from the others by defining for $y\in I_g^i$ the sets 
\begin{gather}\label{eq_deph_If}
I_f^{i+1}(y)=\cur{z\in \Big(\tSr\setminus  R^i\Big)\cap \B{r_i}{y} \ \ s.t. \ \ r_z=r_{i+1}}\, ,
\end{gather}
and
\begin{gather}
J^{i+1}(y)=\cur{z\in \Big(\tSr^{r_{i+2}}\setminus R^i\Big)\cap \B{r_i}{y}\cap \B{r_{i+1}/3}{\tilde V(y,r_i)}}\, .
\end{gather}
From this we can construct the sets
\begin{gather}
I_f^{i+1}=\cup_{y\in I_g^i} I_f^{i+1}(y)\quad \text{ and }\quad J^{i+1}=\cup_{y\in I_g^i} J^{i+1}(y)\, ,
\end{gather}


Note that by construction we have 
\begin{gather}\label{e:gbf_balls:1}
 S_{\bar r}\setminus \tilde R^i \subseteq  \bigg(\bigcup_{z\in I_f^{i+1}}\B{r_{i+1}}{z} \bigcup_{z\in J^{i+1}} \B{r_z}{z}\bigg)\bigcap S_{\bar r} \, .
\end{gather}

Let us now consider a minimal covering of \eqref{e:gbf_balls:1} given by 
\begin{align}
 S_{\bar r} \setminus \tilde R^i \subseteq \bigcup_{z\in I^{i+1}_f}\B{r_{i+1}}{z}\, \bigcup_{z\in I}\B {r_{i+1}}{z}\, ,
\end{align}
where $I\subseteq T_i'$, and for any $p\neq q\in I_f^{i+1}\cup I$, $\B{r_{i+1}/5}{p}\cap \B{r_{i+1}/5}{q}=\emptyset$.  Note that this second property is true by definition for $p,q\in I_f^{i+1}$, we only need to complete this partial Vitali covering with other balls of the same size. To be precise, note that by \eqref{eq_R} and \eqref{eq_E}
\begin{gather}\label{eq_333}
 S_{\bar r}\setminus R^i \setminus \bigcup_{z\in I_f^{i+1}}\B{4 r_{i+1}/5}{z} \subset \ton{\bigcup_{y\in I_g^i} \B {r_{i+1}/4 }{\tilde V(y,r_i)} }\bigcap \ton{\bigcup_{y\in I_g^i} \B{r_i}{y}} \bigcap \ton{\bigcup_{z\in I^{i+1}_f}\B{4r_{i+1}/5}{z}}^C\, .
\end{gather}
Take a finite covering of this last set by balls $\cur{\B {r_{i+1}/3}{y}}_{y\in Y}$. Note that we can pick
\begin{gather}\label{eq_disj}
 Y\cap \ton{\bigcup_{z\in I^{i+1}_f}\B{r_{i+1}}{z}}=\emptyset\, .
\end{gather}
Since $T_i'$ is locally a Lipschitz graph over $\tilde V(y,r_i)$ with \eqref{eq_lipTi}, we can choose $Y\subset T_i'$. Moreover, since we have the inclusion $ S_{\bar r} \setminus \tilde R^i \subseteq \bigcup_{y\in I_g^i} \tB{r_i}{y}$, we can also choose $Y\subset \bigcup_{y\in I_g^i} \B{r_i}{y}$.

Consider a Vitali subcovering of this set, denote $I$ the set of centers in this subcovering. Such a subcovering will have the property that the balls $\cur{\B {r_{i+1}/3} {y}}_{y\in I}$ will be pairwise disjoint. These balls will also be disjoint from $\bigcup_{z\in I^{i+1}_f}\B{r_{z}/5}{z}$ by \eqref{eq_disj}. The (finite version of) Vitali covering theorem ensures that $\bigcup_{y\in I}\B {r_{i+1}} {y}$ will cover the whole set in \eqref{eq_333}.

Now by construction of $I_f$ and the remainder set, all the balls $\cur{\B {r_x}{x}}_{s\in \tilde S_{\bar r}}$ with $r_x\geq r_{i+1}$ have already been accounted for. This means that 
\begin{align}
 S_{\bar r} \setminus \tilde R^i \setminus \bigcup_{z\in I^{i+1}_f}\B{r_{i+1}}{z}\subset \bigcup_{y\in I}\tB {r_{i+1}} {y}\, ,
\end{align}
as desired.

We split the balls with centers in $I$ into two subsets, according to how much measure they carry. In particular, let
\begin{gather}\label{eq_deph_Igb}
 I_g^{i+1} = \cur{y\in I \ \ s.t. \ \ \lambda^k \ton{\tB{r_{i+1}}{y}}\geq \gamma_k r_{i+1}^k}\, , \quad I_b^{i+1} = \cur{y\in I \ \ s.t. \ \ \lambda^k \ton{\tB{r_{i+1}}{y}}< \gamma_k r_{i+1}^k}\, .
\end{gather}
\vspace{.5 cm}

\subsection{Map and manifold structure.}
Let $\cur{\lambda_{s}^{i+1}}=\cur{\lambda_{s}}$ be a partition of unity such that for each $y_s\in I_g^{i+1}$
\begin{itemize}
 \item $\supp{\lambda_{s}}\subseteq \B{3 r_{i+1}}{y_s}$
 \item for all $z\in \cup_{y_s\in I_g^{i+1}} \B{2 r_{i+1}}{y_s}$, $\sum_s \lambda_{s}(z)=1$
 \item $\max_s \norm{\nabla \lambda_{s}}_\infty \leq C(n)/r_{i+1}$.
\end{itemize}
For every $y_s\in I_g^{i+1}$, let $\tilde V(y_s,r_{i+1})$ to be (one of) the $k$-dimensional subspace that minimizes $\int_{\tB{r_{i+1}}{y_s}} d(z,V)^2 d\lambda^k$. By Remark \ref{rem_mon} and by \eqref{eq_tB}, we can estimate
\begin{gather}
 r_{i+1}^{-k-2}\int_{\tB{y_s}{r_{i+1}}} d(z,\tilde V(y_s,r_{i+1}))^2 d\lambda^k(z) \leq (1+\rho)^{k+2}D(y_s,(1+\rho)r_{i+1})\, .
\end{gather}
Let $p_s\in \B{(1+\rho)r_{i+1}}{y_s}$ be the center of mass of $\lambda^k|_{\tB {r_{i+1}}{y_s}}$.  It is worth observing that $p_s\in \tilde V_{y_s,r_{i+1}}$.

Define the smooth function $\sigma_{i+1}:\R^n\to\R^n$ as in Definition \ref{deph_sigma}, i.e.,
\begin{gather}
 \sigma_{i+1}(x)= x+\sum_s \lambda_s^{i+1}(x) \pi_{\tilde V(y_s,r_{i+1})^\perp}\ton{p_s-x} \, .
\end{gather}

With this function, we can define the sets
\begin{gather}\label{e:downward_induction:map_manifold:1}
 T_{i+1}=\sigma_{i+1} (T_i)\, ,\quad  T_{i+1}'=\sigma_{i+1}\ton{T_{i}'\setminus \bigcup_{y\in I_f^{i+1}} \B{r_{y}/6}{y}\bigcup_{y\in I_b^{i+1}}\B{r_{i+1}/6}{y}}\, .
\end{gather}

Fix any $y\in I_g^{i+1}$, and let $z\in I_g^{i}$ be such that $y\in \B{r_i}{z}$. By induction, $T_i\cap \B{10r_{i+1}}{y}\subseteq T_i\cap \B{1.5 r_i}{z}$ is the graph of a $C^1$ function over $\tilde V(z,r_{i})$. Consider the points $\cur{y_s}=I_{g}^{i+1}\cap \B{5r_{i+1}}{y}$. By construction and using an estimate similar to \eqref{eq_dfcm}, it is easy to see that $d(y_s,\tilde V(z,r_{i}))\leq c\delta r_{i+1}$, and so we can apply the estimates in Lemma \ref{lemma_vw} with $M=C_1$ by the first induction. Using condition \eqref{e:reifenberg_displacement_L2}, we obtain that for all $y_s$:
\begin{gather}\label{eq_deltasquash}
 r_{i+1}^{-1} d_H\ton{\tilde V(z,r_i)\cap \B{r_{i+1}}{y_s},\tilde V(y_s,r_{i+1})\cap \B{r_{i+1}}{y_s} }\leq c\ton{D^k_S (y_s,(1+\rho)r_{i+1})+ D^k_S (z,(1+\rho)r_i) }^{1/2}\leq c(n,\rho,C_1)\delta\, .
\end{gather}
This implies that, if $\delta(n,\rho,C_1)$ is small enough, $T_i\cap \B{10r_{i+1}}{y}$ is a graph also over $\tilde V(y,r_{i+1})$ satisfying the same estimates as in \eqref{eq_lipTi}, up to a worse constant $c$.  That is, if $\delta$ is sufficiently small, we can apply Lemma \ref{lemma_squash} and prove induction point \eqref{it_i4}.

It is important to notice that on $\B {1.5 r_{i+1}}{ y}$, the bound on the Lipschitz constant of the graph is \textit{independent} of the previous bound in the induction step by point \eqref{it_s3} in Lemma \ref{lemma_squash}.
\\

\paragraph{Points \eqref{it_i3} and \eqref{it_i5}} Points \eqref{it_i3} and \eqref{it_i5} are proved with similar methods. We briefly sketch the proofs of these two points.\\

Let $y\in T_{i+1}$, and recall the function $\psi_{i+1}\equiv 1-\sum \lambda_s$. If $\psi_{i+1}|_{\B{2r_{i+1}}{y}}$ is identically $1$, then $\sigma_{i+1}|_{\B{2r_{i+1}}{y}}=id$, and there is nothing to prove.

Otherwise, there must exist some $z'\in I_g^{i+1}\cap \B{5r_{i+1}}{y}$, and thus there exists a $z\in I_g^{i}$ such that $\B{3r_{i+1}}{y}\subseteq \B{1.5 r_{i}}{z}$.  By point \eqref{it_i5} in the induction, $T_{i}\cap \B{1.5 r_{i}}{z}$ is a Lipschitz graph over $\tilde V(z,r_{i})$. Proceeding as before, by the estimates in Lemma \ref{lemma_vw} and Lemma \ref{lemma_squash}, we obtain that $T_{i+1}\cap \B{2r_{i+1}}{y}$ is also a Lipschitz graph over $\tilde V(z,r_{i})$ with small Lipschitz constant, and that $\abs{\sigma_{i+1}(p)-p}\leq c\delta r_{i+1}$ for all $p\in T_i$. 

Moreover, $\sigma_{i+1}|_{T_i}$ is locally a diffeomorphism at scale $r_{i+1}$. From this we see that $\sigma_{i+1}$ is a diffeomorphism on the whole $T_i$.  

It is worth to remark a subtle point. In order to prove point \eqref{it_i3}, we cannot use inductively \eqref{it_i3}, we need to use point \eqref{it_i5}. Indeed, as we have seen, given any $z\in I_g^i$, then $T_i\cap \B {1.5 r_i}{z}$ is a Lipschitz graph of a function $f$ where $\abs{\nabla f}\leq c\delta$, and this $c$ is \textit{independent} of the induction step we are considering by \eqref{it_s3} in Lemma \ref{lemma_squash}. If we tried to iterate directly the bound given by \eqref{it_i3}, the constant $c$ would depend on the induction step $i$, and thus we could not conclude the estimate we want.
\\

Now we turn our attention to the items \eqref{it_i6}, \eqref{it_i7}, \eqref{it_i8}.
\paragraph{Properties of the manifolds $T_i'$} Here we want to prove the measure estimate in \eqref{eq_volFB}. The basic idea is that bad and final balls correspond to holes in the manifold $T_i$, and each of these holes carries a $k$-dimensional measure which is proportionate to the measure inside the balls. In particular, let $\B{r_{i+1}}{y}$ be a bad or a final ball. In either case, we will see that $y$ must be $\sim r_{i+1}$-close to $T_i$, which is a Lipschitz graph at scale $r_i$. This implies that $\lambda^k(\B{r_i}{y}\cap T_i) \sim r_i^k$, and thus we can bound the measure of a bad or final ball with the measure of the hole we have created on $T_i$.

In detail, point \eqref{it_i6} is an immediate consequence of the definition of $\tilde R_i$. As for point \eqref{it_i7}, if $y\in I_g^i$, then by definition, there exists $z\in I_g^{i-1}$ such that $y\in \B{r_{i-1}}{z}$, and $y\not \in R^{i-1}$. This implies that $y$ is far away from the balls we discard while building $T_{i-1}'$, in particular 
\begin{gather}
 \B{3r_i}{y}\bigcap \ton{\bigcup_{y\in I_f^{i-1}}\B{r_{y}/6}{y}\bigcup_{I_b^{i-1}}\B{r_{i-1}/6}{y} }= \emptyset
\end{gather}
This proves that $T_{i-1}\cap \B{2r_i}{y}=T_{i-1}'\cap \B{2r_i}{y}$, and in turn point \eqref{it_i7}.

In order to prove the volume measure estimate, consider that 
\begin{gather}
 T_{i}' \setminus \sigma_{i+1}^{-1}(T_{i+1}')\subseteq \ton{\bigcup_{y\in I_f^{i+1}\cup I_b^{i+1}}\B{r_{i+1}/6}{y} }\, .
\end{gather}
Note that the balls in the collection $\cur{\B{r_{i+1}/5}{y}}_{y\in I_f^{i+1}\cup I_b^{i+1}}$ are pairwise disjoint. Pick any $y\in I_b^{i+1}$, and let $z\in I_g^i$ be such that $y\in \B{r_i}{z}$. By definition, $y\in T_i'$ and $\lambda^k(\tB{r_{i+1}}{y})< \gamma_k r_{i+1}^k<10^{-k}\omega_k r^k_{i+1}$. 
Since $y\not \in R^i$, by \eqref{eq_R} $\B{r_{i+1}/6}{y}$ is disjoint from the set
\begin{gather}
 \bigcup_{s=j}^i\ton{\bigcup_{y\in I_b^s}\B {r_s/5}{y} \bigcup_{y\in I_f^s} \B{r_s/5}{y}}
\end{gather}
and thus $T_i\cap \B {r_{i+1}/6}{y}=T_i'\cap \B {r_{i+1}/6}{y}$. Moreover, $T_i\cap \B {2r_i}{z}$ is a graph over $\tilde V(z,r_i)$ with $y\in T_i'\subset T_i$, thus (for $\delta\leq \delta_0(n)$ small enough) 
\begin{gather}
\lambda^k (T_i'\cap \B{r_{i+1}/6}{y})\geq \omega_k 7^{-k} r_{i+1}^k\, .
\end{gather}

A similar estimate holds for the final balls. The only difference is that if $y\in I_f^{i+1}$, then it is not true in general that $y\in T_i$. However, by construction of the balls $\B{r_y}{y}$, using an argument similar to the one in the proof of \eqref{eq_yfinal}, we obtain that $d(y,\tilde V(z,r_i))\leq c\delta r_{i}+ r_{i+1}/10$. Given \eqref{eq_lipTi}, we can conclude
\begin{gather}
\lambda^k (T_i'\cap \B{r_{i+1}/7}{y})\geq \omega_k 10^{-k} r_{i+1}^k\, .
\end{gather}

Now it is evident from the definition of $T_i'$ that
\begin{gather}
 \lambda^k(\sigma_i^{-1}(T_i'))+ \qua{\#\ton{I_b^i}+\#\ton{I_f^i}} \omega_k (r_i/10)^k \leq \lambda^k(T_{i-1}')\, .
\end{gather}

\subsection{Volume estimates on the manifold part} Here we want to prove that for every measurable $\Omega\subseteq T_i$
\begin{gather}\label{eq_volT}
  \lambda^k(\sigma_{i+1}(\Omega))\leq \lambda^k(\Omega) + c(n,\rho,C_1) \int_{ S \cap \B{(1+\rho) r_j}{x}} D(p,4r_{i+1}) d\lambda^k(p)\, .
\end{gather}
The main applications will be with $\Omega=T_i\cap \B {2r_j}{x}$ and $\Omega=T_{i}'\cap \B {2r_j}{x}$. In order to do that, we need to analyze in a quantitative way the bi-Lipschitz correspondence between $T_i$ and $T_{i+1}$ given by $\sigma_{i+1}$. 

As we already know, $\sigma_{i+1}=id$ on the complement of the set $G=\cup_{y\in I_g^{i+1}} \B{5r_{i+1}}{y} $, so we can concentrate only on this set.

Using the same techniques as before, and in particular by Lemmas \ref{lemma_vw} and \ref{lemma_squash}, we prove that for each $y\in I_g^{i+1}$, the set $ T_{i}\cap \B{5r_{i+1}}{y}$ is a Lipschitz graph over $\tilde V(y,r_{i+1})$ with Lipschitz constant bounded by
\begin{gather}
 c(n,\rho,C_1) \ton{D(y,(1+\rho)r_{i+1}) + \sum_{z\in I_g^{i}\cap \B{5r_{i}}{y} } D(z,(1+\rho)r_i) }^{1/2}
\end{gather}
In a similar manner, we also have that $T_{i+1}\cap \B{5r_{i+1}} {y}$ is a Lipschitz graph over $\tilde V(y,r_{i+1})$ with Lipschitz constant bounded by
\begin{gather}
 c(n,\rho,C_1) \ton{\sum_{z\in I_g^{i+1} \cap \B{10 r_{i+1}}{y} } D(z,(1+\rho)r_{i+1}) + \sum_{z\in I_g^{i}\cap \B{5r_{i}}{y} } D(z,(1+\rho)r_i) }^{1/2}\, .
\end{gather}
By the bi-Lipschitz estimates in the squash Lemma \ref{lemma_squash}, we obtain that $\sigma_{i+1}$ restricted to $T_i \cap \B{5r_{i+1}}{y}$ is a bi-Lipschitz equivalence with bi-Lipschitz constant bounded by
\begin{gather}\label{eq_2lip_est}
 L(y,5r_{i+1})\leq 1+ c\ton{\sum_{z\in I_g^{i+1} \cap \B{10 r_{i+1}}{y} } D(z,(1+\rho)r_{i+1}) + \sum_{z\in I_g^{i}\cap \B{5r_{i}}{y} } D(z,(1+\rho)r_i) }
\end{gather}

In order to estimate this upper bound, we use an adapted version of \eqref{eq_estDint} and the definition of good balls to write for all $z\in I_g^{i+1}$
\begin{gather}
 D(z,(1+\rho)r_{i+1})\leq c\fint_{\tB{r_{i+1}}{z}} D(p,4r_{i+1})d\lambda^k(p)\leq c(n,\rho,C_1)r_{i+1}^{-k}\int_{\tB{r_{i+1}}{z}} D(p,4r_{i+1})d\lambda^k(p)\, ,
\end{gather}
and a similar statement holds for $z\in I_g^i$. Since by construction any point $x\in \R^n$ can be covered by at most $c(n)$ different good balls at different scales, we can bound
\begin{gather}
 L(y,5r_{i+1})\leq 1+\frac{c(n,\rho,C_1)}{r_{i+1}^k} \int_{ \tB {5 r_i}{y}} \qua{D(p,4r_{i+1}) + D(p,4r_i)}d\lambda^k(p)
\end{gather}

We can also badly estimate
\begin{gather}
 D(p,4r_{i+1}) + D(p,4r_i) \leq c(n,\rho)D(p,4r_{i})\, .
\end{gather}

Now let $P_s$ be a measurable partition of $\Omega\cap G$ such that for each $s$, $P_s\subseteq \B{5r_{i+1}}{y_s}$. By summing up the volume contributions of $P_s$, and since evidently $\lambda^k(P_s)\leq c r_{i+1}^k$, we get 
\begin{gather}
 \lambda^k(\sigma_{i+1}(\Omega))= \sum_s \lambda^k(\sigma_{i+1}(\Omega\cap P_s)) \leq \sum_s \lambda^k(P_s)\ton{1+\frac{c}{r_{i+1}^k} \int_{ \tB {5 r_i}{y_s}} D(p,4r_{i}) d\lambda^k(p)}\leq \notag\\
 \leq \lambda^k(\Omega) + c\int_{ \bigcup_{y_s\in I_g^{i+1}} \tB {5 r_i}{y_s}} D(p,4r_{i}) d\lambda^k(p)\leq \notag\\
 \leq \lambda^k(\Omega) + c(n,\rho,C_1) \int_{\B{(1+\rho)r_j}{x}\cap S} D(p,4r_{i}) d\lambda^k(p)\label{eq_arg}\, .
\end{gather}

\vspace{.5 cm}

\subsection{Estimates on the excess set}
In this paragraph, we estimate the total measure of the excess set, which is defined by
\begin{gather}
 \tilde E_T = \bigcup_{i=j}^A \bigcup_{y\in I_g^i} \tilde E(y,r_i)\, .
\end{gather}
At each $y$ and at each scale $r_i$ such that $\lambda^k (\tB{r_i}{y})\geq \gamma_k r_i^k$, we have by \eqref{eq_exrough} and \eqref{eq_estDint}
\begin{gather}
 \lambda^k(\tilde E(y,r_i))\leq c(n,\rho) r_i^k D^k_S(y,2r_i)\leq c(n,\rho) r_i^k \fint_{\B{2r_i}{y}\cap S} D^k_S(p,4r_i)d\lambda^k(p)\leq c(n,\rho)\int_{\B{2r_i}{y}\cap S} D^k_S(p,4r_i)d\lambda^k(p)\, 
\end{gather}

Now by construction of the good balls, there exists a constant $c(n)$ such that at each step $i$, each $x\in \R^n$ belongs to at most $c(n)$ many balls of the form $\cur{\B {2r_i}{y}}_{y\in I_g^i}$. Thus for each $i\geq j$, we have
\begin{gather}
 \sum_{y\in I_g^i} \lambda^k(\tilde E(y,r_i))\leq c(n,\rho) \int_{\cup_{y\in I_g^i} \B{2r_i}{y}} D^k_S(p,4r_i)d\lambda^k(p)\leq c(n,\rho)\int_{\B{2r_j}{x}\cap S} D^k_S(p,4r_i)d\lambda^k(p)\, .
\end{gather}
If we sum over all scales, we get
\begin{gather}
 \lambda^k\ton{\tilde E_T} \leq c(n,\rho)\sum_{i=j}^A\int_{\B{2r_j}{x}\cap S}D^k_S(p,4r_i)d\lambda^k(p)\, .
\end{gather}
Since $\rho=2^{-q}$, it is clear that
\begin{gather}\label{eq_volE}
 \lambda^k\ton{\tilde E_T} \leq c(n,\rho)\sum_{i=j}^A\int_{\B{2r_j}{x}\cap S}D^k_S\ton{p,2^{2-qi}}d\lambda^k(p)\leq c(n,\rho)\delta r_j^k\, ,
\end{gather}
since the sum in the middle is clearly bounded by \eqref{eq_sum2^alpha}.  \\

This estimate is exactly what we want from the excess set.
\vspace{.5 cm}

\subsection{Completion of the weak upper bounds}  
By adding \eqref{eq_volT}, with $\Omega\equiv\sigma_{i+1}^{-1}(T_{i+1}')$, and \eqref{eq_volFB}, we prove that for all $i=j,\cdots,A+1,...$
\begin{gather}
 \lambda^k(T_{i+1}')+\qua{\#\ton{I_b^i}+\#\ton{I_f^i}} \omega_k (r_i/10)^k \leq \lambda^k(T_{i}')+ c(n,\rho,C_1) \int_{ \B{4r_j}{y}\cap S_{\bar r}  } D(p,2r_{i}) d\lambda^k(p)\, .
\end{gather}
Adding the contributions from all scales, by \eqref{eq_sum2^alpha} we get 
\begin{align}
 &\lambda^k(T_{i+1}')+\sum_{s=j}^{i}\qua{\#\ton{I_b^s}+\#\ton{I_f^s}} \omega_k (r_s/10)^k \leq \notag \\
 &\leq \lambda^k(T_j\cap \B{2r_j}{x})+ c(n,\rho,C_1) \sum_{s=j}^{i} \int_{ \B{2r_j}{x}\cap S_{\bar r} } D(p,2r_{s}) d\lambda^k(p)\leq\notag \\
 &\leq \lambda^k\ton{T_j\cap \B{2r_j}{x}}\qua{1+ c(n,\rho,C_1)\delta^2}\,\label{eq_volTFB} \, ,
\end{align}
where in the last line we estimated $\lambda^k\ton{T_j\cap \B{2r_j}{x}}\sim r_j^{k}$, since $T_j$ is a $k$-dimensional subspace, and we bounded the sum using \eqref{eq_sum2^alpha}.

In the same way, we can also bound the measure of $T_i$ by
\begin{gather}
 \lambda^k(T_{i+1})\leq \lambda^k\ton{T_j\cap \B{2r_j}{x}}\qua{1+ c(n,\rho,C_1)\delta^2}\leq c(n)r_j^k\,\label{eq_volTFB_2} .
\end{gather}

\vspace{.5 cm}

\paragraph{Upper estimates for $\lambda^k$.}
Given the definition of $\tSr$, all the balls $\B{r_y}{y}$ inside this set have $r_y\geq r_A$. Thus by \eqref{eq_sub}
\begin{gather}
 \tilde B = \tB{r_j}{x} \subseteq \tilde R^A\, .
\end{gather}
In particular, this and the estimates in \eqref{eq_volE} and \eqref{eq_volTFB} imply
\begin{gather}
 \lambda^k(\tilde B) \leq \sum_{s=j}^{A}\qua{\gamma_k\#\ton{I_b^s} + 2\omega_k \#\ton{I_f^s} } (r_s)^k + \lambda^k(\tilde E_T) \leq \notag\\
 \leq C_3(n) \sum_{s=j}^{i}\qua{\#\ton{I_b^s} + \#\ton{I_f^s} } \omega_k (r_s/10)^k+ \lambda^k(\tilde E_T)\leq C_3(n) (1+c(n,\rho,C_1)\delta) r_j^k \, .
\end{gather}
In this last estimate, we can fix $C_1(n)=2C_3(n)\leq 40^n$, and $\rho(n,C_1)$ according to Lemma \ref{lemma_alpharho}. Now, it is easy to see that if $\delta(n,\rho,C_1)$ is sufficiently small, then
\begin{gather}\label{eq_volR}
 \lambda^k(\tilde B) \leq C_1(n) r_j^k \, ,
\end{gather}
which finishes the proof of the downward induction, and hence the actual ball estimate \eqref{eq_lambda_obj2}.

\vspace{.7 cm}
\subsection{\texorpdfstring{Rectifiability, $W^{1,p}$ and improved measure bounds}{Rectifiability, W(1,p) and improved measure bounds}}\label{sec_impro}
We have now proved a mass bound for the set $S$.  In this subsection, we wish to improve this mass bound, as well as prove the rectifiability conditions on $S$.  At several stages we will have to repeat the arguments of the upward and downward inductions, which we will only sketch since the arguments will be almost verbatim, though in some cases technically much easier.

\paragraph{Rectifiability and $W^{1,p}$ estimates}
Now that we have proved the bound \eqref{eq_lambda_obj2}, we can use a construction similar to the one just described to sharpen the upper estimate and prove rectifiability of $\lambda^k|_S$. The main difference with the previous case is that we do not need to be concerned any more with any fixed covering $\B{r_y}{y}$ of our set, and so at every step we can estimate directly the measure $\lambda^k|_S$ of a whole Euclidean ball $\B {r}{z}$ without the need to limit our estimate to $\tB{r}{z}$. For the same reason, we do not need to introduce the subspaces $\tilde V(z,r)$ and the sets $\tilde E, \ \tilde R$, we will only deal with $V(z,r),E,R$. Moreover, since we don't have to stop our construction at any positive scale $r_y$, we do not need to introduce and study the set of final balls $I_f^i$. As a consequence, the construction we need here is technically less involved than the one used before. 

However, as opposed to the previous construction, we are concerned about what happens at an infinitesimal scale, and in particular we want to have uniform estimates also for the limit $\bar r\to 0$ of various quantities.

\vspace{3mm}

Fix any $x\in \R^n$ and $r>0$ such that $\B r x \subseteq \B 1 0$. We are going to prove \eqref{eq_lambda_obj}, i.e.
\begin{gather}
 \lambda^k\ton{S\cap \B r x} \leq (1+\epsilon) \omega_k r^k\, ,
\end{gather}
where $\epsilon>0$ is the arbitrary constant chosen at the beginning of Theorem \ref{t:reifenberg_W1p_holes}.

For convenience and wlog, we assume $x=0$ and $r=1$. Using the same technique as before, we build a sequence of smooth maps $\cur{\sigma_i}_{i=0}^\infty$ on $\R^n$ and manifolds $T_i$ such that
\begin{enumerate}
 \def\theenumi{\roman{enumi}}
 \item $\sigma_0=id$, $T_0=V(x,r)=V(0,1)\subseteq \R^n$ and $\phi_i = \sigma_i\circ \sigma_{i-1}\circ \cdots \sigma_0$,
 \item $T_i=\sigma_i(T_{i-1})=\phi_i(T_0)$ is a smooth $k$-dimensional submanifold of $\R^n$,
 \item \label{it__w1p_i3}for $y\in T_{i}$
 \begin{gather}\label{eq_distdelta3}
  d(\sigma_{i+1}(y),y)\leq c\delta r_{i+1}\, ,
 \end{gather}
and $\sigma_{i+1}|_{T_{i}}$ is a diffeomorphism. In a similar way for all $y\in T_0$, $d(\phi_{i}(y),y)\leq c\delta $ and $\phi_i|_{T_0}$ is a diffeomorphism, 
 \item for every $y\in T_i$, $T_i\cap \B{2 r_i}{y}$ is the graph over some $k$-dimensional affine subspace of a smooth function $f$ satisfying
  \begin{gather}
  \frac{\norm{f}_{\infty}}{r_i} + \norm{\nabla f}_\infty \leq c\delta\, ,
 \end{gather}
\end{enumerate}

In order to do so, we define inductively on $i=0,\cdots,\infty$ a sequence of sets $E(y,r_i)$, $I^i_g,I^i_b,R^i$ and manifolds $T_i'\subseteq T_i$ such that
\begin{enumerate}
 \setcounter{enumi}{4} \def\theenumi{\roman{enumi}}
 \item for every $i\geq 1$ and all $y\in I_g^i$, the set $T_{i-1}'\cap \B{1.5 r_i}{y}$ is a Lipschitz graph over the plane $V(y,r_{i})$ with $d(y,V(y,r_{i}))\leq c\delta r_{i}$
 \item for all $i$, we have the inclusion
 \begin{gather}\label{eq_sub_W}
S\cap \B 1 0 \subseteq \B{r_{i}/2}{T_{i-1}'}\cup R^{i-1}
 \end{gather}
  \item  for every $i\geq j+1$ and all $y\in I_g^i$, the set $T_{i-1}'\cap \B{2 r_i}{y}$ is a Lipschitz graph over the plane $\tilde V(y,r_{i})$ with $d(y,\tilde V(y,r_{i}))\leq c\delta r_{i}$ ,
 \item  we can estimate
 \begin{gather}\label{eq_volB_W}
  \lambda^k(\sigma_i^{-1}(T_i'))+ \#\ton{I_b^i} \omega_k (r_i/10)^k \leq \lambda^k(T_{i-1}')\, .
 \end{gather}
\end{enumerate}

Up to minor modifications (actually simplifications), all these properties are proved in the same way as in the downward induction from the previous subsection. \\

The key extra-property we need is some form of control over the $W^{1,p}$ norm of $\phi_i$. In particular we will prove inductively that for all $j\in \N$
\begin{gather}
 \norm{\nabla \phi_j}_{L^p}^p = \fint_{T_0\cap \B{2}{0}} \abs{\nabla \phi_j}^p d\lambda^k \leq 1+c(n)2^p \delta^2\, ,
\end{gather}
where here $\lambda^k$ is also Lebesgue measure on $\R^k$.

As we will see in \eqref{eq_deltap}, in order for this estimate to work, we will need to choose $\delta(n,\rho,C_1,p)$ small enough, and in particular $\delta\to 0$ as $p\to \infty$. 
\vspace{.5 cm}

\paragraph{Proof of the first points}
As mentioned before, the proof of items (i)-(viii) can be carried out in the same way as in Section \ref{sec_IIind_pre}, and we take from this section also the definitions of the sets $I_g^i$, $I_b^i$, $E(x,r)$ and $R^i$, up to replacing $\tB r y$ with $\B r y$ and $\tilde V(y,r)$ with $V(y,r)$. Recall also that in this case we have no final balls, so we can just assume that $I_f^i=\emptyset$ for all $i$, and that in this case we have $x=0$, $r=1$ and so $j=0$.

For convenience, we recall the definition of the excess and remainder sets:
\begin{gather}
 E(y,r_i)= \B{r_i}{y} \setminus \B{r_{i+1}/4}{V(y,r_i)}\, ,\quad  R^i=\bigcup_{s=0}^i\ton{\bigcup_{y\in I_b^s}\B {r_s}{y} \bigcup_{y\in I_g^s}E(y,r_s) }\, . 
\end{gather}
We also introduce the notation
\begin{gather}
 E^i = \bigcup_{s=0}^i\bigcup_{y\in I_g^s}E(y,r_s)  \, , \quad R^i = E^i \bigcup_{s=0}^i \bigcup_{y\in I_b^s}\B {r_s}{y} \, .
\end{gather}

We briefly sketch again the main steps in the construction. Assuming wlog that $\B 1 0$ is a good ball, i.e., that $\lambda^k(\B 1 0)\geq \gamma_k$, we first estimate the excess set on this ball. Since this set is the set of points which are some definite amount away from the best plane $V(0,1)$, the definition of $D$ immediately gives the following estimate, similar to \eqref{eq_exrough}:
\begin{align}\label{eq_exrou2}
\lambda^k\big(E(0,1)\big) (\rho/5)^2\leq D^k_S(0,1)\, .
\end{align}

Then we cover the non-excess part with a Vitali-type covering by balls $\B {\rho}{x_i}_{i\in I}$ centered on the plane $V(0,1)$. We classify the balls in this covering into good and bad balls, according to how much mass they carry. $\B r x$ is a good ball if $\lambda^k (S\cap \B r x)\geq \gamma_k r^k$, otherwise it's a bad ball.

A good ball carries enough measure to apply Lemma \ref{lemma_vw}, and compare the best subspace $V(0,1)$ with the new best subspace $V(x_i,\rho)$. 

We set $\sigma_1$ to be the map defined in \ref{deph_sigma}, i.e., an interpolation among all the projections onto $\cur{V(x_i,\rho)}_{i\in I_g^1}$, and we also set $T_1=\sigma_1(T_0)=\sigma_1(V(0,1)\cap \B {1+c\delta} 0)$. By the squash lemma \ref{lemma_squash}, $\sigma_1|_{T_0}$ is a smooth diffeomorphisms, and $T_1$ is locally at scale $r_1=\rho^1$ the Lipschitz graph of a function with small Lipschitz bounds. Moreover, $\sigma_1$ is a bi-Lipschitz equivalence with quantitative estimates on the Lipschitz constant. In particular, we get that for all $y\in \B 1 0\cap T_0$, the following version of the estimates in \eqref{eq_2lip_est} holds:
\begin{gather}
 L(y,5\rho)=\max\cur{\norm{\nabla \sigma_1}_{L^\infty(T_0\cap \B {5\rho}{y})  },\norm{\nabla \sigma_1}_{L^\infty\ton{\sigma_1\ton{T_0\cap \B {5\rho}{y}}}  } }\leq 1+ c\ton{\sum_{z\in I_g^{1} \cap \B{10 \rho }{y} } D(z,\rho) + D(0,1) }
\end{gather}

In order to keep track of the measure inside the bad balls, we define the manifold ``with holes''
\begin{gather}
 T_{1}=\sigma_{1} (T_0)\, ,\quad  T_{1}'=\sigma_{1}\ton{T_{0}\setminus \bigcup_{z\in I_b^{1}}\B{\rho/6}{z}}\, .
\end{gather}
Since $V(0,1)$ is a $k$-plane, and by definition of bad balls, each ``hole'' in the manifold $T_1'$ carries more measure than the corresponding bad ball which created it, giving us the estimate
 \begin{gather}\label{eq_volbad3}
  \lambda^k(\sigma_1^{-1}(T_1'))+ \lambda^k \ton{\bigcup_{z\in I_b^1} S\cap \B {\rho}{z}}\leq \lambda^k(\sigma_1^{-1}(T_1'))+ \#\ton{I_b^1} \omega_k (\rho/9)^k \leq \lambda^k(T_{0})\, .
 \end{gather}
 Now we repeat the construction inductively on the scales $r_2=\rho^2,r_3,\cdots$, and we obtain all the desired properties. Moreover, by summing \eqref{eq_exrou2} and its iterations, we obtain the following estimate for the total excess set:
 \begin{gather}\label{eq_extot3}
  \lambda^k\ton{E^\infty} \leq c(n,\rho)\delta^2 \, ,
 \end{gather}
which is the equivalent of \eqref{eq_volE}.

We conclude by noting the following: $\sigma|_i(z)=id$ for all $z\in T_0 \cap \ton{\cup_{j=0}^i\bigcup_{z\in I_b^j}\B{r_i/2}{z}}$. In other words, once a hole is created, it never changes. This implies that the iterations of \eqref{eq_volbad3} lead to 
\begin{gather}\label{eq_volRinfty}
 \lambda^k\ton{T'_i \cap \B {1+c\delta}{0}} +9^{-k}\gamma_k^{-1}\sum_{j=0}^i\sum_{y\in I_b^j} \lambda^k(S\cap \B{r_j}{y})\leq \lambda^k(T_i\cap \B {1+c\delta}{0})\, .
\end{gather}

\paragraph{$W^{1,p}$ estimates}
As part of the downward induction in the proof of (i)-(viii), let us define the maps $\tau_{i,j}=\sigma_j\circ\sigma_{j-1}\circ\cdots \circ\sigma_i$, so that $\tau_{i,j}(T_i)=T_{j+1}$ and $\tau_{0,j}=\phi_j$. We will prove inductively for $i=j,\cdots,0$ that for all $x\in T_i$, 
\begin{gather}\label{eq_w1pC}
 r_i^{-k} \int_{\B{r_i}{x}\cap T_i }\abs{\nabla \tau_{i,j}}^p d\lambda^k \leq 2C_1(n)\, .
\end{gather}
Here the integration is simply the integration on a smooth $k$-dimensional subset wrt the $k$-dimensional Lebesgue measure $\lambda^k$, and the gradient of the functions $\sigma_i$ and $\tau_{i,j}$ is the restriction of the gradient in $\R^n$ to the corresponding manifold $T_i$. 

Suppose that the statement is true for $i+1$. Consider a covering of this ball by $\cur{\B {r_{i+1}}{y_s}}$, then by the chain rule
\begin{gather}
\int_{\B{r_i}{x}\cap T_i }\abs{\nabla \tau_{i,j}}^p d\lambda^k(y) \leq \sum_s \int_{\B{r_{i+1}}{y_s}\cap T_i }\abs{\nabla \tau_{i+1,j}|_{\sigma_i(y)}}^p\abs{\nabla \sigma_i(y)}^p d\lambda^k(y)\leq C(n,p,\rho)\, .
\end{gather}
This gives us a first rough estimate. 

In order to obtain a better estimate, we will prove by induction on $s=i,\cdots,j$ that
\begin{gather}\label{eq_Tind2}
 \int_{\B{r_i}{x}\cap T_i }\abs{\nabla \tau_{i,j}}^p d\lambda^k(y) \leq \int_{B_s}\abs{\nabla \tau_{s,j}}^p d\lambda^k(y) + c(n,\rho)2^p \sum_{t=i}^s \int_{\B {10r_i}{x}\cap S} D(p,10r_{t})d\lambda^k(p)\, ,
\end{gather}
where $B_s\equiv \tau_{i,s-1}(T_i\cap \B{r_i}{x})\subseteq T_s\cap \B{(1+c\delta)r_i}{x}$.  Indeed, suppose that the statement is true for $s$, then we have
\begin{gather}
 \int_{B_s}\abs{\nabla \tau_{s,j}}^p d\lambda^k(y) =\int_{B_s}\abs{\nabla \tau_{s+1,j}|_{\sigma_s(y)}}^p\abs{\nabla \sigma_i(y)}^p d\lambda^k(y) \, .
\end{gather}
Since $\sigma_i$ is a diffeomorphism, we can change the variables and write
\begin{gather}
 \int_{B_s}\abs{\nabla \tau_{s,j}}^p d\lambda^k(y) =\int_{B_{s+1}}\abs{\nabla \tau_{s+1,j}|_{y}}^p\abs{\nabla \sigma_s|_{\sigma_s^{-1}(y)}}^p \det{J(\sigma_s)} \ d\lambda^k(y) \, .
\end{gather}
Now consider the partial covering of this set given by $\cur{\B{5r_s}{y_t}}_{y_t\in I_g^{s}}$. As we have seen before, outside of this set, $\sigma_{s}$ is the identity, so we don't need to make any estimates on it.

Arguing in a manner verbatim to the proof of \eqref{eq_volT}, we can prove bi-Lipschitz estimates for $\sigma|_s$. In particular, we can use the definition of good balls and the bounds in Lemma \ref{lemma_vw} to estimate in a quantitative way the distance between best subspaces at nearby points and scales. This allows us to use the squash Lemma \ref{lemma_squash} and prove that $\sigma_s$ restricted to $T_s\cap \B{5r_{s+1}}{y_t}$ is a bi-Lipschitz map with 
\begin{gather}
 \norm{\abs{\nabla \sigma_s|_{T_s}} - 1}_{L^\infty(T_s\cap \B{5r_{s+1}}{y_t})}\leq  \frac{c(n,\rho)}{r_{s}^k} \int_{ \B {5 r_s}{y_t}\cap S} D(p,10r_{s})d\lambda^k(p)\, .
\end{gather}
Let $\cur{P_t,Q}$ be a measurable partition of $B_{s}$ with $P_t\subseteq \B{5r_{s+1}}{y_t}$ and $Q\subseteq \cap \B{5r_{s+1}}{y_t}^c$, then $\cur{\sigma_s(P_t),\sigma_s(Q)}$ is a measurable partition of $B_{s+1}$ with $\sigma_s(Q)=Q$. So we get 
\begin{gather}
\int_{B_s}\abs{\nabla \tau_{s,j}}^p d\lambda^k(y) =\int_{Q}\abs{\nabla \tau_{s,j}}^p d\lambda^k(y) +\sum_t\int_{\sigma(P_t)}\abs{\nabla \tau_{s+1,j}|_{y}}^p\abs{\nabla \sigma_s|_{\sigma_s^{-1}(y)}}^p \det{J(\sigma_s)}\ d\lambda^k(y)\leq \notag\\
 \leq \int_{Q}\abs{\nabla \tau_{s+1,j}}^p d\lambda^k(y)+ \sum_t \int_{\sigma(P_t)}\abs{\nabla \tau_{s+1,j}}^p d\lambda^k \ \ton{1+\frac{c}{r_{{s+1}}^k} \int_{ \B {5 r_s}{y_t}\cap S} D(p,10r_{s})d\lambda^k(p)}\, .
\end{gather}
Note that these estimates are basically the same as the ones in \eqref{eq_arg}.

By the first induction, we have the upper bound $\int_{P_s}\abs{\nabla \tau_{s+1,j}|_{y}}^p \leq c(n)r_{s+1}^k$, and thus
\begin{gather}
 \int_{B_s}\abs{\nabla \tau_{s,j}}^p d\lambda^k(y) \leq \int_{B_{s+1}}\abs{\nabla \tau_{s+1,j}}^p d\lambda^k(y) + c(n,\rho)2^p\sum_t \int_{ \B {5 r_{s+1}}{y_t}\cap S} D(p,10r_{s})d\lambda^k(p)\, .
\end{gather}
Since all points in $\R^n$ are covered at most $c(n,\rho)$ times by $\cur{\B {5 r_{s+1}}{y_t}}_t$, we can simply estimate
\begin{gather}
 \int_{B_s}\abs{\nabla \tau_{s,j}}^p d\lambda^k(y) \leq \int_{B_{s+1}}\abs{\nabla \tau_{s+1,j}}^p d\lambda^k(y) + c(n,\rho)2^p\int_{ \B {10r_i}{x}\cap S} D(p,10r_{s})d\lambda^k(p)\, .
\end{gather}

This proves \eqref{eq_Tind2}. In order to complete the first induction, observe that
\begin{gather}
  \int_{\B{r_i}{x}\cap T_i }\abs{\nabla \tau_{i,j}}^p d\lambda^k(y) \leq \int_{B_s}\abs{\nabla \tau_{j,j}}^p d\lambda^k(y) + c(n,\rho)2^p \sum_{t=i}^j \int_{\B {2r_i}{x}\cap S} D(p,10r_{t})d\lambda^k(p)\, ,
\end{gather}
Since $\tau_{j,j}$ is the identity on $T_j$, and given the measure bounds \eqref{eq_volTFB_2}, we get the result if we choose $\delta(n,\rho,p)$ such that the rhs is small enough, and in particular
\begin{gather}\label{eq_deltap}
 c(n,\rho)2^p\sum_{j=i}^A\int_{ S\cap \B{2r_i}{x}} D(p,10r_{i-1})d\lambda^k(p)\leq c(n,\rho)2^p\delta^2 \leq C_1(n)\, .
\end{gather}

Note that, once the improved estimate \eqref{eq_lambda_obj} is proven (which will be done in the next subsection) it will likewise be possible to improve this $W^{1,p}$ estimate to
\begin{gather}\label{eq_w1peps}
 \frac{1}{\omega_k r_i^{k}} \int_{\B{r_i}{x}\cap T_i }\abs{\nabla \tau_{i,j}}^p d\lambda^k \leq 1+\epsilon\, .
\end{gather}
\vspace{.5 cm}

\paragraph{Improved measure estimates}
Consider the maps $\phi_j:T_0\to T_{j+1}$. Since $\phi_j$ has uniform $W^{1,p}$ estimates and by point \eqref{it__w1p_i3}, then $\phi_j\to \phi_\infty$ uniformly, with
\begin{gather}
 \int_{T_0\cap \B {2}{0}} \abs{\nabla \phi_\infty}^p \leq 2C_1(n)\, .
\end{gather}
By \eqref{eq_sub_W}, we have the inclusion for all $i$:
\begin{gather}\label{eq_sub_2}
 S\subseteq E_i \bigcup \B{r_{i+1}/2}{T'_i} \bigcup_{j=0}^i \bigcup_{y\in I_b^j}\B{r_j}{y}  \, , \quad \Longrightarrow \quad S\cap \B 1 0 \subseteq E_\infty \bigcup \ton{T'_\infty \cap \B 1 0 }\bigcup_{j=0}^\infty \bigcup_{y\in I_b^j}\B{r_j}{y}\, .
\end{gather}
By \eqref{eq_volRinfty} and \eqref{eq_extot3}, we obtain that
\begin{gather}
 \lambda^k \ton{S\cap \B 1 0} \leq \lambda^k(T_\infty)+\lambda^k(E_\infty)\leq \lambda^k(T_\infty)+c\delta^2\, .
\end{gather}

In order to give better estimates on $T_\infty$, note that $T_\infty \cap \B 1 0 \subseteq \phi(T_0\cap \B {1+c\delta}{0})$. Define the function $f:T_0\to \R$ by
\begin{gather}
 f(x)= \sum_{i=0}^\infty \sup_{B_{r_i}(\phi_i(x))\cap T_i} \abs{\abs{\nabla \sigma_i}- 1}\, .
\end{gather}
By the usual double-induction argument, see for example the proofs of \eqref{eq_volR} and \eqref{eq_w1pC}, we can estimate
\begin{gather}
 \int_{T_0\cap \B {1+c\delta}{0}} f(x) d\lambda^k \leq c\delta^2\, .
\end{gather}
Define the sets $U_a\subseteq T_0$ by
\begin{gather}
 U_a=\cur{x\in T_0 \cap \B {1+c\delta}{0}\ \ s.t. \ \ f(x)>a}\, .
\end{gather}
By simple $L^1$ estimates, we know that $\lambda^k(U_a)\leq \frac{c\delta^2}{a}$, and by Lemma \ref{lemma_monti} for $p>k$ we get
\begin{gather}
 \lambda^k(\phi_\infty(U_a)) \leq C(n,p) \norm{\nabla \phi_\infty }_{L^p}^{k} \lambda^k(U_a)^{\frac{p-k}{p} }\leq C \ton{\frac{c\delta^2}{a}}^{1-k/p}\, .
\end{gather}

Note that, on the complement of $U_a$, $\phi_\infty$ is a Lipschitz function with $\norm{\nabla \phi_\infty}_\infty\leq e^{a}+c\delta$. Indeed, take any $x,y\in U_a^c$, and let $i$ be such that $r_{i+1}/2<\abs{x-y}\leq r_i/2$. We prove that
\begin{gather}\label{eq_lipf}
 \abs{\phi_\infty(x)-\phi_\infty(y)}\leq (e^a+c\delta)\abs {x-y}\, .
\end{gather}

By definition, $\phi_\infty(z)=\tau_{i\infty}(\phi_i(z))$, and by \eqref{eq_distdelta3}, we easily get
\begin{gather}
 \abs{\tau_{i\infty}(\phi_i(x))-\phi_i(x)}+\abs{\tau_{i\infty}(\phi_i(y))-\phi_i(y)} \leq c\delta r_i \, .
\end{gather}
Moreover, the definition of $f$ immediately implies a uniform Lipschitz condition on $\phi_i$. Indeed, 
\begin{gather}
 \norm{\nabla \phi_i}_{L^\infty (\B {r_i}{x} )} \leq \prod_{s=0}^i \norm{\nabla \sigma_s}_{L^\infty (\B {r_s}{\phi_s(x)} )}\leq \exp\ton{\sum_{i=0}^s\log\ton{ \norm{\nabla \sigma_s}_{L^\infty (\B {r_s}{\phi_s(x)} )} }  } \leq e^{a}\, .
\end{gather}
Thus we get
\begin{gather}
 \abs{\psi(x)-\psi(y)}\leq c\delta \abs{x-y}+e^a\abs{x-y}\, ,
\end{gather}
as desired.

By choosing $a(\epsilon)$ sufficiently small, and $\delta(n,\rho,p,\epsilon)$ sufficiently small as well, we get
\begin{gather}\label{eq_588}
 \lambda^k(T_\infty\cap \B {1} {0}) \leq \lambda^k \ton{\phi_\infty(U_a) \cup \phi_\infty(U_a^C) }\leq \ton{1+\frac{1}{2}\epsilon}\omega_k \, .
\end{gather}

Summing up all the estimates, we prove the sharpened upper bound
\begin{gather}
 \lambda^k(S\cap \B 1 0)\leq \lambda^k(E_\infty)+\lambda^k(\phi_\infty(T_0\cap \B{1+c\delta}{0}))\leq (1+\epsilon)\omega_k\, .
\end{gather}

\vspace{.5 cm}

\paragraph{Rectifiability}
As for the rectifiability, we can restrict ourselves to $S^\star$ and consider a covering of this set made by balls $\cur{\B{r_x}{x}}_{x\in S^\star}$ such that $\lambda^k(\B{r_x}{x}\cap S)\geq 2^{-k-1}\omega_k r_x^k$.

If we prove that for all such balls there exists a subset $G\subseteq S\cap \B{r_x}{x}$ which is rectifiable and for which $\lambda^k(G)\geq \lambda^k(\B{r_x}{x}\cap S)/2$, then an easy covering argument gives us the rectifiability of all of $S$.\\

By scale invariance, we assume for simplicity $x=0$ and $r_x=1$. As we have seen before,
\begin{gather}
 S\subseteq R^\infty \cup T_\infty\, .
\end{gather}
Since $T_\infty$ is the image of a $W^{1,p}$ map with $p>k$, this set is rectifiable by Lemma \ref{lemma_w1p_rec}. Moreover, by \eqref{eq_volRinfty} and \eqref{eq_588}, we know that
\begin{gather}
\sum_{j=0}^i\sum_{y\in I_b^j} \lambda^k(S\cap \B{r_j}{y})\leq 10^k \gamma_k 
\end{gather}
 Given the estimates on the excess set given in \eqref{eq_extot3}, we obtain the lower bound
\begin{gather}
 \lambda^k(S\cap T_\infty) \geq \lambda^k(S) - \lambda^k(R^\infty) \geq 2^{-k-1}\omega_k  - \gamma_k 10^k -c\delta\geq 2^{-k-2}\omega_k \, ,
\end{gather}
which therefore completes the proof.
\vspace{.5 cm}

\section{Proof of Theorems \ref{t:reifenberg_W1p} and \ref{t:reifenberg_W1p_discrete}}\label{sec_proofII}
As it will be clear, the proofs of these theorems are simple modifications of the proof of Theorem \ref{t:reifenberg_W1p_holes}, and actually from the technical point of view they are a lot simpler. For this reason, we will simply outline them, pointing out the main differences needed in these cases. 

\subsection{\texorpdfstring{Proof of Theorem \ref{t:reifenberg_W1p}: The $W^{1,p}$-Reifenberg}{Proof of Theorem \ref{t:reifenberg_W1p}: The W-(1,p)-Reifenberg}}
The proof of this theorem is almost a corollary of Theorem \ref{t:reifenberg_W1p_holes} and the classic Reifenberg theorem \ref{t:classic_reifenberg}.  In addition to the upper volume bound proved in Theorem \ref{t:reifenberg_W1p_holes}, we also need to prove a lower volume bound on each ball.  

In fact, one could just trace through the argument of Theorem \ref{t:reifenberg_W1p_holes}, and see inductively that there exists no bad balls to conclude this.  Instead, we will use a standard argument (see for example \cite[lemma 13.2]{davidtoro}) to prove the lower bound on $\lambda^k(S\cap \B r x)$ given in \eqref{eq_lambda_lower_upper} directly by using only the uniform Reifenberg condition.


\begin{lemma}
 Under the assumptions of Theorem \ref{t:classic_reifenberg}, for all $x\in S$ such that $\B{r}{x}\subseteq \B 1 0$, 
 \begin{gather}
  \lambda^k(S\cap \B r x ) \geq (1-c\delta)\omega_k r^k\, .
 \end{gather}
\end{lemma}
\begin{proof}
 By scale invariance, we assume $x=0$ and $r=1$. The classic Reifenberg theorem proves that there exits a bi-\hol continuous map $\phi:L\to \R^n$ where $L$ is a $k$-dimensional plane and
 \begin{enumerate}
  \item $\abs{\phi(x)-x}\leq c\delta$ for all $x\in L$
  \item $\phi(x)=x$ for $\abs x \geq 1+c\delta$
  \item $S\cap \B 1 0=\phi(L')$, where $\B {1-c\delta}{0}\cap L \subseteq L'\subseteq \B {1+c\delta}{0}\cap L$. 
 \end{enumerate}
Now let $f=\pi_L\circ\phi:L\to L$. This map is continuous and it is the identity outside $\B{1+c\delta}{0}$, and thus by topological reasons (degree theory) it is also surjective from $L$ to itself.

In particular, the set $A=\B{1-3c\delta}{0}\cap L$ is contained in the image of $f$. By the uniform Reifenberg condition, $\pi_L^{-1}(A)\cap S\subseteq \B{1-2c\delta}{0}$, and by the properties of $\phi$, $f^{-1}(A)=\phi^{-1}(\pi_L^{-1}(A)\cap S)\subseteq \B{1-c\delta}{0}$. Thus $\phi(f^{-1}(A))\subseteq S$.

Now, since $\pi_L$ has Lipschitz constant $1$, by a standard result (see \cite[2.10.11]{Fed})
\begin{gather}
 \lambda^k(S)\geq \lambda^k(\phi(f^{-1}(A))\geq \lambda^k(\pi_L\circ \phi(f^{-1}(A)))=\lambda^k(A)\geq (1-3c\delta)^k \omega_k r^k\, .
\end{gather}
For $\delta$ small enough, we have the thesis.


\end{proof}

The upper bound on $\lambda^k(S)$ and the rectifiability are direct consequences of Theorem \ref{t:reifenberg_W1p_holes}. The only thing left to prove is the statement about $W^{1,p}$ bounds on the map. This can be obtained by simple modifications (actually simplifications) in the proof of Theorem \ref{t:reifenberg_W1p_holes}.

As before, we denote by $V(x,r)$ (one of) the $k$-dimensional subspace that minimizes $\int_{\B r x } d(y,V)^2 $, and we set $L(x,r)$ to be one of the $k$-dimensional subspaces satisfying the Reifenberg condition, i.e., a $k$-dimensional subspace such that $d_H(S\cap \B r x,L\cap \B r x )<\delta r$.

First of all, note that for any $x\in S$, $r>0$, $\lambda^k(\B{r}{x})>\omega_k r^k /2$, and so there are no bad balls in our covering. Thus, by Lemma \ref{lemma_LV}, we can conclude that
\begin{gather}
 d_H(V(x,r)\cap \B r x , L(x,r)\cap \B r x)\leq c\delta\, .
\end{gather}
As a consequence, for $\delta(n,\rho)$ small enough, all the excess sets at all scales are empty. Indeed, by \eqref{eq_E}, 
\begin{gather}
  E(x,r)= \B{r}{x} \setminus \B{\rho r /3}{V(x,r)}\cap S\subseteq \B{r}{x} \setminus \B{(\rho/3  - c\delta )r}{L(x,r)}\cap S 
  =\emptyset\, .
\end{gather}

Proceeding with the same construction as in Section \ref{sec_impro}, we obtain a sequence of maps $\phi_j:T_0\to \R^n$, $\phi_j(T_0)\equiv T_{j+1}$ converging in $W^{1,p}$ to some $\phi_\infty$ such that for all $i$, $S\subseteq \B{r_i}{T_i}\subseteq \B{r_{i-1}}{T_{i-1}}$, thus proving that $S\subseteq T_\infty\equiv \phi_\infty(T_0)$. 

\paragraph{$W^{1,p}$ estimates for the inverse} Here we want to prove that the map $\phi^{-1}_\infty:S\to T$ is also a $W^{1,p}$ map with bounds. Note that $S$ equipped with the Euclidean distance and the $k$-dimensional Hausdorff measure is a \al regular metric measure space, in the sense that there exists a $C_1(n)$ such that for all $x\in S\cap \B 1 0$ and $r\leq 1$,
\begin{gather}\label{eq_alregC1}
 C_1^{-1} r^k\leq \lambda^k(S\cap \B r x) \leq C_1 r^k\, ,
\end{gather}
where $\B r x$ is the usual Euclidean ball in $\R^n$. On such spaces there are several methods of defining the space of $W^{1,p}$ maps, for instance as the closure of the lipschitz functions under the $W^{1,p}$ norm.  For rectifiable spaces all such definitions are classically understood to be equivalent.\\

Note that since $S$ is rectifiable, we can use the integration by substitution to write 
\begin{gather}
 \int_S f(z) d\lambda^k(z) = \int_{\phi^{-1}_\infty(S)} f\ton{\phi_\infty(x)} J(\phi_\infty)|_x d\lambda^k(x)\, ,
\end{gather}
which will allow us to easily study integrals on our rectifiable spaces.

The next Lemma tells us that the mapping $\phi_\infty^{-1}$ is approximated by a sequence of Lipschitz maps whose gradients form a Cauchy sequence in $L^p$ and have our desired estimates.  In particular, the following Lemma will finish our proof of the $W^{1,p}$-Reifenberg result.  For convenience, we introduce the notation $\psi\equiv \phi_\infty^{-1}$:

\begin{lemma}\label{lemma_w1p_inv}
 There exists a sequence of functions $\psi_t$ such that
 \begin{enumerate}
  \item $\psi_t$ are Lipschitz functions in $\R^n$, with Lipschitz constant bounded by $c(n)t$,
  \item if $R_t=\cur{z\in S \ \ s.t. \ \ \psi_t(z)\neq \psi(z)}$, then $\lambda^k(R_t)\to 0$,
  \item $\psi_t$ converges uniformly to $\psi$,
  \item the sequence $\cur{\nabla \psi_t}$ is a Cauchy sequence in $L^p(S)$.
 \end{enumerate}
 Moreover, there exists a $g\in L^p(S)$ such that for all $t$, $\abs{\nabla \psi_t}\leq g$ $\lambda^k$-a.e. on $S$, and the following is valid for all $t$:
 \begin{gather}
  \int_{S\cap \B 2 0} \abs{\nabla \psi_t}^p \leq C(n) \, , \quad \fint_{S\cap \B 1 0} \abs{\nabla \psi_t}^p\leq 1+\epsilon\, .
 \end{gather}
\end{lemma}

\begin{proof}
The proof is a standard consequence of the properties proved so far for the map $\phi$ and the usual Lusin-type approximation for $W^{1,p}$ functions (see for example \cite[theorem 3, sec 6.6.3]{EG}).
 
First of all, we fix some notation. Given the maps $\phi_i:T\to T_i$ and $\phi_\infty:T\to S$, we set
\begin{gather}
 \psi_i\equiv \phi_i^{-1}:T_i\to T\, , \quad \alpha_i\equiv \phi_\infty \circ \psi_i:T_i\to S\, , \quad \beta_i\equiv \alpha_i^{-1}:S\to T_i\, .
\end{gather}
Note that, by the Reifenberg construction, all these maps are \hol continuous maps.

Next, we introduce a slight variant of the function $f(x)$ defined before. In particular, set for $z\in S$:
\begin{gather}
 g(z)= \prod_{i=0}^\infty \sup_{y\in \B {3r_i}{\beta_i(z)}\cap T_i} \abs{\nabla \sigma_i^{-1}}\, .
\end{gather}
By adapting the proof of the $W^{1,p}$ estimates in \eqref{eq_w1peps}, we prove that $g\in L^p(S)$ with
\begin{gather}
 \int_{S\cap \B 2 0} g(z)^p d\lambda^k(z) \leq \int_{T\cap \B {2.1} 0 } g(\phi(x))^p J(\phi_\infty)|_x d\lambda^k(x) \leq 3C_1(n)\, .
\end{gather}
Moreover, we can also refine the bound to 
\begin{gather}
 \int_{S\cap \B 1 0} g(z)^p d\lambda^k(z) =\int_{T\cap \psi(S\cap \B 1 0)} g(\phi(x))^p J(\phi_\infty)|_x d\lambda^k(x) \leq \omega_k(1+\epsilon)\, .
\end{gather}

Now define for $t\geq 1$ the sets $R_t= g^{-1}[0,t]$. Since $g\in L^p$, then
\begin{gather}\label{eq_gLp}
 \limsup_{t\to \infty} \ t^p \lambda^k(R_t)\leq \limsup_{t\to \infty} \int_{g\geq t} g^p \to 0\, .
\end{gather}
Moreover, $\psi|_{R_t}$ is a Lipschitz function with Lipschitz constant bounded by $c(n)t$. Indeed, let $x,y\in R_t$, and set $i$ to be such that $r_{i+1}/2<\abs{x-y}\leq r_i/2$. We prove that
\begin{gather}\label{eq_lipg}
 \abs{\psi(x)-\psi(y)}\leq c(n) t \abs {x-y}
\end{gather}
that is, $\psi|_{R_t}$ is $c(n)t$ Lipschitz wrt the extrinsic distance on $S$.

By definition, $\psi(z)=\psi_i(\beta_i(z))$. Moreover, the definition of $g$ immediately implies a uniform Lipschitz condition on $\psi_i$. In particular, 
\begin{gather}
 \norm{\nabla \psi_i}_{L^\infty (\B {r_i}{\beta_i(x)} )} \leq \prod_{s=0}^i \norm{\nabla \sigma_s^{-1}}_{L^\infty (\B {r_s}{\beta_s(x)} )} \leq t\, .
\end{gather}
Now, by \eqref{eq_distdelta}, we have that $\abs{\beta_i(x)-x}+\abs{\beta_i(y)-y}\leq c(n,\rho)\delta r_{i+1}$, so, for $\delta\leq \delta_0(n,\rho)$ small enough, we get
\begin{gather}
 \abs{\beta_i(x)-\beta_i(y)}\leq \abs{x-y}+ 2c\delta r_{i+1}\leq 2 \abs{x-y}\, ,
\end{gather}
and thus
\begin{gather}
 \abs{\psi(x)-\psi(y)}\leq \abs{\psi_i(\beta_i(x)) -\psi_i(\beta_i(y))}\leq t \abs{\beta_i(x)-\beta_i(y)}\leq 2t \abs{x-y}\, ,
\end{gather}
as desired.

Now, define $\psi_t$ to be a Lipschitz extension of $\psi|_{R_t}$ over all $\R^n$ with the same Lipschitz constant. In particular, $\psi_t$ will be a Lipschitz function on $S$ with Lipschitz constant bounded by $2t$. Note that, independently of $t$, we have the estimate
\begin{gather}\label{eq_psilip}
 \abs{\nabla \psi_t(z)} \leq 4 g(z) \quad \text{for }\ \ \lambda^k-\text{a.e.}\  z\in S\, .
\end{gather}
Indeed, if $z\in R_t^C$, this estimate is valid for all $x$ for which $\nabla \psi_t(x)$ is defined. Moreover, we prove that this estimate is also valid for all $z\in R_t\subset S$ which have $\lambda^k$-density $1$ (both with respect to $R_t$ and $S$) and such that $\nabla \psi_t(x)$ is defined. Since $R_t\subset S$ is $k$-rectifiable, it is well-known that the density of both these sets is $1$ $\lambda^k$ almost everywhere (see for example \cite[theorem 2.63]{AmFu}). In particular, for $\lambda^k$ a.e. $z\in R_t$ we have
\begin{gather}
 \lim_{r\to 0} \frac{\lambda^k (\B r z \cap R_t)}{\lambda^k (\B r z \cap S)}=1\, .
\end{gather}
For $y\in S$, set $y_t$ to be an almost projection of $y$ onto $R_t$, i.e., a point such that $d(y,y_t)\leq 2 d(y,R_t)$. Then we have
\begin{gather}
 \limsup_{y\to z} \frac{\abs{\psi_t(y)-\psi_t(z)}}{\abs{y-z}}\leq \limsup_{y\to z}\qua{\frac{\abs{\psi_t(y)-\psi_t(y_t)}}{\abs{y-z}} +\frac{\abs{\psi_t(y_t)-\psi_t(z)}}{\abs{y-z}}}\, .
\end{gather}
By the proof of \eqref{eq_lipg}, and since $\abs{y_t-z}\leq 2\abs{y-z}$, we get
\begin{gather}
 \limsup_{y\to z} \frac{\abs{\psi_t(y_t)-\psi_t(z)}}{\abs{y-z}}=\limsup_{y\to z} \frac{\abs{\psi(y_t)-\psi(z)}}{\abs{y-z}}\leq 4g(z)\, .
\end{gather}
Moreover, since $z$ is a density $1$ point on $R_\lambda\subset S$, $\abs{y-y_t}$ cannot be too big around $z$. Indeed, if $\abs{y-y_t}\geq 2r$, then $\B {r}{y}\cap R_t=\emptyset$, and so
\begin{gather}
 \ton{\frac 1 4 \limsup_{y\to z} \frac{\abs{y-y_t}}{\abs{y-z}}}^k= \limsup_{y\to z} \frac{\lambda^k \ton{\B {\abs {y-y_t}/2}{z}\cap  S}}{\lambda^k \ton{\B {2\abs {y-z}}{z}\cap  S}}\leq \limsup_{y\to z} \ton{1-\frac{\lambda^k \ton{\B {2\abs {y-z}}{z}\cap R_t}}{\lambda^k \ton{\B {2\abs {y-z}}{z}\cap  S}}}=0\, .
\end{gather}
By the global Lipschitz estimate on $\psi_t$, we get 
\begin{gather}
 \limsup_{y\to z}\frac{\abs{\psi_t(y)-\psi_t(y_t)}}{\abs{y-z}}\leq 2t\limsup_{y\to z} \frac{\abs{y-y_t}}{\abs{y-z}}=0\, .
\end{gather}
This proves \eqref{eq_psilip}. Now consider any $T\geq t$. Since $S$ is rectifiable, and since $\psi_t=\psi_T$ on $R_t$, then $\nabla \psi_t=\nabla \psi_T$ $\lambda^k$-a.e. on $R_t$. Moreover, by \eqref{eq_psilip}, we have
\begin{gather}
 \int_{S}\abs{\nabla \psi_T-\nabla \psi_t}^p =\int_{R_t^C} \abs{\nabla \psi_T-\nabla \psi_t}^p\leq 2^p \int_{R_t^C}\abs{\nabla \psi_t}^p + 2^p \int_{R_t^C} \abs{\nabla \psi_T}^p\leq 2^p t^p \lambda^k \ton{R_t^C} + 2^p \int_{R_t^C} \abs{g}^p\, .
\end{gather}
By \eqref{eq_gLp}, this proves point $(4)$. Since $\psi$ is a uniformly continuous function, point $(3)$ is a corollary of point $(2)$.

\end{proof}

\subsection{Proof of Theorem \ref{t:reifenberg_W1p_discrete}: The Discrete Rectifiable-Reifenberg}
Up to minor differences, the proof of this theorem is essentially the same as the proof of \eqref{eq_volR}, and is actually much simpler from the technical point of view. Indeed, we do not need to define the sets $\tilde E,\tilde B,\tilde V$, but we can just deal with the sets $E,V,B$. 

In particular, let $r_\alpha=2^{-\alpha}$ and $\bar r = r_A$ for some $A\in \N$. One can define the measure 
\begin{gather}
 \mu\equiv \sum_{x_j\in S}\omega_k r^k_j \delta_{x_j} \, , \quad  \mu_{\bar r}=\omega_k \sum_{x_j\in S\cap \B 1 0 \ \ s.t. \ \ r_j\geq \bar r} r_j^k \delta_{x_j}\, ,
\end{gather}
and prove inductively on $\alpha=A,A-1,\cdots, 0$ that for all $x\in \B 1 0$ such that $\B {r_\alpha/10} x $ is not contained in any of the $\B {r_j}{x_j}$:
\begin{gather}
 \mu_{\bar r}(\B {r_\alpha}{x})\leq D(n) r_\alpha^k\, .
\end{gather}
By proceeding with a rough covering argument as in the proof of \eqref{eq_lambdarough}, one obtains easily a rough upper bound on $\mu_{\bar r}(\B {r_{\alpha-1}} x)$. Moreover, by mimicking the inductive constructions in Section \ref{sec_IIind_pre}, one can define Excess sets, good, bad and final balls, and the maps $\sigma_i$ and the approximating manifolds $T_i$. By studying the bi-Lipschitz properties of $\sigma_i$, and by keeping track of the holes created by final and bad balls in the same way as before, one proves the inductive estimate
\begin{gather}
 \mu_{\bar r}(\B {r_\alpha}{x}) 
 \leq C(n) \ton{1+\sum_{r_\beta\leq 2r_\alpha}\int_{B_{r_\alpha}(x)}D^k_\mu(x,r_\alpha)\, d\mu(x)}\, ,
\end{gather}
as desired.

\vspace{1cm}

\section{\texorpdfstring{$L^2$-Best Approximation Theorems}{L2-Best Approximation Theorems}}\label{s:best_approx}

In this Section we prove the main estimate necessary for us to be able to apply the rectifiable-Reifenberg of Theorems \ref{t:reifenberg_W1p_holes} and \ref{t:reifenberg_W1p_discrete} to the singular sets $S^k_\epsilon(f)$ of the stratification induced for stationary harmonic maps.\\

Namely, we need to understand how to estimate on a ball $B_r(x)$ the $L^2$-distance of $S^k_\epsilon$ from the best approximating $k$-dimensional subspace.  When $r<\inj(M)$ this means we would like to consider subspaces $L^k\subseteq T_xM$ and estimate $d^2(x,L^k)$ for $x\in S^k\cap B_r(x)$, where the distance is taken in the normal coordinate charts.  After rescaling this is equivalent to looking at a ball of definite size, but assuming $K_M<<1$.  The main Theorem of this Section is stated in some generality as we will need apply it with some care when proving Theorems \ref{t:main_quant_strat_stationary} and \ref{t:main_eps_stationary}. We recall that by definition
\begin{align}
W_{\alpha}(x) \equiv W_{r_{\alpha},r_{\alpha-3}}(x)\equiv \theta_{r_{\alpha-3}}(x)-\theta_{r_{\alpha}}(x)\geq 0\, ,
\end{align}
where $r_\alpha=2^{-\alpha}$.

\begin{theorem}[$L^2$-Best Approximation Theorem]\label{t:best_approximation}
Let $f:B_{9r}(p)\to N$ be a stationary harmonic map satisfying \eqref{e:manifold_bounds} with $r^{2-n}\int_{B_{9}(p)}|\nabla f|^2\leq \Lambda$, and let $\epsilon>0$.  Then there exists $\delta(n,K_N,\Lambda,\epsilon)$, $ C(n,K_N,\Lambda,\epsilon)>0$ such that if $K_M<\delta$, and $B_{9r}(p)$ is $(0,\delta)$-symmetric but {\it not} $(k+1,\epsilon)$-symmetric, then for any finite measure $\mu$ on $B_r(p)$ we have that
\begin{align}\label{e:best_approx:L2_estimate}
D_\mu(p,r) = \inf_{L^k} r^{-2-k}\int_{B_r(p)} d^2(x,L^k)\,d\mu(x) \leq C r^{-k}\int_{B_r(p)} W_0(x)\,d\mu(x)\, , 
\end{align}
where the $\inf$ is taken over all $k$-dimensional affine subspaces $L^k\subseteq T_p M$.
\end{theorem}
\begin{remark}
 Throughout this section, for convenience we will fix $r=1$ without essential loss of generality. 
\end{remark}

\begin{remark}
The assumption $K_M<\delta$ is of little consequence, since given any $K_M$ this just means focusing the estimates on balls of sufficiently small radius after rescaling.
\end{remark}
\vspace{.5 cm}

\subsection{Symmetry and Gradient Bounds}\label{ss:symmetry_gradient}

In this subsection we study stationary harmonic maps which are {\it not} $(k+1,\epsilon)$-symmetric on some ball.  In particular, we show that this forces the gradient to be of some definite size when restricted to any $k+1$-dimensional subspace.  More precisely:

\begin{lemma}\label{l:best_subspace:energy_lower_bound}
Let $f:B_4(p)\subseteq M\to N$ be a stationary harmonic map satisfying \eqref{e:manifold_bounds} with $\fint_{B_4(p)}|\nabla f|^2\leq \Lambda$.  Then for each $\epsilon>0$ there exists $\delta(n,K_N,\Lambda,\epsilon)>0$ such that if $K_M<\delta$, and $B_4(p)$ is $(0,\delta)$-symmetric but is {\it not} $(k+1,\epsilon)$-symmetric, then for every $k+1$-subspace $V^{k+1}\subseteq T_p M$ we have 
\begin{align}
\fint_{A_{3,4}(p)} |\langle \nabla f, V\rangle|^2 d\lambda \geq \delta\, ,
\end{align}
where $|\langle \nabla u, V\rangle|^2\equiv \sum_1^{k+1} |\langle\nabla u,v_i\rangle|^2$ for any orthonormal basis $\{v_i\}$ of $V$, and $\lambda$ is the volume measure on $M$ (which is equivalent to $\lambda^n$).
\end{lemma}

\begin{proof}

The proof is by contradiction.  So with $n,K_N,\Lambda,\epsilon>0$ fixed let us assume the result fails.  Then there exists a sequence $f_i:B_4(p_i)\to N_i$ for which $K_{M_i}<\delta_i$, with $B_4(p_i)$ being $(0,\delta_i)$-symmetric but not $(k+1,\epsilon)$-symmetric, and such that for some subspaces $V_i^{k+1}$ we have that
\begin{align}\label{e:best_subspace_lower_energy:1}
\fint_{A_{3,4}(p_i)}|\langle \nabla f_i,V^{k+1}_i\rangle |^2 \leq \delta_i\to 0\, .
\end{align}
Then after passing to subsequences we have that $V_i\to V^k\subseteq \dR^n$ with
\begin{align}
N_i\stackrel{C^{1,\alpha}}{\longrightarrow} N\, ,
\end{align}
and
\begin{align}
f_i\longrightarrow f:B_4(0^n)\to N\, ,
\end{align}
where the convergence is in $L^2\cap H^1_{weak}$.  Now \eqref{e:best_subspace_lower_energy:1} and the $H^1_{weak}$ convergence guarantees that 
\begin{align}
\fint_{A_{3,4}(0^n)}|\langle \nabla f,V^{k+1}\rangle |^2 = 0\, .
\end{align}

On the other hand, the $(0,\delta_i)$-symmetry of the $f_i$, combined with the $L^2$ convergence, tells us that $f$ is $0$-symmetric.  Combining these tells us that 
\begin{align}
\fint_{B_4(0^n)}|\langle \nabla f,V^{k+1}\rangle |^2 = 0\, ,
\end{align}
and hence we have that $f$ is $k+1$-symmetric.  Because the convergence $f_i\to f$ is in $L^2$, this contradicts that the $f_i$ are not $(k+1,\epsilon)$-symmetric for $i$ sufficiently large, which proves the Lemma.

\end{proof}
\vspace{1cm}

\subsection{\texorpdfstring{Best $L^2$-Subspace Equations}{Best L2-Subspace Equations}}\label{ss:best_subspaces}

In order to prove Theorem \ref{t:best_approximation} we need to identify which subspace minimizes the $L^2$-energy, and the properties about this subspace that allow us to estimate the distance.  We begin in Section \ref{sss:second_moments}  by studying some very general properties of the second directional moments of a general probability measure $\mu\subseteq B_1(p)$.  We will then study in Section \ref{sss:restricted_Dirichlet} the Dirichlet energy of a stationary harmonic map when restricted to these directions.\\

\subsubsection{Second Directional Moments of a Measure}\label{sss:second_moments}

Let us consider a probability measure $\mu\subseteq B_1(0^n)$, and let
\begin{align}
x^i_{cm}=x^i_{cm}(\mu)\equiv \int x^i \, d\mu(x)\, ,
\end{align}
be the center of mass.  Let us inductively consider the maximum of the second directional moments of $\mu$.  More precisely:

\begin{definition}\label{d:second_moments}
Let $\lambda_1=\lambda_1(\mu)\equiv \max_{|v|^2=1} \int |\langle x-x_{cm},v\rangle|^2\, d\mu(x)$
and let $v_1=v_1(\mu)$ with $|v_1|=1$ (any of) the vector obtaining this maximum.  Now let us define inductively the pair $(\lambda_{k+1},v_{k+1})$ from $v_1,\ldots,v_{k}$ by
\begin{align}
\lambda_{k+1}=\lambda_{k+1}(\mu)\equiv \max_{|v|^2=1,\langle v,v_i\rangle=0} \int |\langle x-x_{cm},v\rangle|^2\, d\mu(x)\, ,
\end{align}
where $v_{k+1}$ is (any of) the vector obtaining this maximum.
\end{definition}

Thus $v_1,\ldots,v_n$ defines an orthonormal basis of $\dR^n$, ordered so that they maximize the second directional moments of $\mu$.  Let us define the subspaces
\begin{align}\label{e:best_subspace:Vk}
V^k= V^k(\mu)\equiv x_{cm}+\text{span}\{v_1,\ldots,v_k\}\, .
\end{align}
The following is a simple but important exercise:

\begin{lemma}\label{l:best_subspace_Vk}
If $\mu$ is a probability measure in $B_1(0^n)$, then for each $k$ the functional 
\begin{align}
\min_{L^k\subseteq \dR^n} \int d^2(x,L^k)\,d\mu(x)\, ,
\end{align}
where the $\min$ is taken over all $k$-dimensional affine subspaces, attains its minimum at $V^k$.  Further, we have that
\begin{align}
\min_{L^k\subseteq \dR^n} \int d^2(x,L^k)\,d\mu(x) = \int d^2(x,V^k)\,d\mu(x) = \lambda_{k+1}(\mu)+\cdots+\lambda_n(\mu)\, .
\end{align}
\end{lemma}
Note that the best affine subspace $V^k$ will necessarily pass through the center of mass $x_{cm}$ by Steiner's formula, or equivalently by Jensen's inequality and the definition of $x_{cm}$.

\vspace{.5cm}

Now let us record the following Euler-Lagrange formula, which is also an easy computation:

\begin{lemma}\label{l:best_subspace:euler_lagrange}
If $\mu$ is a probability measure in $B_1(0^n)$, then we have that $v_1(\mu),\ldots,v_n(\mu)$ satisfy the Euler-Lagrange equations:
\begin{align}\label{e:second_moment_EL}
\int \langle x-x_{cm},v_k\rangle (x-x_{cm})^i\, d\mu(x) = \lambda_k v_k^i \, ,
\end{align}
where 
\begin{align}
&\lambda_k = \int |\langle x-x_{cm},v_k\rangle|^2\, d\mu(x)\, .
\end{align}
\end{lemma}
\begin{proof}
 The proof is a simple application of the Lagrange multipliers method. Inductively on $k$, consider the function $f(v_k,\lambda_k,\lambda_{k,1},\cdots,\lambda_{k,k-1}):\R^n\times\R^{k}\to \R$ given by
 \begin{gather}
  f(v_k,\lambda_k,\lambda_{k,\ell}) = \int \ps{x-x_{cm}}{v_k}^2 d\mu(x) -\lambda_k \ton{\abs{v_k}^2-1}-2\sum_{\ell=0}^{k-1}\lambda_{k,\ell} \ps{v_k}{v_\ell}\, .
 \end{gather}
By the multipliers method, we obtain that $v_k$ and $\lambda_k$ satisfy the equations
\begin{gather}
 \frac 1 2 \nabla^{(v_k)} f(v_k,\lambda_k,\lambda_{k,\ell}) = \int \ps{x-x_{cm}}{v_k}(x-x_{cm}) d\mu(x) -\lambda_k v_k- \sum_{\ell=0}^{k-1}\lambda_{k,\ell} v_\ell=0\, .
\end{gather}
By taking the scalar product of this equation with $v_\ell$, since $\ps{v_k}{v_\ell}=\delta_{k,\ell}$, we have
\begin{gather}
\lambda_k = \int |\langle x-x_{cm},v_k\rangle|^2\, d\mu(x)\, ,\\
\lambda_{k,\ell} = \int \langle x-x_{cm},v_\ell\rangle\langle x-x_{cm},v_k\rangle\, d\mu(x)= \ps{\int \ps{x-x_{cm}}{v_\ell}(x-x_{cm})\, d\mu(x)}{v_k}=\notag \\
= \ps{\lambda_\ell v_\ell +\sum_{s=1}^{\ell-1} \lambda_{\ell,s}v_s}{v_k}=0\, .
\end{gather}

\end{proof}

\vspace{1cm}

\subsubsection{Restricted Dirichlet Energies}\label{sss:restricted_Dirichlet}

Our goal is to now study the Dirichlet energy of a stationary harmonic map when restricted to the directions $v_1(\mu),\ldots,v_n(\mu)$ associated to a probability measure.  The main result of this subsection is the following, which holds for a general stationary harmonic map:


\begin{proposition}\label{p:best_subspace_estimates}
Let $f:\B 9 p\subseteq M\to N$ be a stationary harmonic map satisfying \eqref{e:manifold_bounds} with $K_M<2^{-4}$ and such that $\fint_{\B 9 p}|\nabla f|^2\leq \Lambda$.  Let $\mu$ be a probability measure on $B_1(p)$ with $\lambda_k(\mu),v_k(\mu)$ defined as in Definition \ref{d:second_moments}.  Then there exists $C(n,K_N)>0$ such that
\begin{align}
\lambda_k \int_{A_{3,4}(p)} |\langle\nabla f(z), v_k\rangle|^2\,dv_g(z) \leq C \int W_0(x)\, d\mu(x)\, .
\end{align}
\end{proposition}
\begin{proof}
Note first that there is no harm in assuming that $x_{cm}\equiv 0$.  If not we can easily translate to make this so, in which case we still have that $\text{supp}(\mu)\subseteq B_2$.  Additionally, we will simplify the technical aspect of the proof by assuming that $M\equiv \dR^n$.  By working in normal coordinates the proof of the general case is no different except up to some mild technical work.  \\

Now let us fix any $z\in A_{3,4}$. By Lemma \ref{l:best_subspace:euler_lagrange}, we can inner product both sides of \eqref{e:second_moment_EL} by $\nabla f(z)$ to obtain for each $k$ and $z\in A_{3,4}$:
\begin{align}
\lambda_k \ps{\nabla f(z)}{v_k} &= \int \ps{x}{v_k}\ps{\nabla f(z)}{x} \, d\mu(x) \, ,
\end{align}
Observe that, by definition of center of mass,
\begin{align}
\int \langle x,z\rangle\, d\mu(x) = \langle x_{cm},z\rangle = 0\, . 
\end{align}
Thus we can write
\begin{align}
\lambda_k \ps{\nabla f(z)}{v_k}&= \int \ps{ x}{v_k}\ps{\nabla f(z)}{x-z}\,d\mu(x) \, .
\end{align}
By \hol inequality, we can then estimate
\begin{align}
\lambda_k^2\abs{\ps{\nabla f(z)}{v_k}}^2 \leq \lambda_k\int \abs{\ps{\nabla f(z)}{x-z}}^2\,d\mu(x) \, .
\end{align}
If $\lambda_k\neq 0$, integrating with respect to $z$ on both sides we get the estimate
\begin{align}\label{e:best_subspace_estimate:1}
\lambda_k \int_{A_{3,4}}\abs{\ps{\nabla f(z)}{v_k}}^2\,dv_g(z)\leq & \int\int_{A_{3,4}} \abs{\ps{\nabla f(z)}{x-z}}^2\,dv_g(z)\,d\mu(x) \, .
\end{align}
Set for convenience $n_x(z)=(z-x)/\abs{z-x}$, i.e., $n_x(z)$ is the radial vector from $x$ to $z$. Now for $x\in \text{supp}(\mu)$ we can estimate
\begin{align}
\int_{A_{3,4}} \abs{\ps{\nabla f(z)}{x-z}}^2\,dv_g(z) &= \int_{A_{3,4}} \abs{\ps{\nabla f(z)}{n_x(z)}}^2|x-z|^{2-n}|x-z|^n\,dv_g(z)\notag\\
&\leq 6^n\int_{A_{3,4}} |\langle \nabla f(z),n_x(z)\rangle|^2|x-z|^{2-n}\,dv_g(z)\notag\\
&\leq C(n)\int_{A_{1,8}(x)} |\langle \nabla f(z),n_x(z)\rangle|^2|x-z|^{2-n}\,dv_g(z) \notag\\
&= C(n)W_0(x)\, .
\end{align}
Applying this to \eqref{e:best_subspace_estimate:1} we get as desired
\begin{align}\label{e:best_subspace_estimate:2}
\lambda_k \int_{A_{3,4}}\abs{\ps{\nabla f(z)}{v_k}}^2\,dv_g(z)\leq &C(n)\int W_0(x)\,d\mu(x)\, .
\end{align}

\end{proof}
\vspace{1cm}

\subsection{Proof of Theorem \ref{t:best_approximation}}\label{ss:proof_L2_best}

Let us now combine the results of this Section in order to prove Theorem \ref{t:best_approximation}.  Indeed, let $\mu$ be a measure in $B_1(p)\subseteq T_p M$. We can assume that $\mu$ is a probability measure without any loss of generality, since both sides of our estimate scale.  Let $\big(\lambda_1(\mu),v_1(\mu)\big),\ldots,\big(\lambda_n(\mu),v_n(\mu)\big)$ be the directional second moments as defined in Definition \ref{d:second_moments}, with $V^k$ the induced subspaces defined as in \eqref{e:best_subspace:Vk}.  Using Lemma \ref{l:best_subspace_Vk} we have that
\begin{align}\label{e:best_subspace:proof:1}
\min_{L^k\subseteq \dR^n} \int d^2(x,L^k)\,d\mu(x) = \int d^2(x,V^k)\,d\mu(x) = \lambda_{k+1}(\mu)+\cdots+\lambda_n(\mu)\leq (n-k)\lambda_{k+1}(\mu)\, ,
\end{align}
where we have used that $\lambda_j\leq \lambda_i$ for $j\geq i$.  Therefore our goal is to estimate $\lambda_{k+1}$.  To begin with, Proposition \ref{p:best_subspace_estimates} tells us that for each $j$ 
\begin{align}
\lambda_j \int_{A_{3,4}(p)} |\langle\nabla f(z), v_j\rangle|^2\,dv_g(z) \leq C \int W_0(x)\, d\mu(x)\, .
\end{align}
Let us sum the above for all $j\leq k+1$ in order to obtain
\begin{align}
\sum_{j=1}^{k+1}\lambda_j\int_{A_{3,4}(p)} |\langle\nabla f(z), v_j\rangle|^2\,dv_g(z) \leq (k+1)C \int W_0(x)\, d\mu(x)\, ,
\end{align}
or by using that $\lambda_{k+1}\leq \lambda_{j}$ for $k+1\geq j$ we get
\begin{align}\label{e:best_subspace:proof:2}
\lambda_{k+1}\int_{A_{3,4}(p)} |\langle\nabla f(z), V^{k+1}\rangle|^2\,dv_g(z)= \lambda_{k+1} \sum_{j=1}^{k+1}\int_{A_{3,4}(p)} |\langle\nabla f(z), v_j\rangle|^2\,dv_g(z)\leq C \int W_0(x)\, d\mu(x)\, .
\end{align}

Now we use that $B_8(p)$ is $(0,\delta)$-symmetric, but {\it not} $(k+1,\epsilon)$-symmetric in order to apply Lemma \ref{l:best_subspace:energy_lower_bound} and conclude that
\begin{align}
\int_{A_{3,4}(p)} |\langle\nabla f(z), V^{k+1}\rangle|^2\,dv_g(z)\geq \delta\, .
\end{align}
Combining this with \eqref{e:best_subspace:proof:2} we obtain
\begin{align}\label{e:best_subspace:proof:3}
\delta\lambda_{k+1}\leq \lambda_{k+1} \int_{A_{3,4}(p)} |\langle\nabla f(z), V^{k+1}\rangle|^2\,dv_g(z)\leq C \int W_0(x)\, d\mu(x)\, ,
\end{align}
or that
\begin{align}
\lambda_{k+1}\leq C(n,K_N,\Lambda,\epsilon) \int W_0(x)\, d\mu(x)\, .
\end{align}
Combining this with \eqref{e:best_subspace:proof:1} we have therefore proved the Theorem. \hfill $\square$
\vspace{1cm}

\section{The Inductive Covering Lemma}\label{s:covering_main}

This Section is dedicated to the basic covering lemma needed for the proof of the main theorems of the paper.  The covering scheme is similar in nature to the one introduced by the authors in \cite{NaVa} in order to prove structural theorems on critical sets.  Specifically, let us consider a stationary harmonic map $f$ between Riemannian manifolds.  We wish to build a covering of the quantitative stratification
\begin{align}
S^k_{\epsilon,r}(f)\cap B_1(p)\subseteq U_r\cup U_+ = U_r\cup\bigcup B_{r_i}(x_i)\, ,
\end{align}
which satisfies several basic properties.  To begin with, the set $U_+$ is a union of balls satisfying $r_i>r\geq 0$, and should satisfy the packing estimate $\omega_k\sum r_i^k\leq C$.  Each ball $B_{r_i}(x_i)$ should have the additional property that there is a definite energy drop of $f$ when compared to $B_1(p)$.  To describe the set $U_r$ we should distinguish between the case $r>0$ and $r\equiv 0$.  In the case $r>0$ we will have that $U_r = \bigcup B_{r}(x^r_i)$ is a union of $r$-balls and satisfies the Minkowski estimate $\Vol(U_r)\leq C r^{n-k}$.  In the case when $r\equiv 0$ we will have that $U_0$ is $k$-rectifiable with the Hausdorff estimate $\lambda^{k}(U_0)\leq C$.  Let us be more precise:\\

\begin{lemma}[Covering Lemma]\label{l:covering}
Let $f:B_{16}(p)\subseteq M\to N$ be a stationary harmonic map satisfying \eqref{e:manifold_bounds} with $\fint_{B_{16}(p)}|\nabla f|^2\leq \Lambda$ and $K_M<\delta(n,K_N,\Lambda,\epsilon)$.  Let $E = \sup_{x\in B_1(p)\cap S^k_{\epsilon,r}} \theta_1(x)$ with $\epsilon>0$, $r\geq 0$, and $k\in\dN$.  Then for all $\eta\leq \eta(n,K_N,\Lambda,\epsilon)$, there exists a covering $S^k_{\epsilon,r}(f)\cap B_1(p)\subseteq U = U_r\cup U_+$ such that 
\begin{enumerate}
\item $U_+=\bigcup B_{r_i}(x_i)$ with $r_i>r$ and $\sum r_i^k\leq C(n,K_N,\Lambda,\epsilon)$.
\item $\sup_{y\in B_{r_i}(x_i)\cap S^k_{\epsilon,r}}\theta_{r_i}(y)\leq E- \eta$.
\item If $r>0$ then $U_r=\bigcup_1^N B_{r}(x^r_i)$ with $N\leq C(n) r^{-k}$.
\item If $r=0$ then $U_0$ is $k$-rectifiable and satisfies $\Vol(B_s\, \ton{U_0})\leq C(n) s^{n-k}$ for each $s>0$.\newline  In particular, $\lambda^k(U_0)\leq C(n)$.
\end{enumerate}
\end{lemma}
\begin{remark}
The assumption $K_M<\delta$ is of little consequence, since given any $K_M$ this just means focusing the estimates on balls of sufficiently small radius after rescaling.
\end{remark}

\vspace{.5 cm}

To prove the result let us begin by outlining the construction of the covering, we will then spend the rest of this section proving the constructed cover has all the desired properties.\\

Thus let us consider some $\eta>0$ fixed, and then define the energy scale for $x\in B_1(p)$ by
\begin{align}\label{e:covering:energy_scale}
s_x=s^{E,\eta}_x \equiv \inf\big\{r\leq s\leq 1: \sup_{B_s(x)\cap S^k_{\epsilon,r}}\theta_{\eta s}(y)\geq E-\eta\big\}\, .
\end{align}
Note that the energy scale implicitly depends on many constants.  If $r=0$ let us define the set $U_0$ by
\begin{align}
U_0\equiv \big\{x\in S^k_{\epsilon,r}(f)\cap B_1(p): s_x=0\big\}\, ,
\end{align}
while if $r>0$ let us first define the temporary covering
\begin{align}
U'_r = \bigcup B_{r}(x^r_i)\, ,
\end{align}
where 
\begin{align}
\{x_i^r\}\subseteq \big\{x\in S^k_{\epsilon,r}(f)\cap B_1(p): s_x=r\big\}\, ,
\end{align}
is any minimal $r/5$-dense set.  In particular, note that the collection of balls $\{B_{r/10}(x^r_i)\}$ are disjoint.\\

In order to define the temporary covering $U'_+=\{B_{r_i}(x_i)\}$ let us first consider
\begin{align}
\big\{x\in S^k_{\epsilon,r}(f)\cap B_1(p): s_x>r\big\}\subseteq \bigcup_{s_x>r} B_{s_x/10}(x)\, ,
\end{align}
and choose from it a Vitali subcovering so that 
\begin{align}
\big\{x\in S^k_{\epsilon,r}(f)\cap B_1(p): s_x>r\big\}\subseteq \bigcup_{i\in I} B_{r_i/2}(x_i) \subseteq \bigcup_{i\in I} B_{r_i}(x_i)\equiv U'_+\, ,
\end{align}
where $r_i\equiv s_{x_i}$.  In particular, we have that the collection of balls $\{B_{r_i/10}(x_i)\}$ are all disjoint.  It is clear from the construction that we have built a covering
\begin{align}
S^k_{\epsilon,r}(f)\cap B_1(p) \subseteq \bigcup B_{r/2}(x^r_i) \cup \bigcup_{i\in I} B_{r_i/2}(x_i)\subseteq \bigcup B_{r}(x^r_i) \cup \bigcup_{i\in I} B_{r_i}(x_i) \equiv U'_r\cup U'_+\, .
\end{align}

Now this is not quite the covering of Lemma \ref{l:covering}, as the energy drop properties of \eqref{e:covering:energy_scale} involve dropping an extra $\eta$-scales.  It will be more convenient to estimate this covering, however with only minimal work let us now build from this covering the desired covering of Lemma \ref{l:covering}, which is only a small modification.  Indeed, consider for each ball $B_{r_i}(x_i)$ (or $B_r(x^r_i)$) a Vitali covering
\begin{align}
	B_{r_i/2}(x_i)\cap S^k_{\epsilon,r}\subseteq \bigcup_1^{N_i} B_{\eta r_i}(x_{ia})\equiv \bigcup_1^{N_i} B_{r_{ia}}(x_{ia})\, ,
\end{align}
where $x_{ia}\in B_{r_i/2}(x_i)\cap S^k_{\epsilon,r}$ and by a standard covering argument $N_i\leq N_i(n,\eta)$.  Then we can define the coverings
\begin{align}
U_r \equiv \bigcup_i \bigcup_{a=1}^{N_i} B_{r_{ia}}(x^r_{ia})\, ,\notag\\
U_+ \equiv \bigcup_i \bigcup_{a=1}^{N_i} B_{r_{ia}}(x_{ia})\, .
\end{align}
It is clear from the construction that $U_+$ now satisfies the energy drop condition of Lemma \ref{l:covering}.2 .  What is left is to show the content estimates of Lemma \ref{l:covering}, and from our estimates on $N_i$ it is clear with $\eta<\eta(n,K_N,\Lambda,\epsilon)$ that it is enough to prove these estimates for the sets $U'_r$ and $U'_+$ themselves, which is therefore the goal of much of this section.  \\

The outline of this Section is as follows.  Section \ref{ss:covering:energy_drop} is dedicated to proving a variety of technical lemmas which will be used to further decompose the sets $U'_+$ and $U'_r$ when $r>0$.  The technical issue is that we cannot directly apply the discrete Reifenberg to the set $U'_+$, and will instead need to exchange $U'_+$ for a more manageable collection of balls without losing much content.  In Section \ref{ss:covering_U_+} we will use these tools in order to prove our content estimates on $U'_+$ and $U'_r$ when $r>0$.  The proof will require an inductive argument combined with applications of the discrete Reifenberg of Theorem \ref{t:reifenberg_W1p_discrete} to our new covering.  Finally in Section \ref{ss:covering_U_0} we will prove the desired estimates on $U_0$.  The volume estimates will follow almost immediately from our previous constructions, and to prove the rectifiable statement will require a careful application of the rectifiable Reifenberg theorem.\\

\subsection{\texorpdfstring{Technical Constructions for Decomposing $U'_r\cup U'_+$}{Technical Constructions for Decomposing U>}}\label{ss:covering:energy_drop}

Let us consider the set of positive radius balls in our covering given by
\begin{align}\label{e:U>}
U_> \equiv \begin{cases}
 U'_+ &\text{ if } r=0\, ,\\	
 U'_r\cup U'_+ &\text{ if } r>0\, .
 \end{cases}
\end{align}

This section will be dedicated to proving some technical results which will be required in estimating this set.  The estimates on this set are a little delicate, the reason being that we cannot directly apply the rectifiable Reifenberg of Theorem \ref{t:reifenberg_W1p_discrete} to this set.  Instead, we will need to replace $U_>$ with a different covering at each stage, which will be more adaptable to Theorem \ref{t:reifenberg_W1p_discrete}.  Thus this subsection is dedicated to proving a handful of technical results which are important in the construction of this new covering.\\

Throughout this subsection we are always working under the assumptions of Lemma \ref{l:covering}.  Let us begin with the following point, which is essentially a consequence of the continuity of the energy:

\begin{lemma}\label{l:covering:energy_bound}
For each $\eta'>0$ there exists $R(n,K_N,\Lambda,\eta')>1$ such that if $\delta<\delta(n,K_N,\Lambda,\eta')$ and $\eta\leq \eta(n,K_N,\Lambda,\eta')$, then we have for each $z\in \B {r_i}{x_i}$ the estimate
\begin{align}
\theta_{Rr_i}(z)>E-\eta'\, .
\end{align}
\end{lemma}
\begin{proof}
The proof relies on a straight forward energy comparison.  Namely, let $x,y\in B_{s/2}$ with $s<1$ and let us denote $d\equiv d(x,y)$.  Then we have the estimate
\begin{align}\label{e:covering:energy_bound:1}
\theta_s(y) = s^{2-n}\int_{B_s(y)}|\nabla f|^2 \leq s^{2-n}\int_{B_{s+d}(x)}|\nabla f|^2 = \ton{\frac{s}{s+d}}^{2-n}\theta_{s+d}(x)\, .
\end{align}
To apply this in our context, let us note for each $x_i$ in our covering, that by our construction of $U_>$ there must exist $y_i\in B_{r_i}(x_i)$ such that $\theta_{(R-1)r_i}(y_i)\geq \theta_{r_i}(y_i)=E-\eta$.  Let us now apply \eqref{e:covering:energy_bound:1} to obtain
\begin{align}
\theta_{Rr_i}(z)\geq \ton{\frac{Rr_i}{(R-2)r_i}}^{2-n}\theta_{(R-2)r_i}(y_i)\geq \ton{\frac{R}{R-2}}^{2-n}\big(E-\eta\big)\, . 
\end{align}
If $R=R(n,K_N,\Lambda,\eta')>0$ and $\eta\leq \eta(n,K_N,\Lambda,\eta')$, then we obtain from this the claimed estimate.
\end{proof}

In words, the above Lemma is telling us that even though we have no reasonable control over the size of $\theta_{r_i}(x_i)$, after we go up a controlled number of scales we can again assume that the energy density is again close to $E$.\\

In the last Lemma the proof was based on continuity estimates on the energy $\theta$.  In the next Lemma we wish to show an improved version of this continuity under appearance of symmetry.  Precisely:\\

\begin{lemma}[Improved Continuity of $\theta$]\label{l:covering:improved_continuity}
Let $f:B_4(p)\to N$ be a stationary harmonic map satisfying \eqref{e:manifold_bounds} with $\fint_{B_4(p)}|\nabla f|^2\leq \Lambda$.  Then for each $0<\tau,\eta<1$ there exists $\delta(n,K_N,\Lambda,\eta,\tau)>0$ such that if
\begin{enumerate}
\item we have $K_M<\delta$ ,
\item there exists $x_0,\ldots,x_k\in B_1(p)$ which are $\tau$-independent with $|\theta_{\delta}(x_j)-\theta_3(x_j)|<\delta$,
\end{enumerate}
then if $V^k$ is the $k$-dimensional subspace spanned by $x_0,\ldots,x_k$, then for all $x,y\in B_1(p)\cap B_{\delta}(V^k)$ and $\eta\leq s\leq 1$ we have that $|\theta_{s}(x)-\theta_s(y)|<\eta$.
\end{lemma}
\begin{proof}
The proof is by contradiction.  Thus, imagine no such $\delta$ exists.  Then there exists a sequence of stationary harmonic maps $f_i: B_4(p_i)\to N_i$ satisfying $\fint_{B_4(p)}|\nabla f|^2\leq \Lambda$ such that
\begin{enumerate}
\item $K_M<\delta_i\to 0$
\item there exists $x_{i,0},\ldots,x_{i,k}\in B_1(p)$ which are $\tau$-independent with $\abs{\theta^{f_i}_{\delta_i r}(x_{i,j})-\theta^{f_i}_3(x_{i,j})}<\delta_i\to 0$,
\end{enumerate}
however we have that there exists $x_i,y_i\in B_1(p)\cap B_{\delta_i}(V^k_i)$ and $\eta\leq s_i\leq 1$ such that $\abs{\theta^{f_i}_{s_i}(x_i)-\theta^{f_i}_{s_i}(y_i)}\geq \eta$.  Let us now pass to a subsequence to obtain the defect measure
\begin{align}
&|\nabla f_i|^2dv_g\to \abs{\nabla f}^2 dv_g +\nu\, , \notag\\
&x_{i,\beta}\to x_\beta\in B_1(0^n)\, ,\notag\\
&x_i\to x,\, y_i\to y\in V=\text{span}\{x_0,\ldots,x_k\}\, ,\notag\\
&s_i\to s\, ,
\end{align}
where $\nu$ is a measure on $B_4(0^n)$, and
\begin{align}\label{e:covering:improved_cont:1}
s^{2-n}\abs{\int_{B_s(x)}\ton{\abs{\nabla f}^2 dv_g + d\nu} - \int_{B_s(y)}\ton{\abs{\nabla f}^2 dv_g + d\nu}}\geq \eta\, .
\end{align}
However, we have by theorem \ref{t:defect:0symmetry} that $f$ and $\nu$ are $0$-symmetric on $B_2$ with respect to each of the points $x_0,\ldots x_k$.  In particular, by the standard cone splitting arguments we have that $\nu$ is $k$-symmetric with respect to the $k$-plane $V$ on $B_1$.  In particular, $\nu$ is invariant under translation by elements of $V$.  However, this is a contradiction to \eqref{e:covering:improved_cont:1}, and therefore we have proved the result.

\end{proof}
\vspace{.5 cm}

Now the first goal is to partition $U_>$ into a finite collection, each of which will have a few more manageable properties than $U_>$ itself.  More precisely:

\begin{lemma}\label{l:covering:break_up_U_+}
For each $R>1$ there exists $N(n,R)>1$ such that we can break up $U_>$ as a union
\begin{align}\label{e:covering:U^a_decomposition}
U_> = \bigcup_{a=1}^N U^a_>=\bigcup_{a=1}^N \bigcup_{i\in I^a} \B{r_i}{x_i}\, ,
\end{align}
such that each $U^a_>$ has the following property:  if $i\in I^a$, then for any other $j\in I^a$ we have that if $x_j\in B_{Rr_i}(x_i)$, then $r_j<R^{-2}r_i$.
\end{lemma}
\begin{proof}
Let us recall that the balls in the collection $\{B_{r_i/5}(x_i)\}$ are pairwise disjoint.  In particular, given $R>1$ if we fix a ball $B_{r_i}(x_i)$ then by the usual covering arguments there can be at most $\bar N(n,R)$ ball centers $\{x_j\}_1^N\in U_>\cap B_{R^3 r_i}(x_i)$ with the property that $r_j\geq R^{-6}r_i$.  Indeed, if $\{x_j\}_1^N$ is such a collection of balls then we get 
\begin{align}
\omega_n (6R^3 r_i)^n = \Vol(B_{6R^3 r_i}(x_i))&\geq \sum_1^N \Vol(B_{r_j/5}(x_j))\geq N\omega_n (R^{-6}r_i/5)^n\, ,
\end{align}
which by rearranging gives the estimate $\bar N\leq \bar N(n,R)$.\\

Now we wish to build our decomposition $U_> =\bigcup_1^N U^a_>$, where $N=\bar N+1$ is from the first paragraph.  We shall do this inductively, with the property that at each step of the inductive construction we will have that the sets $U^a_{+}$ will satisfy the desired property. In particular, for every $a$ and $i\in I^a$, if $j\in I^a$ is such that $x_j\in B_{R r_i}(x_i)$, then $r_j<R^{-2}r_i$.\\


Begin by letting each $I^a$ be empty. We are going to sort the points $\cur{x_i}_{i\in I}$ into the sets $I^a$ one at a time.  At each step let $i\in I\setminus \bigcup_{a=1}^N I^a$ be an index such that $r_i=\max_{} r_j$, where the max is taken over all indexes in $I\setminus \bigcup_{a=1}^N I^a$, i.e., over all indexes which haven't been sorted out yet.  Now let us consider the collection of ball centers $\{y_j\}_{j\in J}$ such that $J\subset I$, $x_i\in B_{R r_j}(y_j)$ and $r_i\leq r_j\leq R^{2}r_i$.  Note that, by construction, either $y_j$ has already been sorted out in some $I^a$, or $r_j=r_i$. Now evidently $y_j\in \B{R^3 r_i}{x_i}$ for all $j\in J$ and so by the first paragraph the cardinality of $J$ is at most $N(n,R)$.   In particular, there must be some $I^a$ such that $I^a\cap \{y_j\}=\emptyset$.  Let us assign $i\in I^a$ to this piece of the decomposition.  Clearly the decomposition $\bigcup U^a_>$ still satisfies the inductive hypothesis after the addition of this point, and so this finishes the inductive step of construction.  Since at each stage we have chosen $x_i$ to have the maximum radius, this process will continue indefinitely to give the desired decomposition of $U_>$.


\end{proof}
\vspace{.5 cm}

Now with a decomposition fixed, let us consider for each $1\leq a\leq N$ the measures
\begin{align}
\mu^a \equiv \sum_{x_i\in U^a_>} \omega_k r_i^k \delta_{x_i}\, .
\end{align}

The following is a crucial point in our construction.  It tells us that each ball $B_{10r_i}(x_i)$ either has small $\mu^a$-volume, or the point $x_i$ must have large energy at scale $r_i$.  Precisely:

\begin{lemma}\label{l:covering:energy_density}
Let $D,\eta'>0$ be fixed. There exists $R=R(n,K_N,\Lambda,\eta',D,\epsilon)>0$ such that if we consider the decomposition \eqref{e:covering:U^a_decomposition}, and if
\begin{enumerate}
\item $K_M<\delta(n,K_N,\Lambda,\eta',\epsilon)$ and $\eta\leq \eta(n,K_N,\Lambda,\eta',\epsilon)$, $r_i<10^{-4}$,
\item we have $\mu^a(B_{10r_i}(x_i))\geq 2\omega_k r_i^k$,
\item for all ball centers $y_j\in A_{r_i/10,10r_i}(x_i)\cap U^a$ and for $s = 10^{-2n}D^{-1}r_i$, we have that $\mu^a(B_{s}(y_j))\leq D s^k$,
\end{enumerate}
then we have that $\theta(x_i,\eta' r_i)\geq E-\eta'$.
\end{lemma}

\begin{proof}
Let us begin by choosing $\eta''<<\eta'$, which will be fixed later in the proof so that $\eta''=\eta''(n,K_N,\Lambda,\epsilon)$, and let us also define $\tau\equiv 10^{-3n}D^{-1}$.  Let $\delta(n,K_N,\Lambda,\eta')$ be from Lemma \ref{l:covering:improved_continuity} so that the conclusions hold with $10^{-1}\eta'$, and let $\delta'(n,K_N,\Lambda,\eta'',\epsilon)$ be chosen so that conclusions of Lemma \ref{l:covering:energy_bound} are satisfied with $\eta''$ if $\eta\leq \eta(n,K_N,\Lambda,\eta'',\epsilon)$.  Now throughout we will assume $\eta<\delta$ and $K_M<\min\{\delta,\delta'\}$.  We will also choose $R=R(n,K_N,\Lambda,\eta'',D,\epsilon)>\max\{\tau^{-1},\delta^{-1},\delta'^{-1}\}$ so that Lemma \ref{l:covering:energy_bound} is satisfied with $\eta''$. \\

Since it will be useful later, let us first observe that $r_i\geq R r$.  Indeed, if not then for each ball center $y_j\in A_{r_i/10,10r_i}(x_i)\cap U^a_>$ we would have $r\leq r_j<R^{-2}r_i<r$.  This tells us that there can be no ball centers in $A_{r_i/10,10r_i}(x_i)\cap U^a_>$.  However, by our volume assumption we have that
\begin{align}
\mu^a\big(A_{r_i/10,10r_i}(x_i)\big) = \mu^a\big(B_{10r_i}(x_i)\big)-\mu^a\big( B_{r_i/10}(x_i)\big)\geq 2\omega_k r_i^k-\omega_k r_i^k\geq \omega_k r_i^k\, ,
\end{align}
which contradicts this.  Therefore we must have that $r_i\geq R r$.\\

Now our first real claim is that under the assumptions of the Lemma there exists ball centers $y_0,\ldots,y_{k}\in U^a_>\cap A_{r_i/10,10r_i}(x_i)$ which are $\tau r_i$-independent in the sense of Definition \ref{d:independent_points}.  Indeed, assume this is not the case, then we can find a $k-1$-plane $V^{k-1}$ such that
\begin{align}
\cur{y_i}_{i\in I^a}\cap A_{r_i/10,10r_i}(x_i) \subseteq B_{\tau r_i} \ton{V^{k-1}}\, .
\end{align}
In particular, by covering $B_{\tau r_i} \ton{V}\cap B_{10r_i}(x_i)\cap I^a$ by $C(n)\tau^{1-k}\leq 10^n \tau^{1-k}$ balls of radius $10\tau r_i$ with centers in $I^a$, and using our assumption that $\mu^a(B_{\tau r_i}(y))\leq D \tau^k r_i^k$, we are then able to conclude the estimate
\begin{align}
\mu^a\Big(A_{r_i/10,10r_i}(x_i)\Big) \leq \mu^a\Big(B_{\tau r_i} \ton{V}\cap B_{10r_i}(x_i)\Big) \leq 10^{2n} D \tau r_i^k<\omega_n r_i^k\, .
\end{align}
On the other hand, our volume assumption guarantees that
\begin{align}\label{e:covering:energy_density:1}
\mu^a\big(A_{r_i/10,10r_i}(x_i)\big) = \mu^a\big(B_{10r_i}(x_i)\big)-\mu^a\big( B_{r_i/10}(x_i)\big)\geq 2\omega_k r_i^k-\omega_k r_i^k\geq \omega_k r_i^k\, ,
\end{align}
which leads to a contradiction.  Therefore there must exist $k+1$ ball centers $y_0,\ldots,y_{k}\in A_{r_i/10,10r_i}(x_i)$ which are $\tau r_i$-independent points, as claimed.\\

Let us now remark on the main consequences of the existence of these $k+1$ points.  Note first that for each $y_j$ we have that $\theta_{R^{-1}r_i}(y_j)>E-\eta''$, since by the construction of $U^a_>$ we have that $r_{j}\leq R^{-2}r_i$, and therefore we can apply Lemma \ref{l:covering:energy_bound}.  Thus we have $k+1$ points in $B_{10r_i}(x_i)$ which are $\tau r_i$-independent, and whose energies are $\eta''$-pinched.  To exploit this, let us first apply Lemma \ref{l:covering:improved_continuity} in order to conclude that for each $x\in B_\delta(V)$, where $V$ is the plane spanned by the $k+1$ independent points just determined, we have
\begin{align}
\theta_{\eta' r_i}(x)\geq \theta_{\eta' r_i}(y_j)-|\theta_{\eta' r_i}(x)-\theta_{\eta' r_i}(y_j)|\geq E-\eta''-10^{-1}\eta'>E-\eta'\, .
\end{align}
In particular, if we assume that $x_i$ is such that $\theta_{\eta' r_i}(x_i)< E-\eta'$, then we must have that $r_i^{-1}d(x_j,V)\geq \delta =\delta(n,K_N,\Lambda,\eta')$.  \\

Therefore, let us now assume $\theta_{\eta' r_i}(x_i)< E-\eta'$, and thus to prove the Lemma we wish to find a contradiction.  To accomplish this notice that we have our $k+1$ points $y_0,\ldots,y_k\in B_{10r_i}$ which are $\tau r_i$-independent and for which $|\theta_{20r_i}(y_j)-\theta_{R^{-1}r_i}(y_j)|<\eta''$.  Therefore by applying the cone splitting of Theorem \ref{t:con_splitting} we have for each $\epsilon'>0$ that if $\eta''\leq \eta''(n,K_N,\Lambda,\epsilon')$ then $B_{10r_i}(x_i)$ is $(k,\epsilon')$-symmetric with respect to the $k$-plane $V^k$.  However, since $d(x_i,V)>\delta r_i$, we have by Theorem \ref{t:quant_dim_red} that if $\epsilon'\leq \epsilon'(n,K_N,\Lambda,\epsilon)$ then there exists some $\tau'=\tau'(n,K_N,\Lambda,\epsilon)$ such that $B_{\tau' r_i}(x_i)$ is $(k+1,\epsilon)$-symmetric.  However, we can assume after a further increase that $R=R(n,K_N,\Lambda,D,\epsilon)>4\tau'^{-1}$, and thus we have that $\tau'r_i>4R^{-1} r_i>4r$.  This contradicts that $x_i\in S^k_{\epsilon,r}$, and thus we have contradicted that $\theta_{\eta' r_i}(x_i)< E-\eta'$, which proves the Lemma.
\end{proof}
\vspace{.5cm}

\subsection{\texorpdfstring{Estimating $U_>$ in Lemma \ref{l:covering}}{Estimating U> in Lemma \ref{l:covering}}}\label{ss:covering_U_+}

Recall the set $U_>$ defined in \eqref{e:U>}, which consists of all positive radii balls in our covering.  We now wish to estimate this set in this subsection.  First let us pick $D'=D'(n)\equiv 2^{16n}D(n)$, where $D(n)$ is from Theorem \ref{t:reifenberg_W1p_discrete}.  Then for some $\eta'$ fixed we can choose $R$ as in Lemma \ref{l:covering:energy_density}.  It is then enough to estimate each of the sets $U^a_>$, as there are at most $N=N(n,K_N,\Lambda,\epsilon,\eta,\eta')$ pieces to the decomposition.  Thus we will fix a set $U^a_>$ and focus on estimating the content of this set.  Let us begin by observing that if $r>0$ then we have the lower bound $r_i\geq r$.  Otherwise, let us fix any $r>0$ and restrict ourselves to the collection of balls in $U^a_>$ with $r_i\geq r$.  There is no loss in this so long as the estimates we will prove will be independent of our choice of $r$, and thus by letting $r\to 0$ we will obtain estimates on all of $U^a_>$.\\

Now let us make the precise statement we will prove in this subsection.  Namely, consider any of ball centers $\cur{x_i}_{i\in I^a}$ and any radius $2^{-4}r_i\leq r_\alpha\leq 2$, where $r_\alpha=2^{-\alpha}$.  Then we will show that
\begin{align}\label{e:covering:U_+:1}
\mu^a\big(B_{r_\alpha}(x_i)\big)\leq 2^{5n}D(n) r_\alpha^k\, ,
\end{align}
where $D(n)$ is the constant from the discrete Reifenberg theorem.  Let us observe that once we have proved \eqref{e:covering:U_+:1} then we have finished the proof of the Covering Lemma, as we will then have the estimate
\begin{align}\label{e:covering:U_+:1'}
\sum r_i^k = \mu^a\big(B_2(x_i)\big)\leq C(n)\, . \\ \notag
\end{align}

We prove \eqref{e:covering:U_+:1} inductively on $\alpha$.  To begin notice that for each $x_i$ if $\beta$ is the largest integer such that $2^{-4}r_i\leq r_\beta$, then the statement clearly holds, by the definition of the measure $\mu^a$.  In fact, we can go further than this.  For each $x_i$, let $r'_i\in [r_i/10,10r_i]$ be the largest radius such that for all $10^{-1}r_i\leq s\leq  r'_i$ we have
\begin{align}
\mu^a\big( B_s(x_i)\big)\leq 2\omega_k s^k\, .
\end{align}
In particular, we certainly have the much weaker estimate $\mu^a\big( B_{s}(x_i)\big)\leq 2^{5n}D(n) s^k$, and hence \eqref{e:covering:U_+:1} is also satisfied for all $2^{-4}r_i\leq r_\beta \leq r'_i$.  Notice that we then also have the estimate
\begin{align}\label{e:covering:U_+:2}
\omega_k r_i^k\leq \mu^a\big(B_{r_i/10}(x_i)\big)\leq \mu^a\big( B_{r'_i}(x_i)\big)\leq 2\omega_k\big(r'_i\big)^k\, .
\end{align}

Now let us focus on proving the inductive step of \eqref{e:covering:U_+:1}.  Namely, assume $\alpha$ is such that for all $x_i$ with $2^{-4}r_i\leq r_{\alpha+1}<2$ we have that \eqref{e:covering:U_+:1} holds.  Then we want to prove that the same estimate holds for $r_\alpha$.  Let us begin by seeing that a weak version of \eqref{e:covering:U_+:1} holds.  
Namely, for any index $i\in I^a$ and any radius $r_{\alpha}\leq s\leq 8r_\alpha$, by covering $B_s(x_i)$ by at most $2^{8n}$ balls $\{B_{r_{\alpha+1}}(y_j)\}$ of radius $r_{\alpha+1}$ we have the weak estimate
\begin{align}
\mu^a\big(B_s(x_i)\big) \leq \sum \mu^a\big(B_{r_{\alpha+1}}(y_j)\big)\leq D'(n) s^k\, ,
\end{align}
where of course $D'(n)>>2^{5n}D(n)$.\\

To improve on this, let us fix an $i\in I^a$ and the relative ball center $x_i\in U^a_>$ with $2^{-4}r_i\leq r_\alpha$.  Now let $\{x_j\}_{j\in J}= \cur{x_i}_{i\in I}\cap B_{r_{\alpha}}(x_i)$ be the collection of ball centers in $B_{r_{\alpha}}(x_i)$.  Notice first that if $r'_j>2r_\alpha$ for any of the ball centers $\{x_j\}$, then we can estimate
\begin{align}
\mu^a\big(B_{r_{\alpha}}(x_i)\big)\leq \mu^a\big(B_{2r_{\alpha}}(x_j)\big) \leq 2\omega_k\big(2r_{\alpha}\big)^k\leq 2^{5n}D(n) r_\alpha^n\, ,
\end{align}
so that we may fairly assume $r'_j\leq 2r_\alpha$ for every $j\in J$.  Now for each ball $B_{r'_j}(x_j)$ let us define a new ball $B_{\bar r_j}(y_j)$ which is roughly equivalent, but will have some additional useful properties needed to apply the discrete Reifenberg. 
Namely, for a given ball $B_{r'_j}(x_j)$, let us consider the two options $r'_j<10r_j$ or $r'_j=10r_j$.  If $r'_j<10r_j$, then we let $y_j\equiv x_j$ with $\bar r_j\equiv r_j$.  In this case we must have that $\mu^a(B_{10 r_j}(x_j))>2\omega_n \big(r_j\big)^k$, and thus we can apply Lemma \ref{l:covering:energy_density} in order to conclude that $\theta_{\eta' \bar r_j}(y_j)\geq E-\eta'$.  In the case when $r'_j=10r_j$ is maximal, let $y_j\in B_{r_j}(x_j)\cap S^k_{\epsilon,r}$ with $y_j\neq x_j$ be a point such that $\theta_{\eta r_j}(y_j)=E-\eta$, such a point exists by the definition of $r_j$, and let $\bar r_j\equiv 9r_j$.   
In either case we then have the estimates
\begin{align}\label{e:covering:U_+:3}
&\theta_{\eta' \bar r_j}(y_j)\geq E-\eta'\, ,\notag\\
&\omega_k 10^{-k} \bar r_j^k\leq \mu\big(B_{\bar r_j/10}(y_j)\big)\leq \mu^a\big( B_{\bar r_j}(y_j)\big)\leq 4\omega_k \bar r_i^k\, ,\notag\\
&B_{r_j}(x_j)\subseteq B_{\bar r_j}(y_j)\, .
\end{align}

Now let us consider the covering $U^a_>\cap B_{r_{\alpha}}(x_i)\subseteq \bigcup B_{\bar r_j/10}(y_j)$, and choose from it a Vitali subcovering
\begin{align}
U^a_>\cap B_{r_{\alpha}}(x_i)\subseteq \bigcup B_{\bar r_j}(y_j)\, ,
\end{align}
such that $\{B_{\bar r_j/10}(y_j)\}$ are disjoint, where we are now being lose on notation and referring to $\{y_j\}_{j\in \bar J}$ as the ball centers from this subcovering.  Let us now consider the measure
\begin{align}
\mu' \equiv \sum_{j\in \bar J} \omega_k \Big(\frac{\bar r_j}{10}\Big)^k \delta_{y_j}\, .
\end{align}
That is, we have associated to the disjoint collection $\{B_{\bar r_j/10}(y_j)\}$ the natural measure.  Our goal is to prove that 
\begin{align}\label{e:covering:U_+:4}
\mu'\big( B_{r_\alpha}(x_i)\big)\leq D(n) r_\alpha^k\, .
\end{align}

Let us observe that if we prove \eqref{e:covering:U_+:4} then we are done.  Indeed, using \eqref{e:covering:U_+:3} we can estimate
\begin{align}
\mu^a\big(B_{r_\alpha}(x_i)\big)\leq \sum \mu^a\big(B_{\bar r_j}(y_j)\big)\leq 4\omega_k\sum \bar r_j^k = 4\cdot 10^k\mu'\big(B_{r_\alpha}(x_i)\big)\leq 2^{5n} D(n) r_\alpha^k\, ,
\end{align}
which would finish the proof of \eqref{e:covering:U_+:1} and therefore the Lemma.\\

Thus let us concentrate on proving \eqref{e:covering:U_+:4}.  We will want to apply the discrete Reifenberg in this case to the measure $\mu'$.  Let us begin by proving a weak version of \eqref{e:covering:U_+:4}.  Namely, for any ball center $y_j$ from our subcovering and radius $\bar r_j<s\leq 4r_\alpha$ let us consider the set $\{z_k\}=\cur{y_s}_{s\in \bar J}\cap B_s(y_j)$ of ball centers inside $B_s(y_j)$.  Since the balls $\{B_{\bar r_k/5}(z_k)\}$ are disjoint we have that $\bar r_k\leq 10s$.  Using this, \eqref{e:covering:U_+:2}, and \eqref{e:covering:U_+:3} we can estimate
\begin{align}\label{e:covering:U_+:5}
\mu'\big(B_s(y_j)\big) = \sum_{z_k\in B_s(y_j)} \omega_k 10^{-k}\bar r_k^k\leq C(n)\sum_{z_k\in B_s(y_j)}\mu^a(B_{\bar r_k/8}(z_k))\leq C(n)\mu^a(B_{2s}(y_j))\leq C(n) s^k\, ,
\end{align}
where of course $C(n)>>2^{5n}D(n)$. \\ 

Now let us finish the proof of \eqref{e:covering:U_+:4}.  Thus let us pick a ball center $y_j\in B_{r_\alpha}(x_i)$ and a radius $s<4r_\alpha$.  Note that by the first equation in \eqref{e:covering:U_+:3} and theorem \ref{t:con_splitting} we have that $B_s(y_j)$ can be made arbitrarily $0$-symmetric by choosing $\eta'$ sufficiently small.  If $\mu'(B_s(y))\leq \epsilon_n s^k$ then $D_{\mu'}(y_j,s)\equiv 0$ by definition, and if $s\leq \bar r_j/10$ then $D_{\mu'}(y,s)\equiv 0$, since the support of $\mu'$ in $B_{\bar r_i/10}(y_j)$ contains at most one point and thus is precisely contained in a $k$-dimensional subspace.  In the case when $s>\bar r_i/10$ and $\mu(B_s(y))> \epsilon_n s^k$ then by Theorem \ref{t:best_approximation} we have that
\begin{align}
D_{\mu'}(y_j,s) \leq C(n,K_N,\Lambda,\epsilon)s^{-k}\int_{B_s(y)} W_s(z)\,d\mu'(z)\, .
\end{align}
By applying this to all $r\leq t\leq s$  we can estimate 
\begin{align}
s^{-k}\int_{B_s(x)} D_{\mu'}(y,t)\,d\mu'(y)&\leq Cs^{-k}\int_{B_s(x)}t^{-k}\int_{B_t(y)} W_t(z)\,d\mu'(z)\,d\mu'(y)\notag\\
&= Cs^{-k}t^{-k}\int_{B_s(x)} \mu'(B_t(y)) W_t(y)\,d\mu'(y)\notag\\
&\leq Cs^{-k}\int_{B_s(x)}W_t(y)\,d\mu'(y)\, ,
\end{align}
where we have used our estimate on $\mu'(B_t(y))$ from \eqref{e:covering:U_+:5} in the last line.  Let us now consider the case when $t=r_\beta\leq s\leq 2r_{\alpha}$.  Then we can sum to obtain:
\begin{align}
\sum_{r_\beta\leq s} s^{-k}\int_{B_s(x)} D_{\mu'}(y,r_\beta)\,d\mu'(y) &\leq C\sum_{r'_y\leq r_\beta\leq s} s^{-k}\int_{B_s(x)}W_{r_\beta}(y)\,d\mu'(y)\notag\\
&=C s^{-k}\int_{B_s(x)}\sum_{\bar r_y\leq r_\beta\leq s}W_{r_\beta}(y)\,d\mu'(y)\notag\\
&\leq C\, s^{-k}\int_{B_s(x)}\big|\theta_{4s}(y)-\theta_{\bar r_y}(y)\big| \,d\mu'(y)\leq C(n,K_N,\Lambda,\epsilon)\eta'\, ,
\end{align}
where we are using \eqref{e:covering:U_+:3} in the last line in order to see that $\big|\theta_s(y)-\theta_{\bar r_y}(y)\big|\leq \eta'$.  Now let us choose $\eta'\leq \eta'(n,K_N,\Lambda,\epsilon)$ such that
\begin{align}\label{e:covering:U_r:2c}
\sum_{r_\beta\leq s} s^{-k}\int_{B_s(x)} D_{\mu'}(y,r_\beta)\,d\mu'(y) &\leq \delta^2\, ,
\end{align}
where $\delta$ is chosen from the discrete rectifiable-Reifenberg of Theorem \ref{t:reifenberg_W1p_discrete}.  Since the estimate \eqref{e:covering:U_r:2c} holds for all $B_{s}\subseteq B_{2r_{\alpha}}(x)$, we can therefore apply Theorem \ref{t:reifenberg_W1p_discrete} to conclude the estimate
\begin{align}
\mu'(B_{r_{\alpha}}(x'_i))\leq D(n) r_{\alpha}^k\, .
\end{align}
This finishes the proof of \eqref{e:covering:U_+:3}, and hence the proof of Lemma \ref{l:covering} for the sets $U_r$ and $U_+$. \hfill $\square$
\vspace{1cm}

\subsection{\texorpdfstring{Estimating $U_0$ in Lemma \ref{l:covering}}{Estimating U0 in covering Lemma}}\label{ss:covering_U_0}
We now finish the proof of Lemma \ref{l:covering}.  Let us begin by proving the Minkowski estimates on $U_0$.  Indeed, observe for any $r>0$ that $U_0\subseteq U_r$, and thus we have the estimate
\begin{align}
\Vol(B_r\,\ton{U_0})\leq \Vol(B_r\, \ton{U_r})\leq \omega_n r^n\cdot N\leq C(n,K_N,\Lambda,\epsilon) r^{n-k}\, ,
\end{align}
which proves the Minkowski claim.  In particular, we have as a consequence the $k$-dimensional Hausdorff measure estimate
\begin{align}
\lambda^k(U_0)\leq C(n,K_N,\Lambda,\epsilon)\, .
\end{align}
In fact, let us conclude a slightly stronger estimate, since it will be a convenient technical tool in the remainder of the proof.  If $B_s(x)$ is any ball with $x\in B_1$ and $s<\frac{1}{2}$, then by applying the same proof to the rescaled ball $B_s(x)\to B_1(0)$, we can obtain the Hausdorff measure estimate
\begin{align}\label{eq_U0est}
\lambda^k(U_0\cap B_r(x))\leq C r^k\, .
\end{align}

To finish the construction we need to see that $U_0$ is rectifiable.  We will in fact apply Theorem \ref{t:reifenberg_W1p_holes} in order to conclude this.  To begin with, let $\mu\equiv \lambda^k\big|_{U_0}$ be the $k$-dimensional Hausdorff measure, restricted to $U_0$.  Let $B_s(y)$ be a ball with $y\in B_1$ and $s<\frac{1}{2}$, now we will argue in a manner similar to Section \ref{ss:covering_U_+}.  Thus, if $\mu(B_s(y))\leq \epsilon_n s^k$ then $D_{U_0}(y,s)\equiv 0$, and by theorem \ref{t:con_splitting} and Theorem \ref{t:best_approximation} we have that
\begin{align}
D_\mu(y,s) \leq C(n,K_N,\Lambda,\epsilon)s^{-k}\int_{B_s(y)} W_s(z)\,d\mu(z)\, .
\end{align}
By applying this to all $t\leq s$  we have 
\begin{align}
s^{-k}\int_{B_s(x)} D_\mu(y,t)\,d\mu(y)&\leq Cs^{-k}\int_{B_s(x)}t^{-k}\int_{B_t(y)} W_t(z)\,d\mu(z)\,d\mu(y)\notag\\
&= Cs^{-k}t^{-k}\int_{B_s(x)} \mu(B_t(y)) W_t(y)\,d\mu(y)\notag\\
&\leq Cs^{-k}\int_{B_s(x)}W_t(y)\,d\mu(y)\, ,
\end{align}
where we have used our estimate \eqref{eq_U0est} in the last line.  Let us now consider the case when $t=r_\beta=2^{-\beta}\leq s$.  Then we can sum to obtain:
\begin{align}
\sum_{r_\beta\leq s} s^{-k}\int_{B_s(x)} D_\mu(y,r_\beta)\,d\mu(y) &\leq C\sum_{r_\beta\leq s} s^{-k}\int_{B_s(x)}W_{r_\beta}(y)\,d\mu(y)\notag\\
&=C s^{-k}\int_{B_s(x)}\sum_{r_\beta\leq s}W_{r_\beta}(y)\,d\mu(y)\notag\\
&\leq C\, s^{-k}\int_{B_s(x)}\big|\theta_s(y)-\theta_0(y)\big|\,d\mu(y)\leq C(n,K_N,\Lambda,\epsilon)\eta\, ,
\end{align}
where we have used two points in the last line.  First, we have used our estimate $\mu(B_s(x))\leq Cs^k$.  Second, we have used that by the definition of $U_0$, for each point in the support of $\mu$ we have that $\big|\theta_s(y)-\theta_0(y)\big|\leq \eta$.  Now let us choose $\eta\leq \eta(n,K_N,\Lambda,\epsilon)$ such that we have
\begin{align}\label{e:covering:U_r:2b}
\sum_{r_\beta\leq s} s^{-k}\int_{B_s(x)} D_\mu(y,r_\beta)\,d\mu(y) &\leq \delta^2\, ,
\end{align}
where $\delta$ is chosen from Theorem \ref{t:reifenberg_W1p_holes}.  Thus, by applying Theorem \ref{t:reifenberg_W1p_holes} we see that $U_0$ is rectifiable, which finishes the proof of Lemma \ref{l:covering} in the context of $U_0$.  \hfill $\square$
\vspace{1cm}

\section{Proof of Main Theorems for Stationary Harmonic Maps}\label{s:main_theorem_stationary_proofs}

In this section we prove the main theorems of the paper concerning stationary harmonic maps.  With the tools of Sections \ref{s:bi-Lipschitz_reifenberg}, \ref{s:best_approx}, and \ref{s:covering_main} developed, we will at this stage mainly be applying the covering of Lemma \ref{l:covering} iteratively to arrive at the estimates.  When this is done carefully, we can combine the covering lemma with the cone splitting in order to check that for $k$-a.e. $x\in S^k_\epsilon$ there exists a unique $k$-dimensional subspace $V^k\subseteq T_xM$ such that every tangent map of $f$ at $x$ is $k$-symmetric with respect to $V$.\\

For the proofs of the Theorems' of this section let us make the following remark.  For any $\delta>0$ we can cover $B_1(p)$ by a collection of balls
\begin{align}\label{e:main_theorem_stationary_proofs:1}
B_1(p)\subseteq \bigcup_1^N B_{K^{-1}_M\delta}(p_i)\, ,
\end{align}
where $N\leq C(n,K_M)\delta^{-n}$.  Thus if $\delta=\delta(n,K_M,K_N,\Lambda,\epsilon)$ and we can analyze each such ball, then this gives us estimates on all of $B_1(p)$.  In particular, by rescaling $B_{K^{-1}_M\delta}(p_i)\to B_1(p_i)$, we see that we can assume in our analysis that $K_M<\delta$ without any loss of generality.  We shall do this throughout this section.  \\

\subsection{Proof of Theorem \ref{t:main_quant_strat_stationary}}\label{ss:proof_thm_main_quant_strat_stationary}

Let $f:B_2(p)\subseteq M\to N$ be a stationary harmonic map satisfying \eqref{e:manifold_bounds} with $\fint_{B_2(p)}|\nabla f|^2\leq \Lambda$.  With $\epsilon,r>0$ fixed, let us choose $\eta(n,K_N,\Lambda, \epsilon)>0$ and $\delta(n,K_N,\Lambda,\epsilon)>0$ as in Lemma \ref{l:covering}.  By the remarks around \eqref{e:main_theorem_stationary_proofs:1} we see that we can assume that $K_M<\delta$, which we will do for the remainder of the proof.\\ 

Now let us begin by first considering an arbitrary ball $B_s(x)$ with $x\in B_1(p)$ and $r<s\leq 1$, potentially quite small.  We will use Lemma \ref{l:covering} in order to build a special covering of $S^k_{\epsilon,r}\cap B_s(x)$.  Let us define
\begin{align}
E_{x,s}\equiv \sup_{y\in B_s(x)\cap S^k_{\epsilon,r}}\theta_s(y)\, ,
\end{align}
and thus if we apply Lemma \ref{l:covering} with $\eta(n,K_N,\Lambda,\epsilon)$ fixed to $B_s(x)$, then we can build a covering 
\begin{align}\label{e:proof_main_quant_strat:1}
S^k_{\epsilon,r}\cap B_s(x)\subseteq U_{r}\cup U_{+} = \bigcup B_{r}(x^r_{i})\cup \bigcup B_{r_{i}}(x_{i})\, ,
\end{align}
with $r_i>r$.  Let us recall that this covering satisfies the following:
\begin{enumerate}
\item[(a)] $r^{k-n}\Vol(B_r\, U'_r)+\omega_k\sum r_i^k\leq C(n)\, s^k$.
\item[(b)] $\sup_{y\in B_{r_i}(x_i)}\theta_{r_i}(y)\leq E_{x,s}-\eta$.
\end{enumerate}

Now that we have built our required covering on an arbitrary ball $B_s(x)$, let us use this iteratively to build our final covering of $S^k_{\epsilon,r}(f)$.  First, let us apply it to $B_1(p)$ is order to construct a covering
\begin{align}
S^{k}_{\epsilon,r}(f)\subseteq U^1_{r}\cup U^1_+ = \bigcup B_{r}\ton{x^{r,1}_i}\bigcup B_{r^1_i}\ton{x^1_i}\, ,
\end{align}
such that
\begin{align}
r^{k-n}\Vol\ton{B_r\ton{U^1_{r}}} + \omega_k\sum \ton{r^1_{i}}^{k} \leq C(n,K_N,\Lambda,\epsilon)\, ,
\end{align}
and with
\begin{align}
\sup_{y\in B_{r^1_i}(x_i^1)\cap S^k_{\epsilon,r}}\theta_{r^1_i}(y)\leq \Lambda-\eta\, .\\ \notag
\end{align}

Now let us tackle the following claim, which is our main iterative step in the proof:\\

{\bf Claim}:  For each $\ell$ there exists a constant $C_\ell(\ell,n,K_M,K_N,\Lambda,\epsilon)$ and a covering 
\begin{align}
 S^{k}_{\epsilon,r}(f)\subseteq U^\ell_{r}\cup U^\ell_+ = \bigcup \B{r}{x^{r,\ell}_i} \bigcup \B{r^\ell_i}{x^\ell_i}\, ,
\end{align}
with $r_i^\ell>r$, such that the following two properties hold:
\begin{align}\label{e:proof_main_quant_strat:2}
&r^{k-n}\Vol\ton{\B r {U^\ell_{r}}} + \omega_k\sum \big(r^\ell_{i}\big)^{k} \leq C_\ell(\ell,n,K_N,\Lambda,\epsilon)\, ,\notag\\
&\sup_{y\in B_{r^\ell_i}(x_i^\ell)\cap S^k_{\epsilon,r}}\theta_{r^\ell_i}(y)\leq \Lambda-\ell\cdot\eta\, .
\end{align}
\vspace{.5 cm}

To prove the claim let us first observe that we have shown this holds for $\ell=1$.  Thus let us assume we have proved the claim for some $\ell$, and determine from this how to build the covering for $\ell+1$ with some constant $C_{\ell+1}(\ell+1,n,K_M,K_N,\Lambda,\epsilon)$, which we will estimate explicitly.  \\

Thus with our covering determined at stage $\ell$, let us apply the covering of \eqref{e:proof_main_quant_strat:1} to each ball $\cur{\B{r^{\ell}_i}{x^\ell_i}}$ in order to obtain a covering
\begin{align}
S^k_{\epsilon,r}\cap B_{r^\ell_i}(x_i^\ell)\subseteq U_{i,r}\cup U_{i,+} = \bigcup_j \B{r}{x^r_{i,j}}  \bigcup_j \B{r_{i,j}}{x_{i,j}}\, ,
\end{align}
such that
\begin{align}\label{e:proof_main_quant_strat:3}
&r^{k-n}\Vol\ton{\B r {U_{i,r}}} + \omega_k\sum_j (r_{i,j})^{k} \leq C(n,K_N,\Lambda,\epsilon)\big(r^\ell_i\big)^{k}\, ,\notag\\
&\sup_{y\in B_{r_{i,j}}(x_{i,j})\cap S^k_{\epsilon,r}}\theta_{r_{i,j}}(y)\leq \Lambda-(\ell+1)\eta\, .
\end{align}
Let us consider the sets
\begin{align}
&U^{\ell+1}_r \equiv U^\ell_r \bigcup_i U_{i,r}\, ,\notag\\
&U^{\ell+1}_+ \equiv \bigcup_{i,j} B_{r_{i,j}}(x_{i,j})\, .
\end{align}
Notice that the second property of \eqref{e:proof_main_quant_strat:2} holds for $\ell+1$ by the construction, hence we are left analyzing the volume estimate of the first property.  Indeed, for this we combine our inductive hypothesis \eqref{e:proof_main_quant_strat:2} for $U^\ell$ and \eqref{e:proof_main_quant_strat:3} in order to estimate
\begin{align}
r^{k-n}\Vol\ton{\B r {U^{\ell+1}_{r}}} + \omega_k\sum_{i,j} (r_{i,j})^{k} &\leq r^{k-n}\Vol\ton{\B r {U^\ell_{r}}}+\sum_i\Big(r^{k-n}\Vol\ton{\B r {U_{i,r}}} + \omega_k\sum_j (r_{i,j})^{k} \Big)\notag\\
&\leq C_\ell + C\sum_i (r^\ell_i)^k\notag\\
&\leq C(n,K_M,K_N,\Lambda,\epsilon)\cdot C_\ell(\ell,n,K_M,K_N,\Lambda,\epsilon)\notag\\
&\equiv C_{\ell+1}\, .
\end{align}

Hence, we have proved that if the claim holds for some $\ell$ then the claim holds for $\ell+1$.  Since we have already shown the claim holds for $\ell=1$, we have therefore proved the claim for all $\ell$. \\

Now we can finish the proof.  Indeed, let us take $\ell = \lceil\eta^{-1}\Lambda\rceil = \ell(\eta,\Lambda)$.  Then if we apply the Claim to such an $\ell$, we must have by the second property of \eqref{e:proof_main_quant_strat:2} that
\begin{align}
U^\ell_+\equiv \emptyset\, ,
\end{align}
and therefore we have a covering
\begin{align}
S^k_{\epsilon,r}\subseteq U^\ell_r = \bigcup_i B_{r}(x_i)\, .
\end{align}
But in this case we have by \eqref{e:proof_main_quant_strat:2} that
\begin{align}
\Vol\ton{\B r {S^k_{\epsilon,r}(f)}}\leq \Vol\ton{ \B r { U^\ell_r}} \leq C(n,K_N,\Lambda,\epsilon) r^{n-k}\, ,
\end{align}
which proves the Theorem.  \hfill $\square$
\vspace{1cm}

\subsection{Proof of Theorem \ref{t:main_eps_stationary}}\label{ss:proof_thm_main_eps_stationary}

There are several pieces to Theorem \ref{t:main_eps_stationary}.  To begin with, the volume estimate follows easily now that Theorem \ref{t:main_quant_strat_stationary} has been proved.  That is, for each $r>0$ we have that
\begin{align}
S^k_\epsilon(f)\subseteq S^k_{\epsilon,r}(f)\, ,
\end{align}
and therefore we have the volume estimate
\begin{align}
\Vol\ton{\B r {S^k_{\epsilon}(f)}}\leq \Vol\ton{ \B r {S^k_{\epsilon,r}(f)}} \leq C(n,K_M,K_N,\Lambda,\epsilon) r^{n-k}\, .
\end{align}
In particular, this implies the much weaker Hausdorff measure estimate
\begin{align}
\lambda^k(S^k_{\epsilon}(f)) \leq C(n,K_M,K_N,\Lambda,\epsilon)\, ,
\end{align}
which proves the first part of the Theorem.\\

Let us now focus on the rectifiability of $S^k_\epsilon$.  We consider the following claim, which is the $r=0$ version of the main Claim of Theorem \ref{t:main_quant_strat_stationary}.  We will be applying Lemma \ref{l:covering}, which requires $K_M<\delta$.  As in the proof of Theorem \ref{t:main_quant_strat_stationary} we can just assume this without any loss, as we can cover $B_1(p)$ by a controlled number of balls of radius $K^{-1/2}_M \delta$, so that after rescaling we can analyze each of these balls with the desired curvature assumption. Thus let us consider the following:\\

{\bf Claim:}  If $K_M<\delta$, then for each $\ell$ there exists a covering $S^{k}_{\epsilon}(f)\subseteq U^\ell_{0}\cup U^\ell_+ =U^\ell_0 \bigcup B_{r^\ell_i}(x^\ell_i)$ such that
\begin{enumerate}
\item  $\lambda^k(U^\ell_0) + \omega_k\sum \big(r^\ell_{i}\big)^{k} \leq C_\ell(\ell,n,K_N,\Lambda,\epsilon)$.
\item $U^\ell_0$ is $k$-rectifiable.
\item $\sup_{y\in B_{r^\ell_i}(x_i^\ell)\cap S^k_{\epsilon}}\theta_{r^\ell_i}(y)\leq \Lambda-\ell\cdot\eta$
\end{enumerate}
\vspace{.5 cm}

The proof of the Claim follows essentially the same steps as those for the main Claim of Theorem \ref{t:main_quant_strat_stationary}.  For base step $\ell=0$, we consider the decomposition $S^k_\epsilon \subseteq U^0_0\cup U^0_+$ where $U^0_0=\emptyset$ and $U^0_+=B_1(p)$.  

Now let us assume we have proved the claim for some $\ell$, then we wish to prove the claim for $\ell+1$.  Thus, let us consider the set $U^\ell_+$ from the previous covering step given by
\begin{align}
U^\ell_+ = \bigcup B_{r^\ell_i}(x^\ell_i)\, .
\end{align}
Now let us apply Lemma \ref{l:covering} to each of the balls $B_{r^\ell_i}(x^\ell_i)$ in order to write
\begin{align}
S^k_\epsilon\cap B_{r^\ell_i}(x^\ell_i)\subseteq U_{i,0}\cup U_{i,+} = U_{i,0}\cup \bigcup_{j} B_{r_{i,j}}(x_{i,j})\, ,
\end{align}
with the following properties:
\begin{enumerate}
\item[(a)] $\lambda^k(U_{i,0})+\omega_k\sum_j r_{i,j}^k\leq C(n,K_N,\Lambda,\epsilon,p) (r^\ell_i)^k$,
\item[(b)] $\sup_{y\in B_{r_{i,j}}(x_{i,j})\cap S^k_{\epsilon}}\theta_{r_{i,j}}(y)\leq \Lambda-(\ell+1)\eta$,
\item[(c)] $U_{i,0}$ is $k$-rectifiable.
\end{enumerate}

Now let us define the sets
\begin{align}
&U^{\ell+1}_0 = \bigcup U_{i,0}\cup U^{\ell}_0\, ,\notag\\
&U^{\ell+1}_+ = \bigcup_{i,j} B_{r_{i,j}}(x_{i,j})\, .
\end{align}

Conditions $(2)$ and $(3)$ from the Claim are clearly satisfied.  We need only check condition $(1)$.  Using $(a)$ and the inductive hypothesis we can estimate that

\begin{align}
\lambda^k(U^{\ell+1}_0) + \omega_k\sum_{i,j} \big(r_{i,j}\big)^{k} &\leq \lambda^k(U^\ell_0)+ \sum_i\Big( \lambda^k(U_{i,0})+\omega_k\sum_j \big(r_{i,j}\big)^k\Big)\, ,\notag\\
&\leq C_\ell+C(n,K_N,\Lambda,\epsilon)\sum_i\big(r^\ell_i\big)^k\notag\\
&\leq C(n,K_N,\Lambda,\epsilon)\cdot C_\ell\notag\\
&\equiv C_{\ell+1}\, .
\end{align}
Thus, we have proved the inductive part of the claim, and thus the claim itself.
\vspace{1cm}

Let us now finish the proof that $S^k_\epsilon(f)$ is rectifiable.  So let us take $\ell = \lceil\eta^{-1}\Lambda\rceil = \ell(\eta,\Lambda)$.  Then if we apply the above Claim to $\ell$, then by the third property of the Claim we must have that
\begin{align}
U^\ell_+\equiv \emptyset\, ,
\end{align}
and therefore we have the covering
\begin{align}
S^k_{\epsilon}\subseteq U^\ell_0\, ,
\end{align}
where $U^\ell_0$ is $k$-rectifiable with the volume estimate $\lambda^k(U^\ell_0)\leq C$, which proves that $S^k_\epsilon$ is itself rectifiable.\\

Finally, we prove that for $k$ a.e. $x\in S^k_\epsilon$ there exists a $k$-dimensional subspace $V_x\subseteq T_xM$ such that {\it every} tangent map at $x$ is $k$-symmetric with respect to $V_x$.  To see this we proceed as follows.  For each $\eta>0$ let us consider the finite decomposition
\begin{align}
S^k_\epsilon = \bigcup_{\alpha=0}^{\lceil \eta^{-1}\Lambda \rceil} W^{k,\alpha}_{\epsilon,\eta}\, ,
\end{align}
where by definition we have
\begin{align}
W^{k,\alpha}_{\epsilon,\eta}\equiv \big\{x\in S^k_\epsilon: \theta_0(x)\in \big[\alpha\eta,(\alpha+1)\eta\big)\big\}\, .
\end{align}
Note then that each $W^{k,\alpha}_{\epsilon,\eta}$ is $k$-rectifiable, and thus there exists a full measure subset $\tilde W^{k,\alpha}_{\epsilon,\eta}\subseteq W^{k,\alpha}_{\epsilon,\eta}$ such that for each $x\in\tilde W^{k,\alpha}_{\epsilon,\eta}$ the tangent cone of $W^{k,\alpha}_{\epsilon,\eta}$ exists and is a subspace $V_x\subseteq T_xM$.

Now let us consider such an $x\in \tilde W^{k,\alpha}_{\epsilon,\eta}$, and let $V^k_x$ be the tangent cone of $W^{k,\alpha}_{\epsilon,\eta}$ at $x$.  For all $r<<1$ sufficiently small we of course have $|\theta_r(x)-\theta_0(x)|<\eta$.  Thus, by the monotonicity and continuity of $\theta$ we have for all $r<<1$ sufficiently small and all $y\in W^{k,\alpha}_{\epsilon,\eta}\cap B_r(x)$ that $|\theta_r(y)-\theta_0(y)|<2\eta$.  In particular, by Theorem \ref{t:quantitative_0_symmetry} we have for each $y\in W^{k,\alpha}_{\epsilon,\eta}\cap B_r(x)$ that $B_r(y)$ is $(0,\delta_\eta)$-symmetric, with $\delta_\eta\to 0$ as $\eta\to 0$.  Now let us recall the cone splitting of Theorem \ref{t:con_splitting}.  Since the tangent cone at $x$ is $V^k_x$, for all $r$ sufficiently small we can find $k+1$ points $x_0,\ldots,x_k\in B_r(x)\cap W^{k,\alpha}_{\epsilon,\eta}$ which are $10^{-1}r$-independent, see Definition \ref{d:independent_points}, and for which $B_{2r}(x_j)$ are $(0,\delta_\eta)$-symmetric.  Thus, by the cone splitting of Theorem \ref{t:con_splitting} we have that $B_r(x)$ is $(k,\delta_\eta)$-symmetric with respect to $V^k_x$ for all $r$ sufficiently small, where $\delta_\eta\to 0$ as $\eta\to 0$.  In particular, every tangent map at $x$ is $(k,\delta_\eta)$-symmetric with respect to $V_x$, where $\delta_\eta\to 0$ as $\eta\to 0$.\\

Now let us consider the sets
\begin{align}
\tilde W^k_{\epsilon,\eta} \equiv \bigcup_\alpha \tilde W^{k,\alpha}_{\epsilon,\eta}\, .
\end{align}
So $\tilde W^k_{\epsilon,\eta}\subseteq S^k_\epsilon$ is a subset of full $k$-dimensional measure, and for every point $x\in \tilde W^k_{\epsilon,\eta}$ we have seen that {\it every} tangent map of is $(k,\delta_\eta)$-symmetric with respect to some $V_x\subseteq T_xM$, where $\delta_\eta\to 0$ as $\eta\to 0$.  Finally let us define the set
\begin{align}
\tilde S^k_{\epsilon} \equiv \bigcap_j \tilde W^k_{\epsilon,j^{-1}}\, . 
\end{align}
This is a countable intersection of full measure sets, and thus $\tilde S^k_{\epsilon}\subseteq S^k_{\epsilon}$ is a full measure subset.  Further, we have for each $x\in \tilde S^k_{\epsilon}$ that {\it every} tangent map must be $(k,\delta)$-symmetric with respect to some $V_x$, for all $\delta>0$.  In particular, every tangent map at $x$ must be $(k,0)=k$-symmetric with respect to some $V_x$.  This finishes the proof of the Theorem.  \hfill $\square$
\vspace{1cm}

\subsection{Proof of Theorem \ref{t:main_stationary}}\label{ss:proof_thm_main_stationary}

Let us begin by observing the equality
\begin{align}\label{e:proof_thm_main_stationary:1}
S^k(f) = \bigcup_{\epsilon>0} S^k_{\epsilon}(f) = \bigcup_{\beta\in \dN} S^k_{2^{-\beta}}(f)\, .
\end{align}
Indeed, if $x\in S^k_\epsilon(f)$, then no tangent map at $x$ can be $(k+1,\epsilon/2)$-symmetric, and in particular $k+1$-symmetric, and thus $x\in S^k(f)$.  This shows that $S^k_\epsilon(f)\subseteq S^k(f)$.  On the other hand, if $x\in S^k(f)$ then we claim there is some $\epsilon>0$ for which $x\in S^k_\epsilon(f)$.  Indeed, if this is not the case, then there exists $\epsilon_i\to 0$ and $r_i>0$ such that $B_{r_i}(x)$ is $(k+1,\epsilon_i)$-symmetric.  If $r_i\to 0$ then we can pass to a subsequence to find a tangent map which is $k+1$-symmetric, which is a contradiction.  On the other hand, if $r_i>r>0$ then we see that $B_r(x)$ is itself $k+1$-symmetric, and in particular {\it every} tangent map at $x$ is $k+1$-symmetric.  In either case we obtain a contradiction, and thus $x\in S^k_\epsilon(f)$ for some $\epsilon>0$.  Therefore we have proved \eqref{e:proof_thm_main_stationary:1}.\\

As a consequence, $S^k(f)$ is a countable union of $k$-rectifiable sets, and therefore is itself $k$-rectifiable.  On the other hand, Theorem \ref{t:main_eps_stationary} tells us that for each $\beta\in\dN$ there exists a set $\tilde S^k_{2^{-\beta}}(f)\subseteq S^k_{2^{-\beta}}(f)$ of full measure such that
\begin{align}
\tilde S^k_{2^{-\beta}}\subseteq \big\{x: \, \exists\, V^k\subseteq T_xM\text{ s.t. every tangent map at $x$ is $k$-symmetric wrt $V$}\big\}\, .
\end{align}
Hence, let us define 
\begin{align}
\tilde S^k(f)\equiv \bigcup \tilde S^k_{2^{-\beta}}(f)\, .
\end{align}
Then we still have that $\tilde S^k(f)$ has $k$-full measure in $S^k(f)$, and if $x\in \tilde S^k(f)$ then for some $\beta$ we have that $x\in \tilde S^k_{2^{-\beta}}$, which proves that there exists a subspace $V\subseteq T_xM$ such that every tangent map at $x$ is $k$-symmetric with respect to $V$.  We have finished the proof of the theorem.  \hfill $\square$

\vspace{1cm}

\section{Proof of Main Theorems for Minimizing Harmonic Maps}\label{s:main_theorem_minimizing_proofs}

In this section we prove the main theorems of the paper concerning minimizing harmonic maps.  That is, we finish the proofs of Theorem \ref{t:main_min_finite_measure} and Theorem \ref{t:main_min_weak_L3}. In fact, the proofs of these two results are almost identical, though the first relies on Theorem \ref{t:main_eps_stationary} and the latter on Theorem \ref{t:main_quant_strat_stationary}.  However, for completeness sake we will include the details of both.\\

\subsection{Proof of Theorem \ref{t:main_min_finite_measure}}

We wish to understand better the size of the singular set $\Sing(f)$ of a minimizing harmonic map.  Let us recall that the $\epsilon$-regularity of theorem \ref{t:eps_reg} tells us that if $f$ is minimizing, then there exists $\epsilon(n,K_N,\Lambda)>0$ 
with the property that if $x\in B_1(p)$ and $0<r<r(n,K_M,K_N,\Lambda)$ is such that $B_{2r}(x)$ is $(n-2,\epsilon)$-symmetric, then $r_f(x)\geq r$.  In particular, $x$ is a smooth point, and we have for $\epsilon(n,K_M,K_N,\Lambda)>0$ that
\begin{align}
\text{Sing}(f)\cap B_1(p)\subseteq S^{n-3}_\epsilon(f)\, .
\end{align}
Thus by Theorem \ref{t:main_eps_stationary} there exists $C(n,K_M,K_N,\Lambda)>0$ such that for each $0<r<1$ we have
\begin{align}
\Vol\ton{\B r {\text{Sing}(f)}\cap B_1(p)}\leq \Vol\ton{\B r{S^{n-3}_\epsilon(f)}}\leq C r^3\, .
\end{align}
This of course immediately implies, though of course is much stronger than, the Hausdorff measure estimate
\begin{align}
\lambda^{n-3}\big(\text{Sing}(f)\cap B_1(p)\big)\leq C\, ,
\end{align}
which finishes the proof of the Theorem.  \hfill $\square$
\vspace{1cm}

\subsection{Proof of Theorem \ref{t:main_min_weak_L3}}

We begin again by considering the $\epsilon$-regularity of theorem \ref{t:eps_reg}.  This tells us that if $f$ is minimizing, then there exists $\epsilon(n,K_N,\Lambda)>0$ 
with the property that if $x\in B_1(p)$ and $0<r<r(n,K_M,K_N,\Lambda)$ is such that $B_{2r}(x)$ is $(n-2,\epsilon)$-symmetric, then $r_f(x)\geq r$.  In particular, we have for such $\epsilon,r$ that
\begin{align}
\{x\in B_1(p): r_f(x)<r\}\subseteq S^{n-3}_{\epsilon,r}(f)\, .
\end{align}
Thus by Theorem \ref{t:main_eps_stationary} there exists $C(n,K_M,K_N,\Lambda)>0$ such that for each $0<r<1$ we have
\begin{align}
\Vol\big(B_r\{x\in B_1(p): r_f(x)<r\}\big)\leq \Vol\ton{\B r{S^{n-3}_{\epsilon,r}(f)}}\leq C r^3\, ,
\end{align}
which proves the second estimate of \eqref{e:main_min_weak_L3:1}.  To prove the first we observe that $|\nabla f|(x)\leq r_f(x)^{-1}$, and to prove \eqref{e:main_min_weak_L3:2} we use the remark after Definition \ref{d:regularity_scale} to conclude that $|\nabla^2 f|(x)\leq C(n,K_M,K_N)r_f(x)^{-2}$.  This concludes the proof of the Theorem. \hfill $\square$

\section{Sharpness of the results}\label{ss:examples}

In this section we present a few examples which motivate the sharpness of our results.  



\subsection{\texorpdfstring{Sharpness of $L^p$ Estimates for Minimizers}{Sharpness of Lp Estimates for Minimizers}}\label{sss:Lp_sharp_example}

This example is completely standard, we wish to simply point out some of its properties.  Namely, consider the mapping $f:B_2(0^3)\to S^2$ given by the projection
\begin{align}
f(x) = \frac{x}{|x|}\, .
\end{align}
This is a minimizing harmonic map (see \cite{lin_min,corgul}).  It is easy to compute that
\begin{align}
|\nabla f|(x) \approx r^{-1}_f(x) \approx \frac{1}{|x|}\, .
\end{align}
We therefore get that $|\nabla f|, r^{-1}_f(x)$ have uniform estimates in $L^3_{weak}$, however neither belong to $L^3$.  In particular, this shows that the estimates of Theorem \ref{t:main_min_weak_L3} are sharp.\\

\subsection{Rectifiable-Reifenberg Example I}

Let us begin with an easy example, which shows that the rectifiable conclusions of Theorem \ref{t:reifenberg_W1p_holes} is sharp.  That is, one cannot hope for better structural results under the hypothesis.  Indeed, consider any $k$-dimensional subspace $V^k\subseteq \dR^n$, and let $S\subseteq V^k\cap B_2(0^n)$ be an arbitrary measurable subset.  Then clearly $D(x,r)\equiv 0$ for each $x$ and $r>0$, and thus the hypotheses of Theorem \ref{t:reifenberg_W1p_holes} are satisfied, however $S$ clearly need not be better than rectifiable.  In the next example we shall see that $S$ need not even come from a single rectifiable chart, as it does in this example.\\

\subsection{Rectifiable-Reifenberg Example II}\label{sss:reifenberg}

With respect to the conclusions of Theorem \ref{t:reifenberg_W1p_holes} there are two natural questions regarding how sharp they are.  First, is it possible to obtain more structure from the set $S$ than rectifiable?  In particular, in Theorem \ref{t:reifenberg_W1p} there are topological conclusions about the set, is it possible to make such conclusions in the context of Theorem \ref{t:reifenberg_W1p_holes}?  In the last example we saw this is not the case.  Then a second question is to ask whether we can at least find a single rectifiable chart which covers the whole set $S$.  
This example taken from \cite[counterexample 12.4]{davidtoro} shows that the answer to this question is negative as well. \\ 

To build our examples let us first consider a unit circle $S^1\subseteq \dR^3$.  Let $M^2\supset S^1$ be a smooth M\"{o}bius strip around this circle, and let $S_\epsilon\subseteq M^2\cap B_\epsilon(S^1)\equiv M^2_\epsilon$ be an arbitrary $\lambda^2$-measurable subset of the M\"{o}bius strip, contained in a small neighborhood of the $S^1$.  In particular, $Area(S_\epsilon)\leq C\epsilon\to 0$ as $\epsilon\to 0$.  It is not hard, though potentially a little tedious, to check that assumptions of Theorem \ref{t:reifenberg_W1p_holes} hold for $\delta\to 0$ as $\epsilon\to 0$.\\

However, we have learned two points from these example.  First, since $S_\epsilon$ was an arbitrary measurable subset of a two dimensional manifold, we have that it is $2$-rectifiable, however that is the most which may be said of $S_\epsilon$.  That is, structurally speaking we cannot hope to say better than $2$-rectifiable about the set $S_\epsilon$.  More than that, since $S_\epsilon$ is a subset of the M\"{o}bius strip, we see that even though $S_\epsilon$ is rectifiable, we cannot even cover $S_\epsilon$ by a single chart from $B_1(0^2)$, as there are topological obstructions, see \cite{davidtoro} for more on this.\\

\section*{Acknowledgments}
We would like to thank prof. \href{https://en.wikipedia.org/wiki/Camillo_De_Lellis}{Camillo de Lellis} for his very precise comments on this paper and its earlier versions.

\bibliographystyle{aomalpha}
\bibliography{qstrat}

\def\cprime{$'$} \def\cprime{$'$} \def\cprime{$'$}
  \def\cftil#1{\ifmmode\setbox7\hbox{$\accent"5E#1$}\else
  \setbox7\hbox{\accent"5E#1}\penalty 10000\relax\fi\raise 1\ht7
  \hbox{\lower1.15ex\hbox to 1\wd7{\hss\accent"7E\hss}}\penalty 10000
  \hskip-1\wd7\penalty 10000\box7} \def\cprime{$'$}
\providecommand{\bysame}{\leavevmode\hbox to3em{\hrulefill}\thinspace}
\providecommand{\noopsort}[1]{}
\providecommand{\mr}[1]{\href{http://www.ams.org/mathscinet-getitem?mr=#1}{MR~#1}}
\providecommand{\zbl}[1]{\href{http://www.zentralblatt-math.org/zmath/en/search/?q=an:#1}{Zbl~#1}}
\providecommand{\jfm}[1]{\href{http://www.emis.de/cgi-bin/JFM-item?#1}{JFM~#1}}
\providecommand{\arxiv}[1]{\href{http://www.arxiv.org/abs/#1}{arXiv~#1}}
\providecommand{\doi}[1]{\url{http://dx.doi.org/#1}}
\providecommand{\MR}{\relax\ifhmode\unskip\space\fi MR }
\providecommand{\MRhref}[2]{%
  \href{http://www.ams.org/mathscinet-getitem?mr=#1}{#2}
}
\providecommand{\href}[2]{#2}
\begin{thebibliography}{CHN13b}

\bibitem[Alm68]{almgren_exreg}
\bgroup\scshape{}F.~J. Almgren, Jr.\egroup{}, Existence and regularity almost
  everywhere of solutions to elliptic variational problems among surfaces of
  varying topological type and singularity structure,  \emph{Ann. of Math. (2)}
  \textbf{87} (1968), 321--391. \mr{0225243}.  \zbl{0162.24703}.

\bibitem[Alm00]{almgren_big}
\bgroup\scshape{}F.~J. Almgren, Jr.\egroup{}, \emph{Almgren's big regularity
  paper}, \emph{World Scientific Monograph Series in Mathematics} \textbf{1},
  World Scientific Publishing Co., Inc., River Edge, NJ, 2000, $Q$-valued
  functions minimizing Dirichlet's integral and the regularity of
  area-minimizing rectifiable currents up to codimension 2, With a preface by
  Jean E. Taylor and Vladimir Scheffer. \mr{1777737}.  \zbl{0985.49001}.

\bibitem[AFP00]{AmFu}
\bgroup\scshape{}L.~Ambrosio\egroup{}, \bgroup\scshape{}N.~Fusco\egroup{}, and
  \bgroup\scshape{}D.~Pallara\egroup{}, \emph{Functions of bounded variation
  and free discontinuity problems}, \emph{Oxford Mathematical Monographs}, The
  Clarendon Press, Oxford University Press, New York, 2000. \mr{1857292}.
  \zbl{0957.49001}.

\bibitem[AT15]{AzzTol}
\bgroup\scshape{}J.~Azzam\egroup{} and \bgroup\scshape{}X.~Tolsa\egroup{},
  Characterization of {$n$}-rectifiability in terms of {J}ones' square
  function: {P}art {II},  \emph{Geom. Funct. Anal.} \textbf{25} (2015),
  1371--1412. \mr{3426057}.  \zbl{06521333}.  \doi{10.1007/s00039-015-0334-7}.

\bibitem[BMT13]{monti}
\bgroup\scshape{}Z.~M. Balogh\egroup{}, \bgroup\scshape{}R.~Monti\egroup{}, and
  \bgroup\scshape{}J.~T. Tyson\egroup{}, Frequency of {S}obolev and
  quasiconformal dimension distortion,  \emph{J. Math. Pures Appl. (9)}
  \textbf{99} (2013), 125--149. \mr{3007840}.  \zbl{1266.28003}.
  \doi{10.1016/j.matpur.2012.06.005}.

\bibitem[Bet93]{beth}
\bgroup\scshape{}F.~Bethuel\egroup{}, On the singular set of stationary
  harmonic maps,  \emph{Manuscripta Math.} \textbf{78} (1993), 417--443.
  \mr{1208652}.  \zbl{0792.53039}.  \doi{10.1007/BF02599324}.

\bibitem[BL15]{brelamm}
\bgroup\scshape{}C.~Breiner\egroup{} and \bgroup\scshape{}T.~Lamm\egroup{},
  Quantitative stratification and higher regularity for biharmonic maps,
  \emph{Manuscripta Math.} \textbf{148} (2015), 379--398. \mr{3414482}.
  \zbl{1327.53079}.  \doi{10.1007/s00229-015-0750-x}.  Available at
  \url{http://arxiv.org/abs/1410.5640}.

\bibitem[CHN13a]{ChNaHa2}
\bgroup\scshape{}J.~Cheeger\egroup{}, \bgroup\scshape{}R.~Haslhofer\egroup{},
  and \bgroup\scshape{}A.~Naber\egroup{}, Quantitative stratification and the
  regularity of harmonic map flow,  \emph{accepted to Calc. of Var. and PDE}
  (2013). Available at \url{http://arxiv.org/abs/1308.2514}.

\bibitem[CHN13b]{ChNaHa1}
\bysame, Quantitative stratification and the regularity of mean curvature flow,
   \emph{Geom. Funct. Anal.} \textbf{23} (2013), 828--847. \mr{3061773}.
  \zbl{1277.53064}.  \doi{10.1007/s00039-013-0224-9}.  Available at
  \url{http://arxiv.org/abs/1308.2514}.

\bibitem[CN13a]{ChNa1}
\bgroup\scshape{}J.~Cheeger\egroup{} and \bgroup\scshape{}A.~Naber\egroup{},
  Lower bounds on {R}icci curvature and quantitative behavior of singular sets,
   \emph{Invent. Math.} \textbf{191} (2013), 321--339. \mr{3010378}.
  \zbl{1268.53053}.  \doi{10.1007/s00222-012-0394-3}.  Available at
  \url{http://arxiv.org/abs/1103.1819}.

\bibitem[CN13b]{ChNa2}
\bysame, Quantitative stratification and the regularity of harmonic maps and
  minimal currents,  \emph{Comm. Pure Appl. Math.} \textbf{66} (2013),
  965--990. \mr{3043387}.  \zbl{1269.53063}.  \doi{10.1002/cpa.21446}.
  Available at \url{http://arxiv.org/abs/1107.3097}.

\bibitem[CNV15]{ChNaVa}
\bgroup\scshape{}J.~Cheeger\egroup{}, \bgroup\scshape{}A.~Naber\egroup{}, and
  \bgroup\scshape{}D.~Valtorta\egroup{}, Critical sets of elliptic equations,
  \emph{Comm. Pure Appl. Math.} \textbf{68} (2015), 173--209. \mr{3298662}.
  \zbl{06399723}.  \doi{10.1002/cpa.21518}.  Available at
  \url{http://arxiv.org/abs/1207.4236}.

\bibitem[CG89]{corgul}
\bgroup\scshape{}J.-M. Coron\egroup{} and
  \bgroup\scshape{}R.~Gulliver\egroup{}, Minimizing {$p$}-harmonic maps into
  spheres,  \emph{J. Reine Angew. Math.} \textbf{401} (1989), 82--100.
  \mr{1018054}.  \zbl{0677.58021}.  \doi{10.1515/crll.1989.401.82}.

\bibitem[DS93]{david_semmes}
\bgroup\scshape{}G.~David\egroup{} and \bgroup\scshape{}S.~Semmes\egroup{},
  \emph{Analysis of and on uniformly rectifiable sets}, \emph{Mathematical
  Surveys and Monographs} \textbf{38}, American Mathematical Society,
  Providence, RI, 1993. \mr{1251061}.  \zbl{0832.42008}.
  \doi{10.1090/surv/038}.

\bibitem[DT12]{davidtoro}
\bgroup\scshape{}G.~David\egroup{} and \bgroup\scshape{}T.~Toro\egroup{},
  Reifenberg parameterizations for sets with holes,  \emph{Mem. Amer. Math.
  Soc.} \textbf{215} (2012), vi+102. \mr{2907827}.  \zbl{1236.28002}.
  \doi{10.1090/S0065-9266-2011-00629-5}.

\bibitem[EG92]{EG}
\bgroup\scshape{}L.~C. Evans\egroup{} and \bgroup\scshape{}R.~F.
  Gariepy\egroup{}, \emph{Measure theory and fine properties of functions},
  \emph{Studies in Advanced Mathematics}, CRC Press, Boca Raton, FL, 1992.
  \mr{1158660}.  \zbl{0804.28001}.

\bibitem[Fed69]{Fed}
\bgroup\scshape{}H.~Federer\egroup{}, \emph{Geometric measure theory},
  \emph{Die Grundlehren der mathematischen Wissenschaften, Band 153},
  Springer-Verlag New York Inc., New York, 1969. \mr{0257325}.
  \zbl{0176.00801}.

\bibitem[FMS15]{FoMaSpa}
\bgroup\scshape{}M.~Focardi\egroup{}, \bgroup\scshape{}A.~Marchese\egroup{},
  and \bgroup\scshape{}E.~Spadaro\egroup{}, Improved estimate of the singular
  set of {D}ir-minimizing {$Q$}-valued functions via an abstract regularity
  result,  \emph{J. Funct. Anal.} \textbf{268} (2015), 3290--3325.
  \mr{3336726}.  \zbl{06475693}.  \doi{10.1016/j.jfa.2015.02.011}.

\bibitem[HL90]{HL42}
\bgroup\scshape{}R.~Hardt\egroup{} and \bgroup\scshape{}F.-H. Lin\egroup{}, The
  singular set of an energy minimizing map from {$B^4$} to {$S^2$},
  \emph{Manuscripta Math.} \textbf{69} (1990), 275--289. \mr{1078359}.
  \zbl{0713.58006}.  \doi{10.1007/BF02567926}.

\bibitem[Lin87]{lin_min}
\bgroup\scshape{}F.-H. Lin\egroup{}, A remark on the map {$x/\vert x\vert $},
  \emph{C. R. Acad. Sci. Paris S\'er. I Math.} \textbf{305} (1987), 529--531.
  \mr{916327}.  \zbl{0652.58022}.

\bibitem[Lin99]{lin_stat}
\bysame, Gradient estimates and blow-up analysis for stationary harmonic maps,
  \emph{Ann. of Math. (2)} \textbf{149} (1999), 785--829. \mr{1709303}.
  \zbl{0949.58017}.  \doi{10.2307/121073}.  Available at
  \url{http://arxiv.org/abs/math/9905214}.

\bibitem[LW06]{linwang}
\bgroup\scshape{}F.~H. Lin\egroup{} and \bgroup\scshape{}C.~Y. Wang\egroup{},
  Stable stationary harmonic maps to spheres,  \emph{Acta Math. Sin. (Engl.
  Ser.)} \textbf{22} (2006), 319--330. \mr{2214353}.  \zbl{1121.58017}.
  \doi{10.1007/s10114-005-0673-7}.

\bibitem[Mat95]{mattila}
\bgroup\scshape{}P.~Mattila\egroup{}, \emph{Geometry of sets and measures in
  {E}uclidean spaces}, \emph{Cambridge Studies in Advanced Mathematics}
  \textbf{44}, Cambridge University Press, Cambridge, 1995, Fractals and
  rectifiability. \mr{1333890}.  \zbl{0819.28004}.
  \doi{10.1017/CBO9780511623813}.

\bibitem[Mor66]{morrey}
\bgroup\scshape{}C.~B. Morrey, Jr.\egroup{}, \emph{Multiple integrals in the
  calculus of variations}, \emph{Die Grundlehren der mathematischen
  Wissenschaften, Band 130}, Springer-Verlag New York, Inc., New York, 1966.
  \mr{0202511}.  \zbl{0142.38701}.

\bibitem[NV]{NaVa}
\bgroup\scshape{}A.~Naber\egroup{} and \bgroup\scshape{}D.~Valtorta\egroup{},
  Volume estimates on the critical sets of solutions to elliptic pdes,
  \emph{Preprint}. Available at \url{http://arxiv.org/abs/1403.4176}.

\bibitem[Rei60]{reif_orig}
\bgroup\scshape{}E.~R. Reifenberg\egroup{}, Solution of the {P}lateau {P}roblem
  for {$m$}-dimensional surfaces of varying topological type,  \emph{Acta
  Math.} \textbf{104} (1960), 1--92. \mr{0114145}.  \zbl{0099.08503}.
  Available at \url{http://link.springer.com/article/10.1007%2FBF02547186}.

\bibitem[SU82]{ScUh_RegHarm}
\bgroup\scshape{}R.~Schoen\egroup{} and \bgroup\scshape{}K.~Uhlenbeck\egroup{},
  A regularity theory for harmonic maps,  \emph{J. Differential Geom.}
  \textbf{17} (1982), 307--335. \mr{664498}.  \zbl{0521.58021}.  Available at
  \url{http://projecteuclid.org/getRecord?id=euclid.jdg/1214436923}.

\bibitem[Sim]{simon_reif}
\bgroup\scshape{}L.~Simon\egroup{}, Reifenberg's topological disc theorem.
  Available at
  \url{http://www.math.uni-tuebingen.de/ab/analysis/pub/leon/reifenberg/reifenberg.dvi}.

\bibitem[Sim83]{simon_stat}
\bysame, \emph{Lectures on geometric measure theory}, \emph{Proceedings of the
  Centre for Mathematical Analysis, Australian National University} \textbf{3},
  Australian National University Centre for Mathematical Analysis, Canberra,
  1983. \mr{756417}.  \zbl{0546.49019}.

\bibitem[Sim96]{Simon_RegMin}
\bysame, \emph{Theorems on regularity and singularity of energy minimizing
  maps}, \emph{Lectures in Mathematics ETH Z\"urich}, Birkh\"auser Verlag,
  Basel, 1996, Based on lecture notes by Norbert Hungerb{\"u}hler.
  \mr{1399562}.  \zbl{0864.58015}.  \doi{10.1007/978-3-0348-9193-6}.

\bibitem[Tol15]{Tol}
\bgroup\scshape{}X.~Tolsa\egroup{}, Characterization of {$n$}-rectifiability in
  terms of {J}ones' square function: part {I},  \emph{Calc. Var. Partial
  Differential Equations} \textbf{54} (2015), 3643--3665. \mr{3426090}.
  \zbl{06544048}.  \doi{10.1007/s00526-015-0917-z}.

\bibitem[Tor95]{Toro_reif}
\bgroup\scshape{}T.~Toro\egroup{}, Geometric conditions and existence of
  bi-{L}ipschitz parameterizations,  \emph{Duke Math. J.} \textbf{77} (1995),
  193--227. \mr{1317632}.  \zbl{0847.42011}.
  \doi{10.1215/S0012-7094-95-07708-4}.

\bibitem[Zie89]{ziemer}
\bgroup\scshape{}W.~P. Ziemer\egroup{}, \emph{Weakly differentiable functions},
  \emph{Graduate Texts in Mathematics} \textbf{120}, Springer-Verlag, New York,
  1989, Sobolev spaces and functions of bounded variation. \mr{1014685}.
  \zbl{0692.46022}.  Available at
  \url{http://www.ams.org/mathscinet-getitem?mr=1014685}.

\end{thebibliography}
\vspace{.5 cm}

\end{document}